\documentclass[a4paper,10pt]{amsart}

\title[Parabolic geometries on varieties]%
{Holomorphic parabolic geometries on smooth projective varieties}

\makeatletter
\DeclareRobustCommand{\scotsMc}{\scotsMcx{c}}
\DeclareRobustCommand{\scotsMC}{\scotsMcx{\textsc{c}}}
\DeclareRobustCommand{\scotsMcx}[1]{%
  M%
  \raisebox{\dimexpr\fontcharht\font`M-\height}{%
    \check@mathfonts\fontsize{\sf@size}{0}\selectfont
    \kern.3ex\underline{\kern-.3ex #1\kern-.3ex}\kern.3ex
  }%
}
\expandafter\def\expandafter\@uclclist\expandafter{%
  \@uclclist\scotsMc\scotsMC
}
\makeatother

\newcommand{\noop}[1]{}

\author{Benjamin \scotsMc{}Kay}
\address{School of Mathematical Sciences,  University College Cork, Cork, Ireland}
\email{b.mckay@ucc.ie}
\thanks{This research was supported in part by the International Centre for Theoretical Sciences (ICTS) during a visit for participating in the program - Analytic and Algebraic Geometry (Code: ICTS/aag2018/03).
Thanks to Indranil Biswas, Anca Musta\c{t}\u{a} and Andrei Musta\c{t}\u{a} for help with algebraic geometry, and to Sorin Dumitrescu for help with Cartan geometries.
This article/publication is based upon work from COST Action CaLISTA CA21109 supported by COST (European Cooperation in Science and Technology). \url{www.cost.eu}.
}
\date{\today}

\usepackage{lmodern}
\usepackage{cfr-lm}
\usepackage[T2A,T1]{fontenc}
\usepackage[utf8]{inputenx}
\usepackage{xparse}
\usepackage{standalone}
\usepackage{mathtext}
\usepackage{graphicx}
\usepackage{rotating}
\usepackage[english]{babel}
\usepackage[pdftex,pagebackref]{hyperref}
\hypersetup{
  colorlinks   = true,  
  urlcolor     = black, 
  linkcolor    = black, 
  citecolor    = black  
}
\usepackage{mathrsfs}
\usepackage{pdftexcmds}
\usepackage[kerning=true,tracking=true]{microtype}
\usepackage{xspace}
\usepackage{amsfonts}
\usepackage{verbatim}
\usepackage{amssymb}
\usepackage{mathtools}
\mathtoolsset{centercolon}
\usepackage{mathabx}
\usepackage{slashed}
\usepackage{braket}
\usepackage{cool}
\usepackage[draft]{varioref}
\usepackage{xstring}
\usepackage{array}
\usepackage{ragged2e}
\usepackage{longtable}
\usepackage[longtable]{multirow}
\usepackage{booktabs}
\usepackage{enumitem}
\usepackage{bm}
\usepackage{varwidth}
\usepackage{tensor}
\usepackage{scalerel}
\usepackage{lie-hasse}
\tikzset{
	background rectangle/.style={
		shade,
		top color=olive!20,
		bottom color=white,
		draw=olive!15,
		very thick,
		rounded corners}}
\usepackage{rank-2-roots}
\usepackage{pgfornament}
\NewDocumentCommand\drawroots{mm}%
{%
\begin{tikzpicture}[baseline=-.5]
\begin{rootSystem}{#1}
\roots
\parabolic{#2}
\parabolicgrading
\end{rootSystem}
\end{tikzpicture}
}%
\usepackage{tikz-cd}
\usepackage{colortbl}
\newtheorem{theorem}{Theorem}
\newtheorem{corollary}{Corollary}
\newtheorem{lemma}{Lemma}
\newtheorem{proposition}{Proposition}
\theoremstyle{remark}
\newtheorem{conjecture}{Conjecture}

\newtheorem{example}{Example}

\newcounter{remarkCounter}
\usepackage{tcolorbox}
\NewDocumentCommand\pr{m}{\ensuremath{\left(#1\right)}}
\NewDocumentCommand\curly{m}{\ensuremath{\left\{#1\right\}}}
\NewDocumentCommand\of{m}{\ensuremath{\!\pr{#1}}}

\RenewDocumentCommand\C{o}%
{%
	\ensuremath%
	{%
		\IfValueTF{#1}%
		{%
			\mathbb{C}^{#1}%
		}%
		{%
			\mathbb{C}%
		}%
	}%
}%
\NewDocumentCommand\R{o}%
{%
	\ensuremath%
	{%
		\IfValueTF{#1}%
		{%
			\mathbb{R}^{#1}%
		}%
		{%
			\mathbb{R}%
		}%
	}%
}%
\NewDocumentCommand\Z{o}%
{%
	\ensuremath%
	{%
		\IfValueTF{#1}%
		{%
			\mathbb{Z}^{#1}%
		}%
		{%
			\mathbb{Z}%
		}%
	}%
}%
\NewDocumentCommand\Q{m}%
{%
	\ensuremath%
	{%
		\IfValueTF{#1}%
		{%
			\mathbb{Q}^{#1}%
		}%
		{%
			\mathbb{Q}%
		}%
	}%
}%

\NewDocumentCommand\OO{D(){0}O{1}}%
{
\ensuremath{
\mathscr{O}%
\ifnum\pdfstrcmp{#1}{0}=0{}\else{(#1)}\fi
\ifnum\pdfstrcmp{#2}{1}=0{}\else{
	\ifnum\pdfstrcmp{#2}{0}=0{}\else{^{\oplus #2}}\fi
}\fi
}
}
\NewDocumentCommand\rank{m}{\ensuremath{\operatorname{rank}{#1}}}

\NewDocumentCommand\Proj{m}{\ensuremath{\mathbb{P}^{#1}}}

\NewDocumentCommand\Lie{mo}{\ensuremath{\mathfrak{\MakeLowercase{#1}}\IfValueT{#2}{_{#2}}}}

\NewDocumentCommand\SL{m}{\ensuremath{\operatorname{SL}_{#1}}}
\NewDocumentCommand\LieSL{m}{\ensuremath{\mathfrak{sl}_{#1}}}
\NewDocumentCommand\PSL{m}{\ensuremath{\mathbb{P}\operatorname{SL}_{#1}}}
\NewDocumentCommand\PSO{m}{\ensuremath{\mathbb{P}\operatorname{SO}_{#1}}}

\NewDocumentCommand\drop{mm}{\ensuremath{#2^{\ifnum\pdfstrcmp{#1}{i}=0{\hat\imath}\else{\ifnum\pdfstrcmp{#1}{j}=0{\hat\jmath}\else{#1}\fi}\fi}}}
\NewDocumentCommand\Presentation{mo}{\Braket{#1\IfValueT{#2}{|#2}}}
\NewDocumentCommand\nForms{omm}
{%
\IfValueTF{#1}
{\ensuremath{\vb{#1}^{#2}}}
{\ensuremath{\Omega^{#2}_{#3}}}
}%
\NewDocumentCommand\holForms{omm}
{%
\IfValueTF{#1}
{\ensuremath{\vb{#1}^{#2}}}
{\ensuremath{\Omega_{\text{hol}}^{#2}}}
}%
\NewDocumentCommand\Lm{mm}{\ensuremath{\Lambda^{#1}#2}}
\NewDocumentCommand\Lmtop{m}{\ensuremath{\Lambda^{\operatorname{top}}\of{#1}}}

\NewDocumentCommand\cohomology{omm}{\ensuremath{H_{\IfValueT{#1}{\IfStrEq{#1}{db}{\bar\partial}{#1}}}^{#2}\of{#3}}}

\NewDocumentCommand\LieDer{}{\ensuremath{\mathcal L}}
\NewDocumentCommand\hook{}{\ensuremath{\mathbin{ \hbox{\vrule height1.4pt
        width4pt depth-1pt \vrule height4pt width0.4pt depth-1pt}}}}

\NewDocumentCommand\Symp{m}{\ensuremath{\operatorname{Sp}_{#1}}}

\NewDocumentCommand\homotopyGroup{mm}{\ensuremath{\pi_{#1}({#2})}}
\NewDocumentCommand\fundamentalGroup{m}{\ensuremath{\homotopyGroup{1}{#1}}}

\NewDocumentCommand\MakeLie{m}{\expandafter\def\csname Lie#1\endcsname{\Lie{#1}}}
\NewDocumentCommand\lb{mm}{\ensuremath{\left[{#1}{#2}\right]}}
\NewDocumentCommand\Kaehler{}{K\"ahler\xspace}
\NewDocumentCommand\DashedArrow{}{\mathrel{\ThisStyle{\ooalign{$\SavedStyle\rightarrow$\cr%
  \hfil\textcolor{white}{\rule{2\LMpt}{1\LMex}}\kern2\LMpt\hfil}}}}

\NewDocumentCommand\graded{mm}{\ensuremath{\operatorname{gr}_{#1}{#2}}}
\NewDocumentCommand\prodquot{mmm}{\ensuremath{#1 \times^{#3}\! #2}}

\NewDocumentCommand\rationalMap{omm}{\ensuremath{\IfValueTF{#1}{#1 \colon #2\DashedArrow #3}{#2\DashedArrow #3}}}
\NewDocumentCommand\G{}{\mathscr{G}}
\NewDocumentCommand\prol{omo}{#2^{(\IfValueTF{#1}{#1}{1})\IfValueT{#3}{#3}}}

\NewDocumentCommand\omQ{m}{\omega_{#1}}

\newcommand{\omX}{\omQ{-}}

\newcommand{\omN}{\omQ{+}}
\newcommand{\omZ}{\omQ{0}}
\newcommand{\GQ}[1]{G_{#1}}
\newcommand{\GZ}{\GQ{0}}

\newcommand{\GN}{\GQ{+}}
\newcommand{\LieGQ}[1]{{\LieG}_{#1}}
\newcommand{\LieGZ}{\LieGQ{0}}
\newcommand{\LieGN}{\LieGQ{+}}
\newcommand{\LieGX}{\LieGQ{-}}
\NewDocumentCommand\slope{m}{\ensuremath{\mu\IfValueT{#1}{\left(#1\right)}}}

\NewDocumentCommand\op{m}{#1^{\text{op}}}

\makeatletter

\NewDocumentCommand\cb{O{1}m}%
{\ensuremath{%
\ifnum\pdf@strcmp{#1}{0}=\z@%
  {\OO{}_{#2}}%
  \else{
    \ifnum\pdf@strcmp{#1}{1}=\z@%
    {}%
    \else{
        \ifnum\pdf@strcmp{#1}{-}=\z@%
            {-}%
        \else{
            \ifnum\pdf@strcmp{#1}{-1}=\z@%
                {-}%
            \else%
            {#1}%
            \fi
            }%
        \fi
        }%
    \fi
    K_{#2}
    }
\fi
}
}
\makeatother
\NewDocumentCommand\acb{O{1}m}{\newcount\n \n=#1 \multiply\n by -1 \cb[\number\n]{#2}}

\newcommand{\amal}[3]{\ensuremath{{#1} \mathbin{\times^{#2}} \! #3}}

\makeatletter
\newcommand*{\Shaf}[2][.]%
{
\operatorname{\text{\foreignlanguage{russian}{ш}}}\ifnum\pdf@strcmp{#1}{.}=\z@\else_{#1}\fi\ifnum\pdf@strcmp{#2}{}=\z@\else\! #2\fi
}
\newcommand*{\Ii}[2][.]%
{
\operatorname{Ii}\ifnum\pdf@strcmp{#1}{.}=\z@\else_{#1}\fi#2
}
\makeatother

\newcommand{\sumOver}[2]%
{
\sum_{\mathclap{\substack{#1}}}#2
}

\newcommand*{\sesq}[2]{\ensuremath{\Braket{#1|#2}}}
\newcommand*{\isesq}[3]{\ensuremath{\Braket{#1|#2|#3}}}
\makeatletter
\newcommand*{\KillingForm}[2]%
{
\ensuremath{
#1\cdot#2
}
}
\makeatother

\NewDocumentCommand\KillingSquare{m}%
{%
\ensuremath{#1^2}%
}%

\newcommand{\rtsp}[2]{\ensuremath{{#1}_{#2}}}
\newcommand{\rtspMax}[1]{\ensuremath{{#1}_{\text{max}}}}
\newcommand{\rtspMin}[1]{\ensuremath{{#1}_{\text{min}}}}

\def\lst{B,G,H,K,L,P,Q,S,Z}
\makeatletter
\@for\i:=\lst\do{\expandafter\MakeLie \i}
\makeatother

\newcommand*{\Roots}{\Delta}
\newcommand*{\cptRts}{\Delta_0}

\newcommand*{\ncptPosRts}{\Delta_+}
\newcommand*{\ncptNegRts}{\Delta_-}
\newcommand*{\ParaRts}{\Delta_{\ge 0}}

\newcommand*{\dimC}[1]{\ensuremath{\dim_{\C} \! #1}}

\NewDocumentCommand\vb{m}{\ensuremath{\bm{#1}}}

\newcommand*{\HH}[1]{\ensuremath{\check #1}}
\newcommand*{\Ha}[1]{\HH{\alpha}}
\newcommand*{\Hb}[1]{\HH{\beta}}
\newcommand*{\Hc}[1]{\HH{\gamma}}
\newcommand*{\Hd}[1]{\HH{\delta}}
\newcommand*{\He}[1]{\HH{\varepsilon}}
\newcommand*{\XX}[1]{\ensuremath{e_{#1}}}

\NewDocumentCommand\Pheight{m}{\ensuremath{#1 \cdot e}}

\renewcommand*{\aa}{\alpha}
\newcommand*{\bb}{\beta}
\newcommand*{\cc}{\gamma}
\newcommand*{\dd}{\varepsilon}
\newcommand*{\ee}{\sigma}
\newcommand*{\om}[1]{\omega^{#1}}
\newcommand*{\nator}{t}

\makeatletter
\def\@tocline#1#2#3#4#5#6#7{\relax
  \ifnum #1>\c@tocdepth 
  \else
    \par \addpenalty\@secpenalty\addvspace{#2}%
    \begingroup \hyphenpenalty\@M
    \@ifempty{#4}{%
      \@tempdima\csname r@tocindent\number#1\endcsname\relax
    }{%
      \@tempdima#4\relax
    }%
    \parindent\z@ \leftskip#3\relax \advance\leftskip\@tempdima\relax
    #5\leavevmode\hskip-\@tempdima #6\nobreak\relax
    ,~#7\par
    \endgroup
  \fi}
\makeatother

\newcolumntype{C}{>{\columncolor[gray]{.9}}>{$}c<{$}}
\newcolumntype{L}{>{\columncolor[gray]{.9}}>{$}l<{$}}
\newcolumntype{R}{>{\columncolor[gray]{.9}}>{$}r<{$}}

\arrayrulecolor{white}
\NewDocumentEnvironment{longtabl}{m}%
{%
\begin{longtable}{#1}%
}%
{%
\end{longtable}%
}%

\begin{document}
\begin{abstract}   
We classify the holomorphic parabolic geometries on compact complex manifolds of general type.
We accomplish this by bounding the numerical dimension of any smooth projective variety in terms of geometric invariants of the flag variety model of any holomorphic parabolic geometry on that variety.
Along the way, we uncover foliations and fibrations on smooth projective varieties that admit holomorphic parabolic geometries.
\end{abstract}
\maketitle
\begin{center}
\tableofcontents
\end{center}
\section{Introduction}
Recall the \emph{infinite groups of Lie and Cartan} \cite{Cartan1904,Cartan1905,Cartan1909,Singer.Sternberg:1965}; roughly speaking these are the geometric structures which locally have infinite dimensional flexibility.
Their germs were classified by Lie and Cartan.
Consider the old story of these geometric structures as they appear on algebraic varieties: they include holomorphic volume forms (so the varieties are Calabi--Yau manifolds \cite{Joyce:2000,Yau:1977,Yau:1978}), holomorphic symplectic forms (so the varieties are hyperk\"ahler manifolds \cite{HKLR:1987,Joyce:2000}), holomorphic contact structures (so the varieties are predicted by the Lebrun--Salamon conjecture \cite{BW:2022,Lebrun.Salamon:1994}), and holomorphic foliations \cite{Baum/Bott:1972,Lins.Neto.Scardua:2020,Painleve:2022}.
Their study forms a large part of the history of algebraic geometry over the last hundred years.
By contrast, there has been very little study of rigid holomorphic geometric structures on algebraic varieties.

Now consider that other old story: the theory of rigid, real or complex, geometric structures.
In some of the most profound research in differential geometry of the 20\textsuperscript{th} century, \'Elie Cartan developed methods to compute differential invariants of the most frequently encountered geometric structures on manifolds \cite{Cartan:30,Cartan:136,Cartan:136bis,Cartan:174,Cartan:1992}.
Roughly speaking, these geometric structures are the ones which admit a homogeneous ``model'' example, now known as \emph{Cartan geometries}.

In recent years, \v{C}ap and Slov\'ak \cite{Cap/Slovak:2009} and their collaborators have significantly advanced Cartan's methods, making heavy use of representation theory.
Their work applies to the most important class of Cartan geometries (after pseudo-Riemannian geometries): the \emph{parabolic geometries}.
Parabolic geometries are Cartan geometries whose model is a flag variety.
More concretely, they include conformal Riemannian and conformal pseudo-Riemannian geometry, projective connections, Levi nondegenerate CR hypersurface geometry, path geometries, and many others.
Roughly speaking, a flag variety is a very flexible homogeneous space: you can hold a point of a flag variety fixed, and still have freedom to move the flag variety using a subgroup of the automorphism group containing a Borel subgroup.
For example, conformal Riemannian geometry is ``modelled'' on the sphere with its conformal geometry, a maximal Lie group action on the sphere, while projective connections are modelled on real projective space with its usual projective geometry, a maximal Lie group action on projective space.

Global questions about Cartan geometries have advanced slowly, by comparison to Riemannian geometry, because suitable analytic tools are not available.
C.~N.~Yang \cite{Yang:1980} wrote that mathematics and physics are like two leaves with a common stem
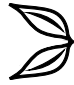
\[
\begin{tikzpicture}
\node{\pgfornament[width = 1cm]{30}};
\fill[
white
] 
(-.55,.5) -- (-.3,.35) -- (0,.1) -- (.2,.03) -- (.385,.0117) 
-- (.38,-.02) -- (.2,-.05) -- (0,-.1) -- (-.3,-.36) -- (-.55,-.5) -- cycle;
\end{tikzpicture}
\]
Rather than looking at all of mathematics and physics, Besse \cite{Besse:1987} p.~11 wrote that we can see this picture as a comparison of Riemannian and Lorentzian geometry.
More precisely, Riemannian and Lorentzian geometry begin with a common stem: the same theory of differential invariants (Levi--Civita connection, torsion, curvature), but they become very different in their global analysis.
Cartan geometry is like a stem with infinitely many leaves, one for each homogeneous model: a common theory of differential invariants, but few common global tools.

Now bring our two old stories together to tell a new story: consider the theory of holomorphic rigid geometric structures on algebraic varieties.
Many years ago \cite{McKay:2004}, I began to look for a unified global theory of holomorphic parabolic geometries on compact complex manifolds, where algebraic geometry provides common global tools.
I hope to find the most beautiful examples of parabolic geometries hiding in plain sight among well known algebraic varieties.
The purpose of this paper is to narrow the search.

Holomorphic parabolic geometries on complex manifolds first arose in works of Poincar\'e \cite{Poincare:1985} in his efforts to uniformize all compact Riemann surfaces, beautifully explained for the modern reader by St.~Gervais \cite{StGervais:2016}.
Poincar\'e recognized that holomorphic linear second order ordinary differential equations are geometric objects on Riemann surfaces: their local solutions have ratios mapping to the projective line, giving a holomorphic parabolic geometry locally modelled on the projective line.
Conversely, every holomorphic parabolic geometry on a Riemann surface arises uniquely from a second order linear ordinary differential equation, i.e. a \(1\)-dimensional linear system.

Study of higher dimensional holomorphic parabolic geometry began almost a century later \cite{Gunning:1978,KobayashiOchiai:1980,Kobayashi/Ochiai:1981,Kobayashi1981,Kobayashi1982} in a burst of activity which immediately faded away.
Every flag variety is the model of a type of holomorphic parabolic geometry, and so has such a geometry on it.
Complex tori carry translation invariant holomorphic parabolic geometries \cite{Biswas.Dumitrescu:2017,Dumitrescu/McKay:2016}.
Every Hermitian symmetric variety carries a natural holomorphic parabolic geometry.
Klingler \cite{Klingler:2001,McKay:2016} proved that its K\"ahler metric is encoded in that holomorphic parabolic geometry.
For example, the following are identical:
\begin{itemize}
\item
smooth ball quotients
\item
ample smooth projective varieties which satisfy the Bogomolov--Miyaoka--Yau inequality \cite{Greb.Kebekus.Taji:2018}
\item
compact complex manifolds of general type which admit holomorphic projective connections.
\end{itemize}
We see the central role of holomorphic parabolic geometries in algebraic geometry.

As pointed out to me by Sorin Dumitrescu, we could see Kodaira's classification of compact complex surfaces \cite{Kodaira:1960,Kodaira:1963,Kodaira:1964,Kodaira:1966,Kodaira:1968a,Kodaira:1968b} as a reflection of the theory of holomorphic Cartan geometries: there is a holomorphic Cartan geometry on a compact complex surface if and only if there is a flat one, and then that surface is a quotient of an open set in the homogeneous model surface, and this occurs (according to an unpublished theorem of Dumitrescu and myself) precisely  when Kodaira defines that surface as that quotient.

Kobayashi, Ochiai and his collaborators \cite{Kobayashi:1984,KobayashiOchiai:1980,Kobayashi1981,Kobayashi/Ochiai:1981,Kobayashi1982} classified holomorphic parabolic geometries on compact complex surfaces.
(On curves, all Cartan geometries are classified \cite{Dumitrescu:2001c,McKay2011b}.)
These results suggested that the holomorphic parabolic geometries on smooth projective varieties might be no more than products of the examples we have mentioned above, so the largest part of the story should be just the theory of locally Hermitian symmetric varieties, itself one of the deepest and most active fields of research in modern mathematics. 
But Jahnke and Radloff \cite{Jahnke/Radloff:2004,Jahnke/Radloff:2015} found very different new examples of holomorphic parabolic geometries arising on complex torus bundles over Shimura varieties.

In this paper, our main goal is to classify the holomorphic parabolic geometries on smooth projective varieties of general type, so that we can see that they are essentially those arising from Hermitian symmetric varieties.
Along the way, we will recall how to classify holomorphic parabolic geometries on compact complex manifolds with \(c_1>0\), or with \(c_1=0\), or even with \(c_1\ge 0\).
We will see that the smooth projective varieties which admit holomorphic parabolic geometries are either those of general type, or have holomorphic fibrations and holomorphic foliations, generalizing those of Jahnke and Radloff, about which we can obtain a great deal of information.
For example (with terminology defined below):
\begin{theorem}
Take an irreducible flag variety \((X,G)\) of dimension two or more.
Suppose that \(M\) is a smooth projective variety bearing a holomorphic parabolic geometry modelled on that flag variety.
If \((X,G)\) is a cominuscule variety (i.e. a compact Hermitian symmetric space) then
\begin{itemize}
\item \(M=X\) with its standard flat holomorphic parabolic geometry or
\item \(M\) is a Hermitian symmetric variety with Hermitian symmetric structure modelled on the noncompact dual to \((X,G)\), hence flat or
\item \(M\) is an abelian variety with translation invariant (perhaps not flat) holomorphic Cartan geometry (explicitly classified, as we will see) or
\item after perhaps replacing \(M\) by a finite covering space, \(M\) is an abelian group scheme with a positive dimensional fiber and a positive dimensional base or
\item
\(M\) is a counterexample to the abundance conjecture.
\end{itemize}
If \((X,G)\) is not a cominuscule variety then 
\begin{itemize}
\item \(M=X\) with its standard flat holomorphic parabolic geometry or
\item \(M\) is an abelian variety with translation invariant, nowhere flat, holomorphic Cartan geometry (explicitly classified, as we will see) or
\item after perhaps replacing \(M\) by a finite covering space, \(M\) is an abelian group scheme, with a positive dimensional fiber and a positive dimensional base, and every holomorphic parabolic geometry on \(M\) is nowhere flat or
\item
\(M\) is a counterexample to the abundance conjecture.
\end{itemize}
\end{theorem}
The abelian group schemes with nowhere flat geometries are the surprise of this paper, although we are not able to prove that they actually exist.

The results are similar to my earlier work with Indranil Biswas on the flat holomorphic parabolic geometries \cite{Biswas.McKay:2024}, but the proofs below are far more difficult, and I am not able to prove anything about developable subvarieties.
It is a long established tradition in the theory of Cartan geometries that the first steps are made purely for flat parabolic geometries, and then followed up by similar results for curved parabolic geometries, and finally for curved Cartan geometries of more general models.

In section~\vref{section:state.of.art} we will summarize almost everything known at the moment about the classification of holomorphic parabolic geometries on smooth projective varieties, including the theorems of this paper.
\subsection{Background}
It seems likely that most readers of this paper are differential geometers, interested in Cartan geometries, and not algebraic geometers, so I provide complete definitions and detailed references for standard results in complex algebraic geometry (which, for me, were difficult to gather together), while I make free use of my book \cite{McKay:2023} as a reference on Cartan geometries.
The reader might also look at \v{C}ap and Slov\'ak \cite{Cap/Slovak:2009} and Sharpe \cite{Sharpe:2002} for more information on Cartan geometries.
\section{The structure of this paper}
In section~\vref{Cartan.geometries}, we give the definition of a holomorphic Cartan geometry, and of lifting and dropping of Cartan geometries, which relate geometries in different dimensions via fibrations.

In section~\vref{section:state.of.art},  we summarize almost everything known at the moment about the classification of holomorphic parabolic geometries on smooth projective varieties, including the theorems of this paper, and define (with references) all of the algebraic geometry terminology and state all of the theorems from algebraic geometry which we will need.
From that point on, we will see that can assume that we are studying a smooth projective variety \(M\), containing no rational curves, with a holomorphic parabolic geometry, and we allow ourselves to assume that the abundance conjecture (defined below) holds for \(M\).
For the benefit of readers more comfortable with locally homogeneous geometry, we explain how to apply our results in the special (but more intuitive) circumstance of a flat holomorphic parabolic geometry, and we use this to derive some known and some new results about quotients of nonclassical flag domains.

In section~\vref{section:flag.varieties}, we define (with references) all of the flag variety terminology and state all of the theorems from the theory of flag varieties which we will need. 
Besides the standard material, we only need a few elementary computational lemmas which are new to this paper.

In section~\vref{section:pseudoeffective}, we recall the concept of pseudoeffective line bundle, a sort of weak nonnegative curvature condition, and explain how it relates to minimality of Cartan geometries, and recall a test for pseudoeffectivity of line bundles due to Campana and Peternell.

In section~\vref{section:brackets}, we recall a theorem of Demailly relating pseudoeffectivity of line bundles and the Frobenius theorem.

In section~\vref{section:pseudoeff.parabolic}, we see how Campana and Peternell's theorem gives rise to pseudoeffective line bundles to which we can apply Demailly's theorem to prove bracket closure of some holomorphic subbundles of the tangent bundle of our variety.
The relevant relationships between line bundles and subbundles require some theory of roots of simple Lie algebras.
At this step, we need that the holomorphic Cartan geometry is parabolic.

In section~\vref{section:structure.equations}, we write out the structure equations of any holomorphic parabolic geometry in detail in terms of a Chevalley basis for the Lie algebra of the model.
This is also new, and we will see that it is useful for computations.

In section~\vref{section:bracket.closed.geometries}, we use these structure equations to unravel, with unfortunately long computations, the curvature equations arising from those subbundles being bracket closed.

In section~\vref{section:boosh}, we see that these curvature equations give rise to a canonical reduction of structure group of any bracket closed parabolic geometry.
This step is motivated by the method of equivalence of Cartan: using some nowhere trivial local information, we can reduce structure group.
The resulting geometric structure, which we call the \emph{boosh} of the parabolic geometry, lives on the same variety \(M\), and is a sort of Cartan geometry (in a certain sense made precise below) with a more rigid model than the original parabolic geometry.
(The model is no longer a homogeneous space, so we need to allow a broader concept of Cartan geometry.)

In section~\vref{section:Chern.boosh}, we compute the Chern class equations on \(M\) arising from the boosh geometry, which are surprisingly stronger than those arising directly from the original parabolic geometry.
We use Beauville's work on the Baum--Bott theorem to make more Chern class equations on \(M\).

In section~\vref{section:bundles.on.tori}, we recall work of Biswas and Upmeier about holomorphic connections on holomorphic principal bundles over tori.
We recall work of Biswas and the author which shows that, if a manifold admits a holomorphic Cartan geometry, then its compact complex submanifolds with torsion canonical bundle are complex tori.
We combine these by showing that the generic fiber of the Iitaka fibration of \(M\) is a complex torus.

In section~\vref{section:Weyl}, we define the Weyl structures of a Cartan geometry, a special kind of reduction of structure group.
We recall that these are identified with choices of holomorphic connection on the canonical bundle.
We apply Weyl structures on open sets to prove that the Iitaka fibration of \(M\) is an abelian group scheme (defined below), a very special type of complex torus fibration.

In section~\vref{section:examples}, we look at the geometries which arose in Cartan's famous \(5\)-variable paper \cite{Cartan:30}, to see that we have made some progress, but we are still left without a complete classification of these types of geometries on smooth projective varieties.

\section{Cartan geometries}\label{Cartan.geometries}
A \emph{complex homogeneous space} \((X,G)\) is a complex manifold \(X\) acted on holomorphically and transitively by a complex Lie group \(G\).
It is a \emph{flag variety} if \(X\) is a simply connected smooth projective variety and \(G\) is semisimple.
We will always pick a point \(x_0\in X\), and often write the stabilizer of \(x_0\in X\) as \(P\subseteq G\).
A \emph{holomorphic Cartan geometry} modelled on \((X,G)\), also called an \((X,G)\)-geometry, on a complex manifold \(M\) is a holomorphic principal \(P\)-bundle \(\G_P\to M\) with a holomorphic connection on \(\G_G:=\amal{\G_P}{G}{P}\) so that the obvious \(P\)-equivariant holomorphic map \(x\in\G_P\mapsto(x,1)G\in\G_G\) is nowhere tangent to the horizontal spaces of that connection.
Denote the connection \(1\)-form as \(\omega\in\nForms{1}{\G_G}\otimes\LieG\).
We always denote the pullback of \(\omega\) to \(\G_P\) also as \(\omega\), and call it the \emph{Cartan connection}, and denote \(\G_P\) by \(\G\).

If \(X\to X'\) is a \(G\)-equivariant holomorphic map of complex homogeneous spaces, say taking points \(x_0\in X\) to \(x_0'\in X'\), with stabilizers \(P:=G^{x_0}\subseteq P':=G^{x'_0}\), and \(\G\to M'\) is an \((X',G)\)-geometry, then there is an associated \((X,G)\)-geometry, the \emph{lift}, defined by taking \(M:=\G/P\), with obvious quotient map \(\G\to M=\G/P\), and the same holomorphic \(1\)-form.
The lift is also called the \emph{correspondence space} \v{C}ap \& Slov{\'a}k \cite{Cap/Slovak:2009} p.~99, Definition~1.5.13. 
Recall that, in my terminology \cite{McKay:2023}, a Cartan geometry \emph{drops} to another one if the first is isomorphic to the lift of the second.
Dropping, not lifting, is the most important Cartan geometry concept in this paper.
A \emph{minimal geometry} is a holomorphic Cartan geometry which does not drop to a holomorphic Cartan geometry on a lower dimensional manifold.
Clearly the classification of holomorphic Cartan geometries follows from the classification of the minimal ones.
\section{The state of the art: geometries on varieties}%
\label{section:state.of.art}
We set the stage by summarizing most of what is known, and what will will prove below, about Cartan geometries on smooth projective varieties.
\subsection{Projectivity}
A \emph{smooth projective variety} is a compact complex manifold which admits a holomorphic embedding into a complex projective space.
A \emph{Moishezon manifold} is a connected compact complex manifold bimeromorphic to a complex projective variety \cite{Ueno:1975} p. 26. 
\begin{theorem}[Biswas \& McKay \cite{Biswas.McKay:2016} p. 3 corollary 2]\label{theorem:Moishezon}
If a Moishezon manifold \(M\) admits a holomorphic Cartan geometry then \(M\) is a connected smooth complex projective variety.
If the holomorphic Cartan geometry on \(M\) drops to a complex manifold \(M'\) then \(M'\) is also a connected smooth complex projective variety.
\end{theorem}
Hence the classification of holomorphic Cartan geometries on Moishezon manifolds follows immediately from the classification of minimal geometries on connected smooth complex projective varieties.
Henceforth we focus on the latter classification exclusively, but we note that all of our results about smooth projective varieties immediately generalize, by this theorem, to Moishezon manifolds.
\subsection{Positive Ricci}
A rational curve in a complex manifold is \emph{tame} if it lies in a compact irreducible component of the Doaudy space of deformations \cite[p. 50, Cor. 6.0.3]{Barlet.Magnusson:2019}.
If a compact complex manifold is dominated by a compact K\"ahler manifold, then every rational curve is tame \cite[p. 189]{Fujiki:1982}.
A compact complex manifold is \emph{tamely rationally connected} if two generic points lie on a connected finite union of tame rational curves.
\begin{example}
Campana \cite{Campana:1992} proved that any connected compact complex manifold with \(c_1>0\) has chains of tame rational curves connecting any two of its points.
Any connected compact complex manifold with a Ricci positive K\"ahler--Einstein metric has \(c_1>0\), so the following theorem classifies the holomorphic parabolic geometries on compact positive Ricci K\"ahler--Einstein manifolds.
\end{example}
\begin{theorem}[Biswas \& McKay \cite{Biswas.McKay:2016} p. 2 theorem 2]
Any tamely rationally connected compact complex manifold \(M\) admits a holomorphic Cartan geometry, say with homogeneous model \((X,G)\), if and only if \(M\) is biholomorphic to \(X\), in which case the geometry is isomorphic to the model geometry and \((X,G)\) is a flag variety.
\end{theorem}
\subsection{Ricci flat}
A complex Lie group \(G\) is \emph{affine} if some finite index subgroup of \(G\) has a morphism of complex Lie groups to a complex linear algebraic group, inducing an isomorphism of Lie algebras. (This is not quite the same as the definition of \emph{affine} of Biswas \& Dumitrescu \cite{Biswas.Dumitrescu:2017}, but I think it is what they meant to write.)
\begin{example}
Any complex linear algebraic group is affine, and hence any complex semisimple Lie group is affine.
\end{example}
\begin{example}
Any complex Lie group which admits a holomorphic representation, injective on its Lie algebra, as unipotent complex linear transformations, is affine; see Milne \cite{Milne:2017} p. 292 corollary 14.38.
\end{example}
\begin{example}
Any complex Lie group with finite fundamental group and finitely many components is affine; see Hilgert \& Neeb \cite{Hilgert.Neeb:2012} p. 602 corollary 16.3.9.
\end{example}
Suppose that \(G\) is affine, say with finite index subgroup \(G_0 \subset G\) having morphism \(G_0 \to \bar{G}_0\) to a complex linear algebraic group, an isomorphism on Lie algebras.
If \(G \to P \to M\) is a holomorphic principal \(G\)-bundle, then on the finite unramified covering space \(\hat{M}:= P/G_0\), we have a holomorphic principal \(G_0\)-bundle \(G_0 \to P \to \hat{M}\), and an associated principal \(\bar{G}_0\)-bundle \(\bar{P}:=\amal{P}{G_0}{\bar{G}_0}\).
Every holomorphic connection on \(P \to M\) pulls back to a holomorphic connection on \(P \to \hat{M}\), and conversely by averaging over the deck transformations.
Every holomorphic connection on \(P \to \hat{M}\) induces a holomorphic connection on \(\bar{P}\), and every holomorphic connection on \(\bar{P} \to \hat{M}\) pulls back to a holomorphic connection on \(P \to \hat{M}\).
We can replace \(G_0\) by a finite index subgroup, so assume that \(\bar{G}_0\) is connected.
So, up to finite covering, the existence of a holomorphic connection, or of a holomorphic flat connection, is the same for the original bundle as for the associated bundle with connected complex linear algebraic structure group.
\begin{theorem}[Biswas \& McKay \cite{Biswas.McKay:2010}]\label{thm:cpt.Kaehler.c.1.0}
A compact connected K\"ahler manifold \(M\) with \(c_1(M)=0\) admits a holomorphic Cartan geometry if and only if \(M\) has a finite unramified covering by a complex torus.
The holomorphic Cartan geometry then pulls back to that covering.
\end{theorem}
\begin{theorem}[Biswas \& Dumitrescu \cite{Biswas.Dumitrescu:2017}]\label{theorem:BD.torus}
If the group \(G\) of a complex homogeneous space \((X,G)\) is affine then every holomorphic \((X,G)\)-geometry on any complex torus is translation invariant.
\end{theorem}
Every flag variety \((X,G)\) has \(G\) complex semisimple, so affine, hence:
\begin{corollary}
If a compact K\"ahler manifold \(M\) with torsion canonical bundle admits a holomorphic parabolic geometry then \(M\) has a finite unramified covering by a complex torus.
Every holomorphic parabolic geometry on any complex torus is translation invariant.
\end{corollary}

\subsection{Semipositive Ricci}
A line bundle \(L\to M\) on a projective variety \(M\) is \emph{numerically effective}, \emph{nef} for short, if \(\int_C c_1(L) \ge 0\) for all complex algebraic curves \(C\) in \(M\).
\begin{theorem}%
\label{theorem:semipos}
Suppose that \(M\) is a compact K\"ahler manifold with nef anticanonical bundle.
After perhaps replacing \(M\) by a finite unramified covering space, every holomorphic Cartan geometry on \(M\) is the lift of a Cartan geometry on a complex torus, and in particular \(M\) is a holomorphic fiber bundle, with flag variety fibers, over a complex torus.
\end{theorem}
\begin{proof}
I previously proved this result for \(c_1\ge 0\) \cite{McKay:2016} p. 8 theorem 2.
But more recently, Matsumura \cite{matsumura2025compactkahlermanifoldsnef} proved that every compact K\"ahler manifold with nef anticanonical bundle is a locally trivial holomorphic fibration with rationally connected fibers and Calabi--Yau base.
Apply theorem~\vref{theorem:drop} and theorem~\vref{thm:cpt.Kaehler.c.1.0}.
\end{proof}

To state our theorems about holomorphic parabolic geometries on compact complex manifolds of general type, we first need some standard definitions from algebraic geometry.

\subsection{Mori's minimal models}
A complex manifold is \emph{uniruled} if a rational curve passes through its generic point.
A projective variety is \emph{minimal} if it has nef canonical bundle.
By Mori's cone theorem \cite{Cascini:2013,Mori:1982} and \cite{Kollar/Mori:1998} p. 22, every smooth projective variety with no rational curves is minimal.
\begin{theorem}[Biswas \& McKay \cite{Biswas.McKay:2016} theorem 2 p. 2]\label{theorem:drop}
On a smooth projective variety \(M\) bearing a holomorphic Cartan geometry, the following are equivalent 
\begin{enumerate}
\item
\(M\) is minimal, i.e. has nef canonical bundle,
\item 
Some holomorphic Cartan geometry on \(M\) is minimal, i.e.~does not drop,
\item 
Every holomorphic Cartan geometry on \(M\) is minimal,
\item
\(M\) contains no rational curves,
\item
\(M\) is not uniruled,
\item
\(M\) is not the total space of a holomorphic fiber bundle, with positive dimensional flag manifold fibers, over a minimal smooth projective variety.
\end{enumerate}
\end{theorem}
Therefore the classification of holomorphic Cartan geometries on smooth projective varieties follows from the classification of minimal geometries on minimal connected smooth projective varieties containing no rational curves.
We mostly focus henceforth on the latter classification.
\subsection{Abundance}
A line bundle on a smooth projective variety is \emph{semiample} if some positive power is spanned by global sections.
The \emph{abundance conjecture} claims that the canonical bundle of any minimal projective variety is semiample; for more information on this conjecture, see Siu \cite{Siu:2009}.
We will occasionally need to assume, in addition to minimality of some smooth projective variety, that the canonical bundle is semiample which (as we will always point out) would follow from the abundance conjecture.
\subsection{Rational maps}
A \emph{modification} \(X\to Y\) of reduced and irreducible compact complex analytic spaces is a surjective morphism which is biholomorphic as a map \(U_X\to U_Y\) on the complements \(U_X:=X-X'\), \(U_Y:=Y-Y'\) of nowhere dense analytic subsets \(X'\subseteq X\), \(Y'\subseteq Y\) \cite{Ueno:1975} p.~14.
A \emph{meromorphic map} \(\rationalMap[\varphi]{X}{Y}\) is a complex analytic subvariety \(Z\subseteq X\times Y\), reduced and irreducible, called the \emph{graph of \(\varphi\)}, so that the projection to \(X\) is a modification.
If the projection to \(Y\) is also a modification, then the set of pairs \((y,x)\in Y\times X\) so that \((x,y)\in Z\) is also a meromorphic map, the \emph{inverse map}, and the meromorphic map is a \emph{bimeromorphism}.
If \(X\) and \(Y\) are algebraic varieties, a meromorphic map is called a \emph{rational map}.
A rational bimeromorphism is a \emph{birational map}.
\subsection{Kodaira dimension}
Take a line bundle \(L\to M\) on a compact irreducible reduced complex space \(M\).
For each integer \(p\ge 1\), we have the vector space \(V:=\cohomology{0}{M,pL}\) of sections of \(pL\).
To each point \(m \in M\), associate the hyperplane in \(V\) consisting of sections vanishing at \(m\).
To each integer \(p\ge 1\) is associated the rational map
\[
\rationalMap{M}{\Proj{}V^*}
\]
taking each point to its hyperplane.
The \emph{Iitaka dimension} of a holomorphic line bundle \(L\) on  is the maximal dimension of the images of these maps, for \(p=1,2,3,\dots\); see \cite{Ueno:1975} p. 50, \cite{Lazarsfeld:2004} p. 122 definition 2.1.3.
If \(0=\cohomology{0}{M,pL}\) for all integers \(p>0\), the Iitaka dimension is defined to be \(-\infty\).

The line bundle \(L \to M\) is \emph{big} if the Iitaka dimension of \(L\) is the dimension of \(M\).
A nef line bundle \(L \to M\) is big just when \(c_1(L)^n>0\) where \(n:=\dimC{M}\) \cite{Demailly:2012} p. 47 definition 6.20, \cite{Lazarsfeld:2004} p. 144 theorem 2.2.16.
The \emph{numerical dimension} of a holomorphic line bundle \(L\) on a complex manifold \(M\) is the largest integer \(k\) for which \(0\ne c_1(L)^k\in\cohomology{2k}{M,\R{}}\), so big means precisely that numerical dimension equals dimension.
The Iitaka dimension is at most the numerical dimension \cite{Demailly:2012} p. 47 proposition 6.21.
A holomorphic line bundle is \emph{abundant} if its Iitaka and numerical dimensions are equal \cite{Lazarsfeld:2004} p. 165 remark 2.3.17.
The \emph{Kodaira dimension} \(\kappa_M\) of a compact irreducible reduced complex space \(M\) is the Iitaka dimension of the canonical bundle of \(M\).
\subsection{Minimal models}
Smooth projective varieties are \emph{birational} if they have isomorphic fields of rational functions.
The \emph{minimal model conjecture} claims that every projective variety with nonnegative Kodaira dimension is birational to a minimal projective variety, i.e. a projective variety with numerically effective canonical bundle \cite{Kollar/Mori:1998}.
As we mentioned previously, by Mori's cone theorem \cite{Cascini:2013,Mori:1982} and \cite{Kollar/Mori:1998} p. 22, every smooth projective variety with no rational curves is minimal.
In particular, every smooth projective variety with nonnegative Kodaira dimension which bears a holomorphic Cartan geometry is minimal.

\subsection{The Iitaka fibration}
An \emph{Iitaka fibration} of a compact irreducible reduced complex space \(M\) is a dominant rational map \(\rationalMap{M}{S}\) with connected fibers, whose generic fiber is smooth and irreducible, so that subvarieties of \(M\) on which the canonical bundle \(\cb{M}\) has zero Iitaka dimension lie in the fibers and, on the very general fiber, \(\cb{M}\) has zero Iitaka dimension \cite{Fujino:2020,Iitaka:1972}, \cite{Lazarsfeld:2004} p.~133 theorem~2.1.33.
Any projective variety of nonnegative Kodaira dimension has a unique Iitaka fibration up to birational isomorphism \cite{Iitaka:1982} p. 302 theorem 10.3 or \cite{Ueno:1975} p. 58 theorem 5.10, which we denote by \(\rationalMap{M}{\Ii{M}}\).
If the canonical bundle is semiample (i.e.~some power of the canonical bundle is spanned by global sections), then there is a unique holomorphic Iitaka fibration, and its codomain \(\Ii{M}\) is the image of the \(\cb[p]{M}\)-maps
\[
M \to \Proj{}\cohomology{0}{M,\cb[p]{M}}^*
\]
for all but finitely many integers \(p>0\) \cite{Lazarsfeld:2004} p.~129, theorem 2.1.27.
\subsection{Abelian group schemes}
Recall that an \emph{abelian variety} is a complex torus which is also a projective variety; the standard reference is Mumford \cite{Mumford:2008}.
For our purposes, an \emph{abelian group scheme} is a surjective holomorphic submersion \(X\xrightarrow{\pi} Y\) of smooth projective varieties, equipped with a holomorphic section \(Y\to X\), so that every fiber \(X_y\subseteq X\) is an abelian variety \cite{Mumford:2008} p.~89, and so that \(X\to Y\) is a smooth morphism of varieties \cite{Mumford:1999} p.~214, hence in particular is a trivial \(C^\infty\) torus fiber bundle 
\[
\begin{tikzcd}
X \arrow[rd,"\pi"'] \arrow[rr,<->,"C^\infty"] &   & Y\times T \arrow[ld,"\operatorname{proj}_Y"] \\                        & Y &
\end{tikzcd}
\]
where \(T:=\R^{2n}/\Z^{2n}\) as a \(C^\infty\) manifold and \(\operatorname{proj}_Y(y,t)=y\).
Careful: the fibers \(X_y\) may not be biholomorphic as complex tori; their complex structure can vary as a function of \(y\in Y\).

\subsection{Parabolic geometries as abelian group schemes}
We have promised the reader that we will find holomorphic fibrations, and some information about their geometry.
The first main theorem of this paper:
\begin{theorem}\label{thm:Iitaka.group.scheme}
Take a smooth projective variety \(M\) with a minimal holomorphic parabolic geometry \(\G\to M\).
Suppose that the model is \((X,G)\).
Assume that the canonical bundle of \(M\) is semiample (which would follow from the abundance conjecture).
Then, after perhaps replacing \(M\) by a finite covering space of \(M\), the Iitaka fibration of \(M\) is an abelian group scheme, so every fiber is an abelian variety (possibly of dimension zero, i.e. just a point).
The fibration has a global holomorphic section.
The base of the Iitaka fibration has ample canonical bundle (or its just a point).
The restriction of the Cartan geometry bundle \(\G\) to any fiber of the Iitaka fibration admits a flat holomorphic connection.
The tangent bundle of the base of that Iitaka fibration satisfies all equations on Chern classes that are satisfied by the Chern classes of the tangent bundle of the associated cominuscule variety \((\breve X,\breve G)\) of \((X,G)\).
\end{theorem}
\begin{example}
The smooth projective varieties admitting holomorphic projective connections are classified \cite{Jahnke/Radloff:2004,jahnke2009holomorphicnormalprojectiveconnections,Jahnke/Radloff:2015,KobayashiOchiai:1980,Kobayashi1981}.
The holomorphic projective connections these varieties admit are also classified \cite{Biswas/Dumitrescu:2023,Dumitrescu:2010,Klingler:1998}.
Every smooth projective variety admitting a holomorphic projective connection is precisely one of:
\begin{itemize}
\item
the model projective space \(\mathbb{P}^n\), 
\item 
an abelian variety, 
\item
a ball quotients, i.e. a quotient of the unit ball in \(\C^n\) by a cocompact group of biholomorphisms,
\item
an abelian group scheme of a certain type (which is complicated to make precise).
\end{itemize}
The abelian group schemes are explicitly constructed \cite{Jahnke/Radloff:2004,jahnke2009holomorphicnormalprojectiveconnections,Jahnke/Radloff:2015}, occurring only in complex dimension three or more.
\end{example}

\subsection{Associated cominuscule varieties}
A flag variety \(X=G/P\) is \emph{cominuscule} if its tangent space \(T_{x_0} X=\LieG/\LieP\) is a direct sum of irreducible modules of the stabilizer \(P=G^{x_0}\).
Equivalently, \(X\) admits a K\"ahler metric for which \(X\) becomes a compact Hermitian symmetric space, and \(G\) has a maximal compact subgroup acting as isometries of that metric.
Every effective flag variety \(X=G/P\) has a unique complex homogeneously embedded effective cominuscule subvariety not tangent to any \(G\)-invariant holomorphic subbundle of \(TX\) \cite{McKay:2020}, its \emph{associated cominuscule subvariety}.

Among all \(G\)-invariant holomorphic filtrations of the tangent bundle \(TX\), there is a unique finest one, as we will see, and thereby an associated graded holomorphic vector bundle.
In section~\vref{subsec:tgt.filtration} we will introduce some integer invariants of flag varieties, which we call \emph{tallies}.
The tally of an irreducible flag variety is the minimal dimension of a nonzero invariant subbundle of the associated graded of the tangent bundle; we will see that tallies are easy to calculate in examples.
\subsection{Minimal geometries}
The second main theorem of this paper:
\begin{theorem}\label{theorem:Iitaka.associated}
Suppose that \(\G\to M\) is a minimal holomorphic parabolic geometry, with model \((X,G)\), on a smooth projective variety \(M\).
Then the numerical dimension of the canonical bundle of \(M\), hence the Kodaira dimension of \(M\), is at most the sum of the tallies of the model \((X,G)\).
Every polynomial 
\[
P(c_1,c_2,\dots,c_n)
\]
in characteristic classes of the tangent bundle which is satisfied on the associated cominuscule subvariety \((\breve X,\breve G)\) of \((X,G)\) is satisfied on \(M\).
Every meromorphic map \(\rationalMap{Z}{M}\) from a compact complex manifold \(Z\) extends uniquely to a holomorphic map.
\end{theorem}
This theorem surprised the author, since the dimension of \(\breve X\) is smaller than that of \(X\) unless \(X\) is cominuscule, and the tally can be smaller still, even smaller than the numerical dimension of \(\breve X\).
\begin{example}
In real differential geometry, quaternionic contact structures are regular normal parabolic geometries modelled on a certain real form of the flag variety \(\dynkin[parabolic=2]C3\); they live on \(7\)-dimensional manifolds, and can be described as certain \(4\)-plane fields.
Consider instead the holomorphic parabolic geometries with this model, on \(7\)-dimensional smooth projective varieties.
We won't assume regularity or normality.
Every smooth projective variety \(M\) (or compact K\"ahler manifold) with such a geometry satisfies all of the characteristic class equations of its model \cite{McKay:2011} p.~2, theorem~1.
These relations are the ideal generated by the equations
\begin{align*}
0&=c_{1}-5\,\varepsilon \\
0&=7\,\varepsilon^4\,c_{2}-82\,\varepsilon^6 \\
0&=c_{3}-2\,c_{2}\,\varepsilon+7\,\varepsilon^3 \\
0&=3\,c_{4}-5\,c_{2}\,\varepsilon^2+14\,\varepsilon^4 \\
0&=3\,c_{5}-10\,\varepsilon^3\,c_{2}+91\,\varepsilon^5 \\
0&=7\,c_{6}-24\,\varepsilon^6 \\
0&=7\,c_{7}-6\,\varepsilon^7 \\
0&=3\,c_{2}^2-74\,c_{2}\,\varepsilon^2+455\,\varepsilon^4 \\
\end{align*}
where \(\varepsilon\) is the characteristic class of a certain line bundle \cite{McKay:2011} p.~23, table~3.
If the geometry is not minimal, it drops to another holomorphic parabolic geometry modelled on a quotient flag variety, i.e. with fewer crosses in its Dynkin diagram.
But there is only one cross in the Dynkin diagram.
So either \(M=X=C_3/P_2=\dynkin[parabolic=2]C3\) is the model, with its standard holomorphic quaternionic contact structure, or it is a minimal geometry satisfying the characteristic class relations of its associated cominuscule \(\breve{X}\).
The associated cominuscule \(\breve{X}\) is the \(3\)-dimensional quadric hypersurface \(\breve{X}=\dynkin C{*x}\)\cite{McKay:2020} p.~10, theorem~1, with root system
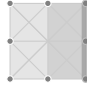
\[
\drawroots{B}{1} 
\]
The characteristic class ring of a \(3\)-dimensional quadric hypersurface \(\dynkin C{*x}\) is generated by the relations
\begin{align*}
0&=c_{1}-4\delta \\
0&=c_{2}-6\delta^2 \\
0&=c_{3}-4\delta^3 \\
\end{align*}
where \(\delta\) is the characteristic class of a certain line bundle  \cite{McKay:2011} p.~23, table~3.
Comparing the two equations for \(c_3\), we find first that \(\varepsilon=3\delta/5\) in Dolbeault cohomology, and then plug in to find the surprising relation \(c_3=0\).
So every minimal \(\dynkin[parabolic=2]C3\)-geometry has \(c_3=0\), which is not clear from looking at the model.
As we will see when we define tallies, the tally of \(\dynkin[parabolic=2]C3\) is \(3\), visible in the Hasse diagram of the root lattice of \(\dynkin[parabolic=2]C3\) on p.~44, \cite{McKay:2020}, which in this example provides no further information.
Since \(c_1^4=0\), the Kodaira dimension is at most three.
By theorem~\vref{thm:Iitaka.group.scheme}, assuming that \(M\) does not contradict the abundance conjecture, after perhaps replacing \(M\) by a finite covering space, \(M\) is an abelian group scheme of dimension \(7\), with ample base of dimension at most \(3\), so fibers of dimension at least \(4\).
\end{example}
\subsection{Normalization}
Recall that an irreducible affine variety \(X\) is \emph{normal} if every rational function \(f\) on \(X\) which satisfies a monic polynomial
\[
0=f^n+a_1f^{n-1}+\dots+a_n
\]
with coefficients \(a_1,\dots,a_n\) regular functions on \(X\) is regular on \(X\) (i.e. the regular functions are integrally closed) \cite{Mumford:1999} p.~197 \S{}III.8, \cite{Shafarevich2013} p.~125 \S5.
Equivalently, every continuous function holomorphic on the smooth locus of the affine variety \(X\) extends to be holomorphic on \(X\) \cite{Brieskorn.Knoerrer:1986} p.~391.
A variety \(X\) is \emph{normal} if it is covered by open normal affine varieties.
In particular, smooth varieties are normal.
A \emph{finite morphism} \(X\to Y\) of projective varieties is a morphism so that, locally, the regular functions on open sets of \(X\) form an integral ring over those on open sets of \(Y\) \cite{Cutkosky:2018} p.~40.
A \emph{normalization} of a variety \(X\) is a normal variety \(X_0\) with a birational and finite morphism \(X_0\to X\).
Every variety has a normalization, unique up to isomorphism \cite{Mumford:1999} p.~197 \S{}III.8, \cite{Shafarevich2013} p.~125 \S5.

\subsection{Fundamental groups}
A connected complex projective variety \(M\) has \emph{large fundamental group} if every connected positive dimensional closed
complex subvariety \(Z \,\subset\, M\) with normalization \(Z_0 \to Z\) has infinite fundamental group image \(\fundamentalGroup{Z_0} \to \fundamentalGroup{M}\) \cite{Kollar:1993}.
We recall that any variety with a minimal geometry has large fundamental group:
\begin{lemma}[Biswas \& McKay \cite{Biswas.McKay:2024}]\label{lemma:dev.large}
Suppose that \(M\) is a smooth projective variety bearing a minimal geometry modelled on a complex homogeneous space \((X,\,G)\).
Suppose that \rationalMap{Z}{M} is a meromorphic map from a positive dimensional reduced connected compact complex space.
Then this meromorphic map extends uniquely to a holomorphic map.
Suppose that \(Z \to M\) is a finite map.
Take the normalization \(Z_0 \to  Z \to  M\).
Either
\begin{enumerate}
\item
the image of \(\fundamentalGroup{Z_0}\to\fundamentalGroup{M}\) is infinite or
\item
\(Z\) has large fundamental group.
\end{enumerate}
\end{lemma}
\subsection{General type}\label{section:KE}
\begin{example}
Suppose that \(X=G/P\) is an effective cominuscule variety.
Denote by \(X'=G'/P'\) its dual noncompact Hermitian symmetric space \cite{Helgason:1978}, i.e. \(G' \subset G\) is a real form acting on \(X\), and \(P'=P\cap G'\) is compact, and \(G'\) acts on \(X\) with an open orbit \(X'=G'/P'\subseteq X\). 
Pullback the standard flat \((X,G)\)-geometry on \(X\) to a flat \((X,G)\)-geometry on \(X'\).
If \(\Gamma \subset G'\) acts on \(X'\) freely and properly then the flat \((X,G)\)-geometry on \(X'\) quotients to a unique flat \((X,G)\)-geometry on the manifold \(M:=\Gamma \backslash X'\), called the \emph{standard geometry} on \(M\).
The Bergman metric of the open set \(X'\subseteq X\) is a biholomorphism invariant K\"ahler metric of negative Ricci curvature \cite{Helgason:1978} chapter~VIII.
Hence \(M\) inherits this metric.
If \(\Gamma\) is cocompact, i.e. \(M\) is compact, then the metric on \(M\) gives \(M\) a positive curvature metric on its canonical bundle, so \(M\) is ample, hence is a smooth projective variety, a \emph{Hermitian symmetric variety} \cite{Satake:1980}.
When \(\Gamma\) is defined by ``congruence relations'', \(M\) is called a \emph{Shimura variety} \cite{Milne:2005}; we won't need the notion of Shimura variety below, but we have seen it arise in the Jahnke--Radloff classification of holomorphic projective connections on smooth projective varieties.
\end{example}
\begin{example}
If \(\pr{X,G}=\pr{\Proj{n},\PSL{n+1}}\) then any holomorphic \(\pr{X,G}\)-geometry is called a \emph{holomorphic projective connection}.
The space of holomorphic projective connections on any Riemann surface is canonically identified with the space of holomorphic quadratic differentials \cite{Loray/MarinPerez:2009}. 
\end{example}
\begin{example}
Call a holomorphic cominuscule geometry \emph{standard} if, after perhaps replacing by the pull back to a finite \'etale covering space, it becomes a product geometry of the form \(M=M' \times \prod_j M_j\), where each \(M_j\) is a compact Riemann surface with a holomorphic projective connection and \(M'=\Gamma \backslash X'\) is the standard geometry, for \(X'\) the noncompact dual of a cominuscule variety.
\end{example}
A compact irreducible reduced complex space is of \emph{general type} if its canonical bundle is big.
\begin{example}
Any compact complex manifold \(M\) with \(c_1<0\) has general type, as does any modification.
\end{example}
The third main theorem of this paper:
\begin{theorem}\label{theorem:general.type}
Suppose that \(M\) is a connected smooth projective variety \(M\) of general type.
Suppose that \(M\) admits an effective holomorphic parabolic geometry.
Then \(c_1(TM)<0\) and the model \((X,G)\) is effective cominuscule and \(M\) is a locally Hermitian symmetric variety with universal covering space \(X'\) the Hermitian symmetric space dual to \(X\).
All effective holomorphic parabolic geometries on \(M\) have that same model \((X,G)\).
The moduli space of holomorphic parabolic geometries on \(M\) is a finite dimensional complex vector space.
The normal holomorphic parabolic geometries on \(M\) form a finite dimensional linear subspace.
Every holomorphic normal parabolic geometry on \(M\) is a standard \((X,G)\)-geometry.
The following are equivalent:
\begin{itemize}
\item
The variety \(M\) admits a unique holomorphic parabolic geometry.
\item
No finite unramified covering space of \(M\) splits into a product \(M_1 \times M_2\) with \(M_1\) a compact Riemann surface.
\end{itemize}
If \(X\) does not split \(G\)-equivariantly into a product \(X=Y\times\Proj{1}\), then both are true.
\end{theorem}
\begin{proof}
Since \(M\) is general type, \(M\) is not a fiber bundle with flag variety fibers.
If \(M\) contains a rational curve, the geometry drops and \(M\) is such a fiber bundle, so not of general type, by theorem~\vref{theorem:drop}.
Therefore \(M\) contains no rational curves.
Every smooth projective variety of general type which contains no rational curves has \(c_1<0\) \cite{deBarre:2001} p. 219.
Since \(c_1<0\), \(M\) has Kodaira dimension equal to dimension.
By theorem~\vref{theorem:Iitaka.associated}, either the geometry is cominuscule or the Iitaka fibration has positive dimensional fibers and hence \(c_1\) is not negative.
Cominuscule geometries on compact complex manifolds with \(c_1(M)<0\) are classified in \cite{McKay:2016} p. 5 theorem 1, giving the result above.
\end{proof}
\begin{theorem}\label{theorem:comini}
Suppose that \(M\) is a smooth projective variety carrying a regular effective holomorphic parabolic geometry.
Suppose that \(M\) is not uniruled.
Then the model of that geometry is cominuscule.
\end{theorem}
\subsection{Locally homogeneous structures}
The theory of Cartan geometries is unfamiliar to many geometers, but many are familiar with the theory of \((X,G)\)-structures.
We will make the link to that theory explicit here, so that we can discuss the relation of our theorems to the theory of nonclassical domains.
A complex homogeneous space \((X,G)\) is \emph{strong} if, for any component of \(X\), any element of \(G\) which fixes every point of that component fixes every point of \(X\).
It is \emph{effective} if only \(1\in G\) fixes every point of \(X\).
An \emph{\((X,G)\)-structure} is a flat holomorphic parabolic geometry modelled on a strong effective complex homogeneous space \((X,G)\) \cite{McKay:2023} p~92.
If a connected complex manifold \(M\) carries an \((X,G)\)-structure then the universal covering space \(\tilde{M}\) has a local biholomorphism to the model \(\tilde{M}\xrightarrow{\delta}X\) which is also a local isomorphism of \((X,G)\)-structures, a \emph{developing map} \cite{McKay:2023} p.~35.
The developing map is equivariant for a unique group morphism \(\pi_1(M)\xrightarrow{h}G\), the \emph{holonomy morphism}.
Form the pullback bundle \(\tilde\G\) by
\[
\begin{tikzcd}
\tilde\G\arrow[r]\arrow[d]&\G\arrow[d]\\
\tilde{M}\arrow[r]&M\\
\end{tikzcd}
\]
so that 
\[
\G=\pi_1(M)\backslash\tilde\G.
\]
The developing map extends \cite{McKay:2023} to a \(P\)-equivariant bundle map \(\Delta\):
\[
\begin{tikzcd}
\tilde\G\arrow[r,"\Delta"]\arrow[d]&G\arrow[d]\\
\tilde{M}\arrow[r,"\delta"]&X
\end{tikzcd}
\]
identifying the Cartan connections.
\subsection{Dropping structures}
Suppose that \((X,G)\), \((X',G)\) are two flag varieties with a \(G\)-equivariant holomorphic map \(X\to X'\).
Pick points \(x_0\in X\) mapping to \(x_0'\in X'\) and let \(P:=G^{x_0}\) and \(P':=G^{x_0'}\) so \(P\subseteq P'\).
Every \((X',G)\)-structure \(\G\to M'\) on a complex manifold \(M'\) has a lift: we let \(M:=\G/P\) as before, but then use the same map \(\tilde\G\to G\) as above:
\[
\begin{tikzcd}
\tilde\G\arrow[r,"\Delta"]\arrow[d]&G\arrow[d]\\
\tilde{M}\arrow[r,"\delta'"]&X
\end{tikzcd}
\]
If an \((X,G)\)-structure drops, its developing map arises from the same map \(\tilde\G\xrightarrow{\Delta}G\):
\[
\begin{tikzcd}
\tilde\G\arrow[r,"\Delta"]\arrow[d]&G\arrow[d]\\
\tilde{M}\arrow[r,"\delta"]\arrow[d]&X\arrow[d]\\
\tilde{M}'\arrow[r,"\delta'"]&X'\\
\end{tikzcd}
\]

\subsection{Nonclassical domains}
From our theorems, we easily recover part of the theory of nonclassical domains.
Take an effective flag variety \((X,G)\).
A \emph{homogeneous domain} \(X_0\subseteq X\) is an open orbit of a real Lie subgroup \(G_0\subseteq G\).
It is a \emph{flag domain} if, in addition, \(G_0\subseteq G\) is a noncompact real form.
A homogeneous domain is \emph{classical} if there is a cominuscule variety \((X',G)\) with \(G\)-equivariant holomorphic map \(X\to X'\), and a \(G_0\)-invariant domain \(X_0'\subseteq X'\), so that \(X\to X'\)  restricts to a \(G_0\)-equivariant surjection \(X_0\to X_0'\).
Let \(G_0'\) be the quotient of \(G_0\) by the elements acting trivially on \(X_0'\).
We say \((X_0,G_0)\) \emph{drops} to \((X'_0,G_0')\).

An \emph{\((X_0,G_0)\)-structure} is an \((X,G)\)-structure whose developing map has image inside a homogeneous domain \(X_0\subseteq X\) and whose holonomy morphism has image inside \(G_0\).

The developing map of an \((X,G)\)-structure is an embedding if and only if \(M\) is the quotient of the open set
\[
\delta(\tilde{M})\subseteq X
\]
by the group
\[
h(\pi_1(M))\subseteq G.
\]
So an \((X_0,G_0)\)-structure is a generalization of the concept of a cocompact quotient of a flag domain.

\begin{corollary}\label{cor:Carlson}
Suppose that a smooth projective variety \(M\) has an \((X_0,G_0)\)-structure, for some flag variety \((X,G)\) and homogeneous domain \(X_0\subseteq X\).
Either 
\begin{itemize}
\item
\(X_0=X\) and \(M\) is biholomorphic to \(X\) by a biholomorphism which identifies the \((X,G)\)-structures or
\item
The homogeneous domain \((X_0,G_0)\) is classical, say dropping to \((X_0',G_0)\), and the \((X_0,G_0)\)-structure on \(M\) drops to an \((X_0',G_0')\)-structure \(M\to M'\).
Either 
\begin{itemize}
\item
\(M'\) is an abelian variety with translation invariant \((X_0',G_0)\)-geometry or 
\item
\(M'\) is an abelian group scheme (after perhaps replacing \(M\) and \(M'\) by finite covering spaces) or
\item
\(M'\) is a Hermitian symmetric variety.
If \(X\) is not a \(G\)-equivariant product \(X=Y\times\Proj{1}\) with a projective line factor then the \((X_0',G_0')\)-structure on \(M'\) is unique, hence the \((X_0,G_0)\)-structure on \(M\) is unique, hence is the one induced from the Hermitian symmetric metric on \(M'\) or
\item
\(M'\) is a counterexample to the abundance conjecture.
\end{itemize}
If in addition \(X_0'\) is Brody hyperbolic then \(M'\) is a Hermitian symmetric variety or a counterexample to the abundance conjecture.
Finally, if \((X_0',G_0')\) is a noncompact Hermitian symmetric space then \(M'\) is a Hermitian symmetric variety.
\end{itemize}
\end{corollary}
\begin{proof}
The \((X,G)\)-geometry is flat, so regular.
The geometry drops, say from an \((X,G)\)-geometry to an \((X',G)\)-geometry, where \(X\to X'\) is a \(G\)-equivariant morphism of flag varieties.
The geometry is isomorphic to the model just when \(X'\) is a point.
So suppose that \(X'\) is not a point.
The geometry has the same Cartan connection, so remains flat, so regular.
By theorem~\vref{theorem:comini}, \((X',G)\) is cominuscule.
Write \(X=G/P\) and \(X'=G/P'\) so \(P\subseteq P'\).

Let \(X'_0\subseteq X'\) be the image of \(X_0\subseteq X\).
The fibers \(M\to M'\) of the drop are flag varieties, so simply connected, so their preimages in \(\tilde{M}\to\tilde{M}'\) are isomorphic flag varieties.
Each lies in the image of the developing map, as fibers map to fibers when we drop, so there is an open set of points of \(X'\) above which for which the fibers of \(X\to X'\) lie inside \(X_0\).
The group \(G\) takes those fibers to one another.
By \(G_0\)-homogeneity, \(X_0\) is a union of fibers of \(X\to X'\).
By \(G_0\)-homogeneity, \(X_0\to X_0'\) is a holomorphic fiber bundle mapping.

By theorem~\vref{thm:Iitaka.group.scheme}, since \(M'\) contains no rational curve, \(M'\) is an abelian variety, abelian group scheme (after perhaps replacing by a finite covering space) or has ample canonical bundle.

If \(X_0'\) is Brody hyperbolic then \(X_0'\) contains no entire curves, so \(\tilde{M}'\) contains no entire curves, so \(M'\) contains no entire curves, so \(M'\) is not an abelian variety or abelian group scheme with positive dimensional abelian variety fibers.
A \(G_0'\)-invariant positive curvature metric on the canonical bundle of \(X_0'\) pulls back to \(\tilde{M}'\) and then drops to \(M'\), so \(M'\) has canonical bundle spanned by global sections.
In the absence of projective line factors, the \((X',G)\)-structure is unique \cite{McKay:2016} p.~6, theorem~1.
\end{proof}
As of yet, we don't recover all of the results on compact K\"ahler manifolds of Carlson and Toledo\cite{Carlson.Toledo:2014} or on improper varieties from Griffiths, Robles and Toledo \cite{Griffiths.Robles.Toledo:2014}.
But we can easily generate similar theorems, replacing the flag variety \((X,G)\) by any complex algebraic homogeneous space by applying \cite{Biswas.McKay:2024} p.~1 theorem~1.
\section{Flag varieties}%
\label{section:flag.varieties}
We use notation and terminology for flag varieties from Knapp \cite{Knapp:2002}. 
A \emph{flag variety} \((X,G)\), also called a \emph{generalized flag variety} or a \emph{rational homogeneous variety}, is a complex projective variety \(X\) acted on transitively and holomorphically by a connected complex semisimple Lie group \(G\) \cite{Borel:1991} p. 148.
We will need to make use of ineffective flag varieties, i.e. \(G\) might not act faithfully on \(X\).
It is traditional to denote the stabilizer \(G^{x_0}\) of a point \(x_0 \in X\) as \(P\); the group \(P \subset G\) is a complex linear algebraic subgroup.
Denote the Lie algebras of \(P \subset G\) by \(\LieP \subset \LieG\). 
One can select a Cartan subgroup of \(G\) lying inside \(P\), whose positive root spaces all lie in \(\LieP\).
A simple root \(\aa\) is \(P\)-\emph{compact} (\emph{compact} if \(P\) is understood) if the root space of \(-\aa\) belongs to the Lie algebra of \(P\).
Each flag variety is determined uniquely up to isomorphism by the Dynkin diagram of \(G\) decorated with \(\dynkin{A}{*}\) on each compact simple root and \(\dynkin{A}{x}\) on each noncompact root \cite{Baston/Eastwood:1989} p. 14.
\begin{example}
The complexified octave projective plane is the flag variety \((X,G)\) where \(G=E_6\), with decorated Dynkin diagram
\[
\dynkin[parabolic=1]E6
\]
\end{example}
\subsection{Grading}
The Cartan subalgebra, together with the root spaces of the compact roots, generates a reductive subalgebra \(\LieGZ \subset \LieP\); let \(\GZ:=e^{\LieGZ}\subseteq G\).
There is a unique element \(e \in \LieH\), the \emph{grading element}, in the centre of \(\LieGZ\) so that \(\Pheight{\aa}=0\) for all compact roots, and \(\Pheight{\aa}=1\) just when \(\aa\) is a noncompact root with root space \(\rtsp{\LieG}{\aa} \subset \LieP\) and \(\aa\) is not a sum of noncompact roots \cite{Cap/Slovak:2009} p. 239 proposition 3.1.2.
We grade weights by inner product with \(e\).
For example, if \((X,G)\) has \(G=G_2\) and \(X=G/P\) with Dynkin diagram \(\dynkin[parabolic=2]G2\),  the root system with positive grades marked looks like
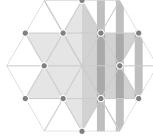
\[
\drawroots{G}{2}
\]
on the roots, hence on \(\LieP\).
The zero grade:
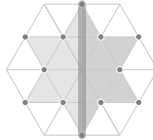
\[
\begin{tikzpicture}[baseline=-.5]
\begin{rootSystem}{G}
\roots
\parabolic{2}
\draw
[line width=3pt,gray,opacity=0.5,line cap=round] 
(hex cs:x=-1,y=2) -- (hex cs:x=1,y=-2);
\end{rootSystem}
\end{tikzpicture}
\]
The corresponding grading on \(\LieG/\LieP\) is therefore
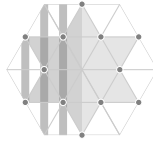
\[
\begin{tikzpicture}[baseline=-.5,scale=-1]
\begin{rootSystem}{G}
\roots
\parabolic{2}
\parabolicgrading
\end{rootSystem}
\end{tikzpicture}
\]
Hence the associated flag variety has \(3\)-graded tangent bundle.

We partial order weights by grade: write \(\mu \le \nu\) to mean that \(\Pheight{\mu} \le \Pheight{\nu}\) and so on.
Let \(\ParaRts\) be the set of roots which are \(\ge 0\), i.e. whose root spaces lie in \(P\).
Let
\[
\cptRts:=\ParaRts\cap-\ParaRts,
\]
and 
\[
\ncptPosRts:=\ParaRts\backslash\cptRts.
\]
Note that these are \emph{not} necessarily the positive roots in the sense of the Cartan subgroup of \(G\); indeed they are the positive roots in the usual sense just when \(P\) is a Borel subgroup.
Similarly define \(\Delta_{\ge k}\) to be the roots \(\aa\) with \(\Pheight{\aa} \ge k\).
We make \(\LieG\) a graded \(\LieG_0\)-module by grading each root vector by the grade of its root.
Since \(\LieP\) is precisely the sum of nonnegatively graded roots, \(\LieG\) is thereby a filtered \(\LieP\)-module, hence a filtered \(P\)-module.
For an extensive discussion of filtrations and gradings of flag varieties, see \cite{McKay:2020}.
\subsection{Cominuscule flag varieties}
A flag variety is \emph{cominuscule} if \(\LieG/\LieP=T_{x_0} X\) is a sum of irreducible complex algebraic \(P\)-modules.
This occurs just when, for some and hence every maximal compact subgroup \(K\subseteq G\), \((X,K)\) is a compact Hermitian symmetric space \cite{Kostant:1961} p. 379 Proposition 8.2, \cite{Baston/Eastwood:1989} p. 26.
Some authors prefer the term \emph{compact Hermitian symmetric space}, \emph{cominuscule Grassmannian}, or \emph{generalized Grassmannian} to \emph{cominuscule variety}.
Every effective cominuscule variety is a product of the following irreducible effective cominuscule varieties \cite{Baston/Eastwood:1989} p.~26, Example~3.1.10, \cite{Kobayashi/Nagano:1964} theorem 1 p. 401:\par\noindent{}%
\begingroup
\small
\begin{longtabl}{@{}%
>{\columncolor[gray]{.9}$}r<{$}%
>{\columncolor[gray]{.93}}p{2.6cm}%
>{\columncolor[gray]{.9}$}r<{$}%
>{\columncolor[gray]{.93}}l@{}}
\toprule
G&\(G/P\)&\operatorname{dim}&\text{description}\\
\midrule
\endfirsthead
\toprule
G&\(G/P\)&\operatorname{dim}&\text{description}\\
\midrule
\endhead
\bottomrule
\endfoot
\bottomrule
\endlastfoot
A_r&\dynkin[labels={,,,k,,,}]{A}{**.*x*.**}&k(r+1-k)&Grassmannian of $k$-planes in $\C[r+1]$
\\
B_r&\dynkin[parabolic=1]{B}{}&2r-1&quadric hypersurface in $\Proj{2r}$\\
C_r&\dynkin[parabolic=16]{C}{}&\frac{r(r+1)}{2}&space of Lagrangian $r$-planes in $\C[2r]$\\
D_r&\dynkin[parabolic=1]{D}{}&2r-2&quadric hypersurface in $\Proj{2r-1}$\\
D_r&\dynkin[parabolic=32]{D}{}&\frac{r(r-1)}{2}& 
space of null $r$-planes in $\C[2r]$ \\
E_6&\dynkin[parabolic=1]{E}{6}&16&complexified octave projective plane\\
E_7 &\dynkin[parabolic=64]{E}{7}&27&space of null octave \(3\)-planes in octave \(6\)-space
\end{longtabl}
\endgroup
\subsection{Parabolic subgroups}
A Zariski closed subgroup \(P\subseteq G\) of a connected linear algebraic group \(G\) is \emph{parabolic} if \(X:=G/P\) is a projective variety, and this occurs just when \(X\) is a flag variety for a maximal semisimple Levi factor of \(G\), and just when \(P\) contains a maximal connected solvable subgroup of \(G\) \cite{Borel:1991} p. 148.
Every parabolic subgroup is connected p. 197 Proposition 14.18.
The unipotent radical of \(P\) is denoted \(\GN\subseteq P\), and a maximal reductive Levi factor of \(P\) is denoted \(\GZ\subseteq P\), so \(P=\GZ\ltimes\GN\).
(This is potentially confusing; the reader might expect to write these as \(P_+\) and \(P_0\) since they lie in \(P\), but this notation is standard \cite{Cap/Slovak:2009} p. 293 theorem 3.2.1, and due to the presence of the grading of the Lie algebra of \(G\) which we have defined above.)
A flag variety is cominuscule just when \(\GN\) is abelian \cite{Cap/Slovak:2009} p. 296 \S{}3.2.3.
Denote the center of the unipotent radical by \(Z\coloneq Z_{\GN}\).
\subsection{The tangent filtration}\label{subsec:tgt.filtration}
The quotient \(V=\LieG/\LieP\) has the obvious quotient filtration \(V_*\) as a filtered \(P\)-module.
The Lie bracket on \(\LieG\) descends to a graded algebra structure on \(\graded{*}{V}\): the \emph{algebraic bracket}.
Denote the Cartan subgroup by \(H\subseteq P\).
For any set \(\Gamma\) of weights, and any \(H\)-module \(V\), let \(V_{\Gamma}\) be the sum of the weight spaces of \(V\) for those weights.
In particular, consider \(V:=\LieG/\LieP\); the \emph{tally} \(N_X\) of an irreducible flag variety \((X,G)\) is the infimum number of elements in a nonempty set \(\Gamma\) of roots for which \(V_{\Gamma}\subset V\) is a \(\GZ\)-module.
The \emph{tallies} \(N_i\) of a reducible flag variety \(X=\prod X_i\) are the tallies of its factors.
\begin{example}
Take the flag variety \((X,G)\) with \(G=G_2\) given by the grading
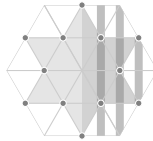
\[
\drawroots{G}{2}
\]
on the roots, hence on \(\LieP\).
The corresponding grading on \(\LieG/\LieP\) is
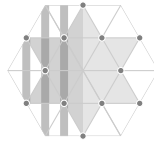
\[
\begin{tikzpicture}[baseline=-.5,scale=-1]
\begin{rootSystem}{G}
\roots
\parabolic{2}
\parabolicgrading
\end{rootSystem}
\end{tikzpicture}
\]
The tally is \(1\), since the middle vertical line in 
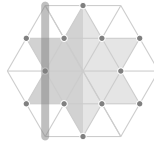
\[
\begin{tikzpicture}[baseline=-.5,scale=-1]
\begin{rootSystem}{G}
\roots
\parabolic{2}
\draw
[line width=3pt,gray,opacity=0.5,line cap=round] 
(hex cs:x=0,y=2) -- (hex cs:x=2,y=-2);
\end{rootSystem}
\end{tikzpicture}
\]
has one root on it, say \(\gamma\), and the associated \(H\)-module \(V=V_{\{\gamma\}}\) is a \(\GZ\)-submodule of \(\LieG/\LieP\).
\end{example}
\subsection{Opposite parabolic subgroups}
Two parabolic subgroups \(P,\op{P}\subseteq G\) of a complex semisimple Lie group are \emph{opposite} if \(P\cap\op{P}\) is a reductive Levi factor of both \(P\) and \(\op{P}\).
All Borel subgroups of \(G\) are conjugate \cite{Borel:1991} p.~147 chapter IV 11.1, each containing a Cartan subgroup, hence the Lie algebra of \(P\) is the sum of (1) the Cartan subalgebra with (2) various root spaces.
Every automorphism of a root system arises from an automorphism of the associated semisimple Lie group \cite{Fulton/Harris:1991} p.~498 Proposition D.40, \cite{Humphreys:1978}, p.~87, \S16.5.
Hence there is an automorphism \(G\xrightarrow{a}G\) of \(G\) which yields \(\alpha\mapsto-\alpha\) in the root system.
(We can define such an automorphism explicitly as \(e_{\alpha}\mapsto -e_{-\alpha}\) on root vectors in a Chevalley basis; see \S\vref{subsubsection:ChevalleyBases}.)
Our automorphism sends \(P\) to an opposite parabolic subgroup \(\op{P}:=aP\) with \(P\cap \op{P}=\GZ\).
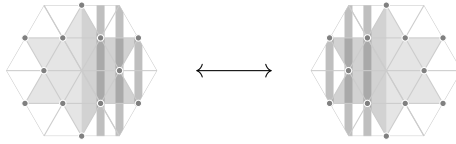
\[
\begin{tikzpicture}[baseline=-.5]
\begin{rootSystem}{G}
\roots
\parabolic{2}
\parabolicgrading
\end{rootSystem}
\end{tikzpicture}
\begin{tikzpicture}[baseline=-.5]
\draw[white] (-1,-1) rectangle (1,1);
\draw[<->] (-.5,0) -- (.5,0);
\end{tikzpicture}
\begin{tikzpicture}[baseline=-.5,scale=-1]
\begin{rootSystem}{G}
\roots
\parabolic{2}
\parabolicgrading
\end{rootSystem}
\end{tikzpicture}
\]
Letting \(G_-:=aG_+\), \(G_+\cap G_-=\set{1}\).
An open subset of \(G\) consists of elements uniquely expressed as a product \(p,q\in P\times G_-\mapsto pq\in G\) \cite{Cap/Slovak:2009} p.~294, \cite{Borel:1991} p.~198 Proposition 14.21.
Every root system also has an automorphism, traditionally called \(w_0\), which belongs to the Weyl group and which interchanges the positive and negative roots of a root system \cite{Cap/Slovak:2009} pp.~323-324; it is the unique element of the Weyl group of maximum length.
Note that \(w_0\) might not reverse the signs of simple roots \cite{Cap/Slovak:2009} p.~324.
The Weyl group lifts to a group of automorphisms of the Lie group \(G\), after perhaps extension by some finite group of order a power of \(2\).
We can use such an extension of \(w_0\) in place of \(a\) throughout this paper, as we will only need that \(a\) is an automorphism of a given root system which extends to an automorphism of \(G\) taking a given parabolic subgroup to an opposite.
%
\subsection{Definition of the associated cominuscule}
Take a flag variety \((X,G)\) and opposite parabolic subgroups \(P,\op{P}\subseteq G\), so that \(P\) is the stabilizer of \(x_0\in X\).
Take their unipotent radicals \(G_+,\op{G_+}\) and the centers \(Z,\op{Z}\) of these.
Let \(\breve{G}:=\left<Z,\op{Z}\right>\subseteq G\) be the subgroup generated by \(Z\cup\op{Z}\), \(\breve{P}:=\breve{G}\cap P\), \(\breve{X}:=\breve{G}/\breve{P}\).
Then \((\breve{X},\breve{G})\) is the \emph{associated cominuscule subvariety} through the point \(x_0\in X\), a positive dimensional homogeneously embedded cominuscule subvariety of \(X\) \cite{McKay:2020}.
The associated cominuscule of every flag variety is known explicitly \cite{McKay:2020}.
The group \(\breve{G}\) is semisimple with roots precisely the \(G\)-roots generated, as a root system, by the \(P\)-maximal roots.
The subgroup \(\breve{P}\subset\breve{G}\) is a parabolic subgroup and \(\breve{G}\) acts on \(\breve{X}\) with only a finite normal subgroup acting trivially.

Let \(G'\subseteq G\) be the subgroup preserving the associated cominuscule \(\breve{X}\subseteq X\), a closed subgroup because \(\breve{X}\) is compact, and a complex linear algebraic subgroup because \(\breve{X}\) is a complex subvariety.
While \(G'\) contains \(\breve{G}\), it is often larger than \(\breve{G}\) and is not always semisimple \cite{McKay:2020}.
The obvious map \(G'\to\breve{G}\) is onto.

Let \(\rtspMax{\LieG}, \rtspMin{\LieG} := \bigoplus \LieG_{\aa}\), where the sum is over the roots \(\aa\) where \(\aa\) is maximal, minimal in its irreducible factor of the root system of \(G\).
Let \(\LieP':=\LieGZ \oplus \rtspMax{\LieG} \subset \LieG\), a Lie subalgebra, and let \(P'\subseteq P\) be the connected subgroup with that Lie subalgebra.
This subgroup preserves the associated cominuscule, since \(\LieGZ\) does, acting on the sum of \(P\)-maximal and \(P\)-minimal root vectors.
\begin{example}
There are two flag varieties of \(B_2=\PSO{5}=C_2=\Symp{4}\); consider \(X=\dynkin[parabolic=2]B2\), which is \(\mathbb{P}^3\) acted on by the symplectic group of the standard complex linear symplectic form on \(\C[4]\), a homogeneous projective \(3\)-fold with a \(B_2\)-invariant contact path geometry, with roots
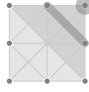
\[
\drawroots{B}{2}
\]
Note that \(\LieG_0=\LieSL{2}\oplus\C{}\).
The associated cominuscule subvariety is \(\breve{X}=\Proj{1}\subseteq X\), a line transverse to the contact planes:
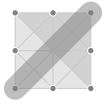
\[
\begin{tikzpicture}[baseline=-.5]
\begin{rootSystem}{B}
\roots
\parabolic{2}
\draw[/root system/grading,line width=.3cm] (square cs:x=1,y=1) -- (square cs:x=-1,y=-1);
\end{rootSystem}
\end{tikzpicture}
\]
Therefore \(\LieP'=\breve{\LieP}\oplus\LieSL{2}=\breve{\LieP}\oplus\LieG_0\).
\end{example}
The Lie algebra of \(\breve\LieG\) has root system \(\breve\Delta\subset\Delta\) consisting precisely of the \(P\)-maximal and \(P\)-minimal roots and every \(P\)-compact root which is difference of a \(P\)-maximal and a \(P\)-minimal.
On the other hand, the \(P\)-compact roots form another root system \(\Delta_0:=\breve\Delta^{\perp}\), which can intersect \(\breve\Delta\).
The Lie algebra \(\breve\LieG\) is that of the root system \(\breve\Delta\), i.e. the direct sum of the \(\breve\Delta\)-root vectors and the \(\breve\Delta\)-coroots.
Similarly, the maximal complex semisimple subalgebra of \(\LieG_0\) is that of the root system \(\Delta_0\).
\begin{example}
There might be no overlap between cominuscule and compact roots: 
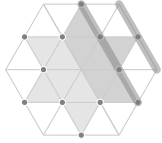
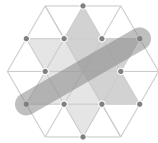
\[
\drawroots{G}{1}
\]
with associated cominuscule
\[
\begin{tikzpicture}[baseline=-.5]
\begin{rootSystem}{G}
\roots
\parabolic{1}
\draw[/root system/grading,line width=.3cm] (hex cs:x=1,y=1) -- (hex cs:x=-1,y=-1);
\end{rootSystem}
\end{tikzpicture}
\]
\end{example}
\begin{example}
There might be some overlap between cominuscule and compact roots: 
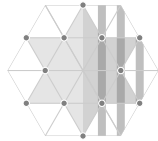
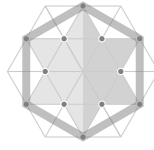
\[
\drawroots{G}{2}
\]
with associated cominuscule
\[
\begin{tikzpicture}[baseline=-.5]
\begin{rootSystem}{G}
\roots
\parabolic{2}
\draw[/root system/grading] (hex cs:x=-1,y=2) -- (hex cs:x=1,y=1) -- (hex cs:x=2,y=-1) -- (hex cs:x=1,y=-2) -- (hex cs:x=-1,y=-1) -- (hex cs:x=-2,y=1) -- cycle;
\end{rootSystem}
\end{tikzpicture}
\]
\end{example}
\begin{example}
There might be no compact roots: consider \(A_2=\PSL{3}\) acting on the space \(X\) of all pointed lines in projective space:
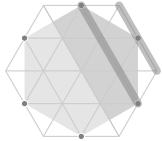
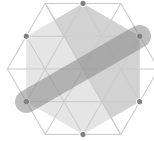
\[
\drawroots{A}{3}
\]
with associated cominuscule
\[
\begin{tikzpicture}[baseline=-.5]
\begin{rootSystem}{A}
\roots
\parabolic{3}
\draw[/root system/grading,line width=.3cm] (hex cs:x=1,y=1) -- (hex cs:x=-1,y=-1);
\end{rootSystem}
\end{tikzpicture}
\]
\end{example}
\begin{example}
The \(P\)-compact roots and the cominuscule roots might not be perpendicular.
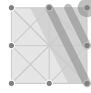
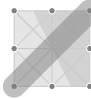
\[
\drawroots{B}{3}
\]
with associated cominuscule
\[
\begin{tikzpicture}[baseline=-.5]
\begin{rootSystem}{B}
\roots
\parabolic{3}
\draw[/root system/grading,line width=.3cm] (square cs:x=1,y=1) -- (square cs:x=-1,y=-1);
\end{rootSystem}
\end{tikzpicture}
\]
\end{example}
As for coroots, \(\LieG_0\) contains the Cartan subalgebra, which is precisely the span of the \(\Delta\)-coroots.
\begin{lemma}
\(\LieP'=\breve\LieP+\LieG_0\) and \(\breve\LieP\cap\LieG_0=\breve\LieH\) is the Cartan subalgebra of \(\breve\LieG\).
\end{lemma}
\begin{proof}
Clearly \(\LieP'=\LieG_0\oplus\rtspMax{\LieG}\subseteq\LieG_0+\breve\LieP\).
But \(\breve\LieP\) is the span of the \(P\)-maximal root vectors and the \(\breve\Delta\)-Cartan subalgebra, so lies in \(\LieP'\).
\end{proof}
We draw the gradings of the positive roots of the parabolic subgroups of the rank \(2\) simple groups.
Under the heading  \(\breve\LieG\), we draw the roots of the symmetry Lie algebra of the associated cominuscule, and under the heading \(\LieG'\) the roots of the automorphism Lie algebra.
\begin{longtabl}{@{}>{$\mathrlap}c<{$}>{$}c<{$}ccc@{}>{$}c<{$}}
\toprule
& 
&
\text{Grading} & 
$\breve\LieG$ & 
$\LieG'$&
\text{Tally}\\ 
\midrule
\endfirsthead
\multicolumn{3}{c}{continued \dots}\\
\toprule
& 
&
\text{Grading} & 
$\breve\LieG$ & 
$\LieG'$&
\text{Tally}\\ 
\midrule
\endhead
\multicolumn{3}{c}{continued \dots}\\
\endfoot
\bottomrule
\endlastfoot
A_2 & 
\dynkin A{*x} &
\drawroots{A}{1} 
& 
\begin{tikzpicture}[baseline=-.5]
\begin{rootSystem}{A}
\roots
\parabolic{1}
\draw[/root system/grading] (hex cs:x=-1,y=2) -- (hex cs:x=1,y=1) -- (hex cs:x=2,y=-1) -- (hex cs:x=1,y=-2) -- (hex cs:x=-1,y=-1) -- (hex cs:x=-2,y=1) -- cycle;
\simpleroots
\end{rootSystem}
\end{tikzpicture}
& 
\begin{tikzpicture}[baseline=-.5]
\begin{rootSystem}{A}
\roots
\parabolic{1}
\draw[/root system/grading] (hex cs:x=-1,y=2) -- (hex cs:x=1,y=1) -- (hex cs:x=2,y=-1) -- (hex cs:x=1,y=-2) -- (hex cs:x=-1,y=-1) -- (hex cs:x=-2,y=1) -- cycle;
\simpleroots
\end{rootSystem}
\end{tikzpicture}
&2
 \\ \midrule
A_2 & 
\dynkin A{xx} &
\drawroots{A}{3} 
& 
\begin{tikzpicture}[baseline=-.5]
\begin{rootSystem}{A}
\roots
\parabolic{3}
\draw[/root system/grading,line width=.3cm] (hex cs:x=1,y=1) -- (hex cs:x=-1,y=-1);
\simpleroots
\end{rootSystem}
\end{tikzpicture}
& 
\begin{tikzpicture}[baseline=-.5]
\begin{rootSystem}{A}
\roots
\parabolic{3}
\draw[/root system/grading,line width=.3cm] (hex cs:x=1,y=1) -- (hex cs:x=-1,y=-1);
\simpleroots
\end{rootSystem}
\end{tikzpicture}
&
1
 \\ \midrule
B_2 & 
\dynkin B{x*} &
\drawroots{B}{1} 
& 
\begin{tikzpicture}[baseline=-.5]
\begin{rootSystem}{B}
\roots
\parabolic{1}
\draw[/root system/grading] (square cs:x=1,y=1) -- (square cs:x=-1,y=1)-- (square cs:x=-1,y=-1)-- (square cs:x=1,y=-1)--cycle;
\simpleroots
\end{rootSystem}
\end{tikzpicture}
& 
\begin{tikzpicture}[baseline=-.5]
\begin{rootSystem}{B}
\roots
\parabolic{1}
\draw[/root system/grading] (square cs:x=1,y=1) -- (square cs:x=-1,y=1)-- (square cs:x=-1,y=-1)-- (square cs:x=1,y=-1)--cycle;
\simpleroots
\end{rootSystem}
\end{tikzpicture}
&
3
 \\ \midrule
B_2 & 
\dynkin B{*x} &
\drawroots{B}{2} 
& 
\begin{tikzpicture}[baseline=-.5]
\begin{rootSystem}{B}
\roots
\parabolic{2}
\draw[/root system/grading,line width=.3cm] (square cs:x=1,y=1) -- (square cs:x=-1,y=-1);
\simpleroots
\end{rootSystem}
\end{tikzpicture}
& 
\begin{tikzpicture}[baseline=-.5]
\begin{rootSystem}{B}
\roots
\parabolic{3}
\draw[/root system/grading,line width=.3cm] (square cs:x=1,y=1) -- (square cs:x=-1,y=-1);
\draw[/root system/grading,line width=.3cm] (square cs:x=-1,y=1) -- (square cs:x=1,y=-1);
\simpleroots
\end{rootSystem}
\end{tikzpicture}
&
1
 \\ \midrule
B_2 & 
\dynkin B{xx} &
\drawroots{B}{3} 
& 
\begin{tikzpicture}[baseline=-.5]
\begin{rootSystem}{B}
\roots
\parabolic{3}
\draw[/root system/grading,line width=.3cm] (square cs:x=1,y=1) -- (square cs:x=-1,y=-1);
\simpleroots
\end{rootSystem}
\end{tikzpicture}
& 
\begin{tikzpicture}[baseline=-.5]
\begin{rootSystem}{B}
\roots
\parabolic{3}
\draw[/root system/grading,line width=.3cm] (square cs:x=1,y=1) -- (square cs:x=-1,y=-1);
\draw[/root system/grading,line width=.3cm] (square cs:x=1,y=-1) -- (square cs:x=0,y=0);
\simpleroots
\end{rootSystem}
\end{tikzpicture}
 & 1
 \\ \midrule
G_2 & 
\dynkin G{x*} &
\drawroots{G}{1} 
& 
\begin{tikzpicture}[baseline=-.5]
\begin{rootSystem}{G}
\roots
\parabolic{1}
\draw[/root system/grading,line width=.3cm] (hex cs:x=1,y=1) -- (hex cs:x=-1,y=-1);
\simpleroots
\end{rootSystem}
\end{tikzpicture}
& 
\begin{tikzpicture}[baseline=-.5]
\begin{rootSystem}{G}
\roots
\parabolic{1}
\draw[/root system/grading,line width=.3cm] (hex cs:x=1,y=1) -- (hex cs:x=-1,y=-1);
\draw[/root system/grading,line width=.3cm] (hex cs:x=-1,y=1) -- (hex cs:x=1,y=-1);
\simpleroots
\end{rootSystem}
\end{tikzpicture}
& 1
 \\ \midrule
G_2 & 
\dynkin G{*x} &
\drawroots{G}{2} 
& 
\begin{tikzpicture}[baseline=-.5]
\begin{rootSystem}{G}
\roots
\parabolic{2}
\draw[/root system/grading] (hex cs:x=-1,y=2) -- (hex cs:x=1,y=1) -- (hex cs:x=2,y=-1) -- (hex cs:x=1,y=-2) -- (hex cs:x=-1,y=-1) -- (hex cs:x=-2,y=1) -- cycle;
\simpleroots
\end{rootSystem}
\end{tikzpicture}
& 
\begin{tikzpicture}[baseline=-.5]
\begin{rootSystem}{G}
\roots
\parabolic{2}
\draw[/root system/grading] (hex cs:x=-1,y=2) -- (hex cs:x=1,y=1) -- (hex cs:x=2,y=-1) -- (hex cs:x=1,y=-2) -- (hex cs:x=-1,y=-1) -- (hex cs:x=-2,y=1) -- cycle;
\simpleroots
\end{rootSystem}
\end{tikzpicture}
&
1
 \\ \midrule
G_2 & 
\dynkin G{xx} &
\drawroots{G}{3} 
& 
\begin{tikzpicture}[baseline=-.5]
\begin{rootSystem}{G}
\roots
\parabolic{3}
\draw[/root system/grading,line width=.3cm] (hex cs:x=1,y=1) -- (hex cs:x=-1,y=-1);
\simpleroots
\end{rootSystem}
\end{tikzpicture}
& 
\begin{tikzpicture}[baseline=-.5]
\begin{rootSystem}{G}
\roots
\parabolic{3}
\draw[/root system/grading,line width=.3cm] (hex cs:x=1,y=1) -- (hex cs:x=-1,y=-1);
\draw[/root system/grading,line width=.3cm] (hex cs:x=0,y=0) -- (hex cs:x=1,y=-1);
\simpleroots
\end{rootSystem}
\end{tikzpicture}
&1
\end{longtabl}
Any \(P\)-character restricts to an \(P'\)-character; these have the same Cartan subgroups so the same characters.

\subsection{Chevalley bases}\label{subsubsection:ChevalleyBases}
Pick a complex semisimple Lie group \(G\) and a Cartan subgroup \(H \subset G\).
Denote the Killing form on \(\LieH^*\) by \(\alpha, \beta \mapsto \KillingForm{\alpha}{\beta}\).
A \emph{Chevalley basis} \(\XX{\alpha} \in \rtsp{\LieG}{\alpha}, \HH{\alpha} \in \LieH\) is a spanning set of \(\LieG\) parameterized by roots \(\alpha \in \LieH^*\) so that
\begin{enumerate}
\item \([\HH{\beta}\XX{\alpha}]=2 \frac{\KillingForm{\alpha}{\beta}}{\KillingSquare{\beta}} \XX{\alpha}\)
\item \([\HH{\alpha}\HH{\beta}]=0\),
\item
\[
\lb{\XX{\alpha}}{\XX{\beta}}=
\begin{cases}
\HH{\alpha},&\text{if } \alpha+\beta=0, \\
N_{\alpha\beta} \XX{\alpha+\beta},&\text{if \(\alpha+\beta\) a root}, \\
0,&\text{otherwise} \\
\end{cases}
\]
with
\begin{enumerate}
\item \(N_{\alpha\beta}=\pm (p+1)\), where \(p\) is the largest integer for which \(\beta-p \, \alpha\) is a root,
\item \(N_{-\alpha,-\beta}=-N_{\alpha\beta}\),
\item we set \(N_{\alpha \beta}=0\) if \(\alpha+\beta = 0\) or \(\alpha+\beta\) is not a root.
\end{enumerate}
\end{enumerate}
It follows then that
\[
\aa(\check{\bb})=2\frac{\KillingForm{\aa}{\bb}}{\KillingSquare{\bb}}.
\]
Every complex semisimple Lie group \(G\) with a Cartan subgroup \(H \subset G\) admits a Chevalley basis \cite{Serre:2001} p. 51.
A Chevalley basis in this sense is only a spanning set; if we pick out a basis of simple roots \(\alpha_i\) then 
\(\set{\check{\alpha}_i}_{i=1}^{r}\sqcup\set{e_{\aa}}_{\alpha\in\Delta}\) is a basis of \(\LieH^*\).
The \emph{Cartan integers} are
\[
2\frac{\KillingForm{\alpha_i}{\alpha_j}}{\KillingSquare{\alpha_j}}.
\]


\begin{lemma}\label{lemma:Jacobi.mystery}
For any roots \(\aa, \bb\) of any complex semisimple Lie group \(G\), in any Chevalley basis,
\[
N_{\bb,-\aa}N_{\aa,\bb-\aa}
+
N_{\aa\bb}N_{-\aa,\aa+\bb}
=
\begin{cases}
0, & \text{if \(\aa=\pm \bb\),} \\
-2\frac{\KillingForm{\bb}{\aa}}{\KillingSquare{\aa}}, & \text{otherwise.}
\end{cases}
\]
\end{lemma}
\begin{proof}
If \(\aa \ne \pm \bb\), the Jacobi identity gives
\begin{align*}
0
&=
\lb{\XX{\bb}}{\lb{\XX{-\aa}}{\XX{\aa}}}
+
\lb{\XX{\aa}}{\lb{\XX{\bb}}{\XX{-\aa}}}
+
\lb{\XX{-\aa}}{\lb{\XX{\aa}}{\XX{\bb}}},
\\
&=
-\lb{\XX{\bb}}{\HH{\aa}}
+
N_{\bb,-\aa}\lb{\XX{\aa}}{\XX{\bb-\aa}}
+
N_{\aa\bb}\lb{\XX{-\aa}}{\XX{\aa+\bb}},
\\
&=
\lb{\HH{\aa}}{\XX{\bb}}
+
N_{\bb,-\aa}N_{\aa,\bb-\aa}\XX{\bb}
+
N_{\aa\bb}N_{-\aa,\aa+\bb}\XX{\bb},
\\
&=
2\frac{\KillingForm{\bb}{\aa}}{\KillingSquare{\aa}}\XX{\bb}
+
N_{\bb,-\aa}N_{\aa,\bb-\aa}\XX{\bb}
+
N_{\aa\bb}N_{-\aa,\aa+\bb}\XX{\bb},
\\
&=
\pr{
2\frac{\KillingForm{\bb}{\aa}}{\KillingSquare{\aa}}
+
N_{\bb,-\aa}N_{\aa,\bb-\aa}
+
N_{\aa\bb}N_{-\aa,\aa+\bb}
}
\XX{\bb}.
\end{align*}
\end{proof}
\subsection{Dual bases}
The dual basis is defined by
\begin{align*}
\aa(\HH{\bb})&=2\frac{\KillingForm{\aa}{\bb}}{\KillingSquare{\bb}},\\
\alpha(\XX{\bb})&=0,\\
\omega^{\aa}(\HH{\bb})&=0,\\
\omega^{\aa}(\XX{\bb})&=\delta^{\aa}_{\bb}.
\end{align*}
These extend uniquely to left invariant \(1\)-forms on \(G\), to which we apply
\[
d\omega(X,Y)=\LieDer_X(\omega(Y))-\LieDer_Y(\omega(X))-\omega(\lb{X}{Y}),
\]
to left invariant \(1\)-forms \(\omega\) and vector fields \(X,Y\).
The first two terms vanish by left invariance:
\[
d\omega(X,Y)=-\omega(\lb{X}{Y}).
\]
Let
\(
\om{\beta \gamma} := \om{\beta} \wedge \om{\gamma},
\)
etc. and compute the structure equations of semisimple Lie groups:
\begin{align*}
d \om{\aa} &= - \aa \wedge \om{\aa} - \frac{1}{2}
\sum_{\beta + \gamma=\alpha} N_{\beta \gamma} \om{\beta\gamma}
\\
d \alpha &=
-
\sum_{\beta} \frac{\KillingForm{\alpha}{\beta}}
{\KillingSquare{\beta}} \om{\beta,-\beta}
\end{align*}
with sums over all roots.
\subsection{Characters and modules}
Every character of \(P\) is a weight of \(\LieG\), since it restricts to a character of the Cartan subalgebra of \(\LieG\).
Every character of \(P\) is invariant under \(\GZ\).
A character of \(P\) is \emph{dominant} if it is a dominant weight of \(\LieG\).
If \(V\) is a holomorphic \(P\)-module, the \emph{character} \(\chi_V \in \LieH^*\) of \(V\) is the character of \(\Lmtop{V}\).
\begin{example}
\[
\chi_{\ge k} :=
\chi_{\LieG_{\ge k}}
=
\sum_{\Pheight{\bb}\ge k} \beta.
\]
\end{example}
\begin{example}
For the flag variety \((X,G)=(X^6,G_2)=(G_2/B,G_2)=\dynkin G{xx}\), the grading is
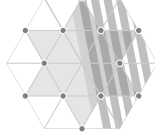
\[
\drawroots{G}{3} 
\]
The positive roots are
{
\pgfkeys{/root system/weight length=1.5cm}
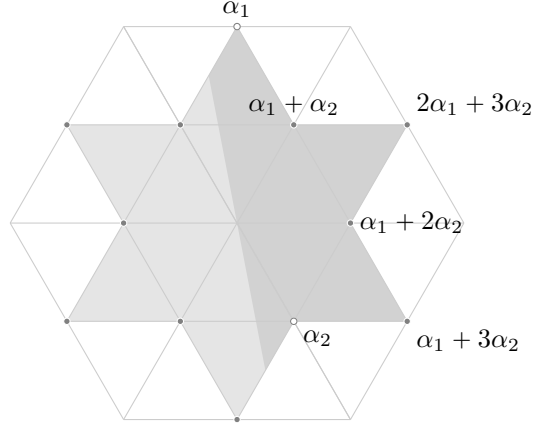
\[
\begin{tikzpicture}
\begin{rootSystem}{G}
\roots
\parabolic{3}
\simpleroots
\node [above] at \Root {1}{0} {\(\alpha_1\)};
\node [below right] at \Root {0}{1} {\(\alpha_2\)};
\node [above] at \Root {1}{1} {\(\alpha_1+\alpha_2\)};
\node [right] at \Root {1}{2} {\(\alpha_1+2\alpha_2\)};
\node [below right] at \Root {1}{3} {\(\alpha_1+3\alpha_2\)};
\node [above right] at \Root {2}{3} {\(2\alpha_1+3\alpha_2\)};
\end{rootSystem}
\end{tikzpicture}
\]
\pgfkeys{/root system/weight length=.5cm}
}
So
\begin{align*}
\chi_{\ge 5}&=2\alpha_1+3\alpha_2,\\
\chi_{\ge 4}&=3\alpha_1+6\alpha_2,\\
\chi_{\ge 3}&=4\alpha_1+8\alpha_2,\\
\chi_{\ge 2}&=5\alpha_1+9\alpha_2,\\
\chi_{\ge 1}&=6\alpha_1+10\alpha_2=2\delta.
\end{align*}
\end{example}
\begin{lemma}[Knapp \cite{Knapp:2002} p. 330 Proposition 5.99.]\label{lemma:Knapp.dominant}
For any positive integer \(p \ge 0\), the character of any \(P\)-submodule of \(\LieG_{+}^{\otimes p}\) is dominant.
\end{lemma}
For any \(P\)-module \(V\) and weight \(\lambda \in \LieH^*\), denote by \(V_{\lambda}\) the \(\lambda\)-weight space of \(V\).
For any set \(\Gamma\) of weights of a \(P\)-module \(V\), let 
\[
V_{\Gamma}=\bigoplus_{\lambda \in \Gamma} V_{\lambda}.
\]
A set \(\Gamma\) of weights of \(V\) is \emph{saturated} if for \(\lambda\in\Gamma\) and \(\aa\) a root whose root vector lies in \(\LieP\), if \(\lambda+\aa\) is also a weight of \(V\), and if \(e_{\aa} \colon V_{\lambda} \to V_{\lambda+\aa}\) is not zero, then \(\lambda+\aa \in \Gamma\).
So \(\Gamma\) is saturated just when \(V_{\Gamma} \subset V\) is a \(P\)-submodule.
In particular, if \(\dimC{V_{\lambda}}=1\) or \(0\) for all \(\lambda\), then 
\[
\chi_{\Gamma} := \chi_{V_{\Gamma}} = \sum_{\lambda \in \Gamma} \lambda.
\]
\begin{lemma}\label{lemma:quotient.modules}
Take a complex semisimple Lie group \(G\) and a parabolic subgroup \(P \subset G\) with Lie algebras \(\LieP \subset \LieG\).
There are quotient holomorphic \(P\)-modules \(\LieGN \to Q_i \to 0\) whose characters \(\Q{\ge 0}\)-span the dominant \(P\)-characters.
\end{lemma}
\begin{proof}
Expand every root \(\beta\) into a sum of simple roots \(\beta=\sum_i n_i\of{\beta} \alpha_i\).
Note that \(\Pheight{\alpha_i}=0\) (compact) or \(1\) (noncompact).
If \(\alpha_j\) is noncompact, let \(\Gamma_j \subset \ncptPosRts\) be the set of 
roots \(\bb > 0\) with either \(n_j\of{\beta}\ge 2\) or \(n_i\of{\beta}>0\) for some simple noncompact root \(\alpha_i \ne \alpha_j\).
This set \(\Gamma_j\) is saturated.
The associated holomorphic \(P\)-submodule \(V_j := \rtsp{\LieG}{\Gamma_j}\)  has quotient \(Q_j:= \LieGN/V_j\) the sum of root spaces of roots \(\alpha_j+\beta\) for \(\beta\) zero or compact positive; let \(\chi_j:= \chi_{Q_j}\).
Expand each character as a sum of weights, and use the positive simple roots as a \(\Q{}\)-basis of the weights: \(\chi_j=q_j \alpha_j \pmod{\cptRts}\) where \(q_j=\dimC{Q_j}\).
For any character \(\chi = \sum p_i \alpha_i\), if \(\alpha_1, \alpha_2, \dots, \alpha_p\) are the noncompact positive simple roots, then \(\chi - \sum_{i=1}^p \frac{p_i}{q_i} \chi_i\) is a character of \(P\), a sum of compact roots, so a weight of the maximal semisimple subalgebra of \(\LieGZ \subset \LieP\), invariant under reflections in all compact roots, so vanishes.
\end{proof}

We use this to relate characters of submodules of \(\LieGN\) to those of quotient modules of \(\LieGN\).
For any positive integer \(p \ge 0\), every holomorphic \(P\)-submodule of \(\LieGN^{\otimes p}\) has character a positive rational linear combination of characters of quotients of \(P\)-submodules of \(\LieGN\).

\begin{lemma}\label{lemma:not.perp.roots}
For any flag variety, let \(\chi_{\ge j}\) be the sum of the roots of grade \(j\) or more, \(j=1,2,\dots\).
Then \(\chi_{\ge 1}\) is perpendicular precisely to the roots of grade zero (the compact roots).
Moreover \(\chi_{\ge 2}\) is perpendicular precisely to the roots \(\beta\) for which either \(\beta\) has grade zero or \(\beta\) is a root of a cominuscule factor.
Finally, for all \(j\), \(\chi_{\ge j+1}\) has positive inner product with any root of grade \(j\) or more.
\end{lemma}
\begin{proof}
Roots from different irreducible factors are perpendicular, so it suffices to prove the result for an irreducible flag variety.
For \(\chi_{\ge 1}\), the result is corollary~5.100, p. 330 \cite{Knapp:2002}.
Every \(\chi_{\ge j}\) is perpendicular to all compact roots, i.e. grade zero, because the roots of the parabolic subgroup are invariant under reflections in the compact roots.

A flag variety is cominuscule just when all roots have grade at most \(1\), i.e. just when \(\chi_{\ge 2}=0\), for which the result is obvious.
So we can assume that there are roots of grade \(2\).

Next consider the case of \(\chi=\chi_{\ge j+1}\) for some \(j \ge 1\).
Suppose that \(\bb>0\) and that \(\KillingForm{\bb}{\chi}<0\).
Then \(\KillingForm{\bb}{\aa}<0\) for some \(\aa\) of grade \(\ge j+1\).
When two roots have negative inner product, they lie on a root string
\[
\aa,\aa+\bb,\dots,r_{\bb}\aa
\]
going up to the reflection \(r_\bb\) in \(\bb\) p. 
\cite{Knapp:2002} p. 144 proposition 2.29, \cite{Serre:2001} p. 29 proposition 3.
So \(\aa,\aa+\bb,\dots,r_{\bb}\aa\) is a string of roots of increasing grades, and 
\[
\KillingForm{\bb}{r_{\bb}\aa}=-\KillingForm{\bb}{\aa}>0
\]
cancels with \(\KillingForm{\bb}{\aa}\) in the expansion of \(\KillingForm{\bb}{\chi}\).
So all negative contributions \(\KillingForm{\bb}{\aa}<0\) to \(\KillingForm{\bb}{\chi}\) cancel with various positive contributions.
Therefore \(\KillingForm{\bb}{\chi}\ge 0\) for any root \(\bb > 0\).
Moreover \(\KillingForm{\bb}{\chi}=0\) with \(\bb>0\) just when every positive contribution cancels out too in the process of cancelling negatives.
In other words, \(\KillingForm{\bb}{\chi}=0\) with \(\bb>0\) just when the set of roots of grade \(\ge j+1\) is closed under taking \(\pm\bb\)-root strings.

Assume that \(\KillingForm{\bb}{\chi}=0\) and \(\Pheight{\bb}=j\).
Pick a root \(\cc\) with \(\KillingForm{\bb}{\cc}<0\) and with grade \(1\le\Pheight{\cc}\le j\).
Then the \(\bb\)-root string of \(\cc\) crosses into grade \(j+1\).
But then the set of roots of grade \(\ge j+1\) is not closed under taking \(\pm\bb\)-root strings, a contradiction.
So \(\KillingForm{\bb}{\cc}\ge 0\) if \(1\le\Pheight{\cc}\le j\).
All roots of any positive grade are sums of roots \(\cc\) with \(1\le\Pheight{\cc}\le j\), so \(\KillingForm{\bb}{\cc}\ge 0\) for \(\Pheight{\cc}>0\).
But then
\[
\KillingForm{\bb}{r_{\bb}\cc}
=
-\KillingForm{\bb}{\cc}\le 0.
\]
But \(r_{\bb}\cc=\cc+\bb+\dots+\bb\) so
\[
\KillingForm{\bb}{r_{\bb}\cc}\ge 0.
\]
So \(\KillingForm{\bb}{r_{\bb}\cc}=0\), and so \(\KillingForm{\bb}{\cc}=0\).
In particular \(\KillingForm{\bb}{\bb}=0\), i.e. \(\bb=0\), a contradiction.
So if \(\Pheight{\bb}\ge j\) then \(\KillingForm{\bb}{\chi}>0\).
\end{proof}
\subsection{Characters and the associated cominuscule}
Take a flag variety \(X=G/P\) and its associated cominuscule subvariety \(\breve{X}=\breve{G}/\breve{P}\subseteq X\).
Write \(X\) as a product of irreducible flag varieties
\[
X=\prod_j X_j, G=\prod_j G_j.
\]
Let
\[
\chi_j=\sum_{\alpha}\alpha
\]
be the character which is the sum of all \(P\)-maximal roots in simple factor number \(j\).
\begin{lemma}\label{lemma:chi.j}
For every character \(\chi\) of \(P\), there are unique rational numbers \(a_j\) so that 
\[
\chi':=\chi-\sum a_j\chi_j
\]
is a character perpendicular to all \(P\)-compact, \(P\)-maximal and \(P\)-minimal roots.
\end{lemma}
\begin{proof}
Clearly we can assume that \((X,G)\) is irreducible.
Let
\[
\chi''=\sum_{\alpha}\alpha
\]
be the character which is the sum of all \(P\)-maximal roots.
The characters of \(P\) are trivial on the semisimple Levi factor of \(P\), and on the unipotent radical, hence reduce to the center of the reductive Levi factor, i.e. to the quotient of the Cartan subalgebra of \(P\) by the Cartan subalgebra of the \(P\)-compact roots, i.e. of \(G_0\).
So these characters belong to the weight lattice of the Cartan subgroup of \(P\).
For any compact root \(\beta\),
\[
0=\chi([e_\beta,e_{-\beta}])=\chi(\check\bb).
\]
But then
\[
0=\chi(\check\bb)=2\frac{\KillingForm{\chi}{\bb}}{\KillingSquare{\bb}},
\]
so \(\chi\) is perpendicular to all \(P\)-compact roots, hence invariant under reflection in all \(P\)-compact roots.

The \(P\)-maximal roots form an irreducible \(G_0\)-module \cite{Knapp:2002} p. 332 Theorem 5.104.
Hence each \(P\)-maximal root \(\beta\) is a sum of the unique \(\breve{P}\)-positive \(\breve{G}\)-simple \(\breve\Delta\)-root and various \(P\)-compact roots. 
Since those compact roots are perpendicular to \(\chi\), every \(P\)-maximal root \(\beta\) is perpendicular to \(\chi\) just when any one of them is, or when the sum \(\chi''\) is.
A character \(\chi\) of \(P\) is perpendicular to all \(P\)-maximal roots \(\beta\) just when it is perpendicular to \(\chi''\).
We can subtract a suitable multiple of \(\chi''\) from \(\chi\) to arrange that 
\[
\chi':=\chi-a\chi''
\]
is perpendicular to \(\chi''\), for unique rational number \(a\), as the inner product is rational (since weights lie in the rational linear combinations of the roots).
\end{proof}

\section{Pseudoeffective line bundles}\label{section:pseudoeffective}
Take a holomorphic line bundle \(L \to M\) on a complex manifold \(M\), and a Hermitian metric on \(L\), with norm denoted 
\(h\colon L \to \R{}\).
On the associated principal \(\C[\times]\)-bundle \(L^{\times} \to M\) (the complement of the zero section), let \(\phi:=-2\log h\); \(\phi \colon L^{\times} \to \R{}\) is the \emph{weight} of the Hermitian metric.
Conversely pick any locally integrable function \(\phi \colon L^{\times} \to \R{}\) so that \(\phi(as)=-\log |a| + \phi(s)\) for any \(a \in \C[\times]\) and \(s \in L^{\times}\) and let \(h(s):= e^{-\phi(s)/2}\) and call \(h\) a \emph{singular Hermitian metric}.
The \emph{curvature} of \(h\) is the basic current \(\sqrt{-1} \partial \bar\partial \phi\).
A line bundle \(L\) is \emph{pseudoeffective} if it admits a singular Hermitian metric with curvature nonnegative in the sense of currents \cite{Demailly:2014} p. 1.
A compact K\"ahler manifold is uniruled if and only if its canonical bundle is not pseudoeffective \cite{ou2025characterizationuniruledcompactkahler}.
\begin{theorem}
A holomorphic Cartan geometry on a compact K\"ahler manifold is minimal if and only if the canonical bundle of the manifold is pseudoeffective.
\end{theorem}
\begin{proof}
This is immediate from Biswas \& McKay \cite{Biswas.McKay:2016} theorem 2 p. 2.
\end{proof}
\begin{lemma}\label{lemma:pseudoeffective.trivial}
The dual of a pseudoeffective line bundle on a smooth complex projective variety is either trivial or has no nonzero holomorphic sections.
The Chern class of a pseudoeffective line bundle is nonnegative.
\end{lemma}
\begin{proof}
The Chern class of any pseudoeffective line bundle \(L\) is the limit of a sequence of Chern classes of line bundles of effective divisors \(D_i\) (\cite{Boucksom/Demailly/Paun/Peternell:2013} p. 3 proposition 1.2).
If \(L^* \to M\) is the dual of a pseudoeffective line bundle, then \(\int_C c_1(L^*)=-\int_C c_1(L) \le 0\) on any curve \(C \subset M\).
If \(L_i \to M\) is the line bundle of a nonzero effective divisor \(D_i\), since \(M\) is projective, we can find a curve (intersection with a suitable linear subspace in projective space) which intersects that effective divisor \(C\cdot D_i\) times counting multiplicity, so \(\int_C c_1(L_i)=C\cdot D_i>0\) is a positive integer.
If \(L^*\to M\) has a holomorphic section, vanishing on some hypersurface \(D\), then \(\int_C c_1(L^*)=C\cdot D \ge 0\) for any curve \(C\).
But \(C\cdot D=-C\cdot D_i\le 0\), so \(C\cdot D=0\) for all curves \(C\), so \(D=0\), so \(L\) is trivial.
\end{proof}
\begin{theorem}[Campana and Peternell \cite{Campana/Peternell:2011} theorem 1.9]%
\label{theorem:Campana.Peternell}
Suppose that \(M\) is a smooth complex projective variety, not uniruled.
Suppose that \(\pr{\nForms{1}{M}}^{\otimes m} \to \vb{Q}\) is a torsion-free quotient sheaf, some \(m \ge 1\).
Then \(\det \vb{Q}\) is pseudoeffective.
\end{theorem}

\section{Brackets of vector fields}\label{section:brackets}
Suppose that \(V \to M\) is a holomorphic vector bundle on a complex manifold with a Hermitian metric.
If \(s,t\) are \(V\)-valued functions and \(\xi,\eta\) are differential forms, say a \((p,0)\)-form and a \((q,0)\)-form, define a \(\pr{p,q}\)-form by
\(
\sesq{s \otimes \xi}{t \otimes \eta}
=
\frac{\sqrt{-1}^{p^2}}{2^p}
\left<s,t\right> \xi \wedge \bar{\eta}.
\)
Extend by \(\C\)-linearity and \(\bar{\mathbb{C}}\)-linearity to define \(\sesq{\lambda}{\mu}\) for any \(\lambda, \mu \in \cohomology{0}{M,V \otimes \Omega^{*,0}}\).
If additionally \(T\) is a section of \(\operatorname{End} V\) and \(\zeta\) is a \((k,k)\)-form, define
\[
\isesq{s \otimes \xi}{T \otimes \zeta}{t \otimes \eta}
=
\frac{\sqrt{-1}^{p^2}}{2^p}
\left<Ts,t\right> \xi \wedge \zeta \wedge \eta.
\]
Extend by \(\C\)-linearity and \(\bar{\mathbb{C}}\)-linearity to define \(\isesq{\lambda}{\tau}{\mu}\) for any
\[
\lambda, \mu \in \cohomology{0}{M,V \otimes \Omega^{\bullet,0}}, \tau \in \cohomology{0}{M,\Omega^{k,k} \otimes \operatorname{End} V}.
\]
If \(\Omega\) is a K\"ahler form, clearly
\(
\sesq{\lambda}{\lambda}\wedge\Omega^{n-p}\ge 0
\)
for any \(\lambda\), with equality just when \(\lambda=0\).

A subsheaf \(\vb{V} \subset TM\) is \emph{bracket closed} if the Lie bracket of any two local sections is a local section.
\begin{theorem}[Demailly \cite{Demailly:2002}]%
\label{theorem:Demailly.bracket}
Suppose that \(L \to M\) is a holomorphic line bundle on a compact \Kaehler manifold \(M\) with \Kaehler form \(\Omega\), and \(L\) has a singular Hermitian metric with Chern connection \(\nabla\).
Suppose that \(\lambda\) is a holomorphic differential form valued in \(L\).
Let \(\vb{V} \subset TM\) be the subsheaf of holomorphic tangent vector fields \(X\) on \(M\) so that \(X \hook \lambda=0\).
Then
\[
\int_M
\isesq{\lambda}{\frac{\sqrt{-1}}{2 \pi}\nabla^2_L}{\lambda}
\wedge e^{\Omega}
=
\frac{1}{\pi}
\int
\sesq{\nabla \lambda}{\nabla \lambda} \wedge e^{\Omega} \ge 0.
\]
Equality occurs just when \(\vb{V}\) is bracket closed and \(\lambda\) is parallel and the curvature of \(L\) satisfies \(\nabla^2_L \wedge \lambda=0\). 
\end{theorem}
Clearly \(\vb{V}\) is a coherent sheaf of \(\OO\)-modules, and a holomorphic foliation by the Frobenius theorem.
We give only a few details of Demailly's proof, but written without his use of local trivializations, which we hope clarifies his proof.
\begin{proof}
A section \(s\) of \(L\) is clearly identified with a function \(L^{\times}\xrightarrow{f}\C\) so that
\[
f(at)=a^{-1}f(t),
\]
by asking that, if \(t\in L^{\times}_m\) for some \(m\in M\), then \(f(t)=s(m)/t\).
Conversely, every function \(L^{\times}\xrightarrow{f}\C\) so that
\[
f(at)=a^{-1}f(t)
\]
arises uniquely from a section: define \(s(t):=f(t)t\), and see that \(s(at)=s(t)\), so that \(s\) descends to a section of \(L\).

For \(a\in\C[\times]\), denote by \(\ell_a\) the scaling operation
\[
\ell_a(t)=at,
\]
for \(t\in L^{\times}\).
Let \(v\) be the infinitesimal generator of this action
\[
v(t):=\left.\frac{d}{da}\right|_{a=1}\ell_a(t).
\]
A \((1,0)\)-connection on \(L\) is precisely a \((1,0)\)-form \(\gamma\) on \(L^{\times}\) so that 
\[
v\hook\gamma=1.
\]
For any singular Hermitian metric with weight \(\phi\), 
\[
v\hook d\phi=-1.
\]
So
\[
\gamma=-\partial\phi
\]
is a \((1,0)\)-connection, the Chern connection of the weight \(\phi\).

Any \(L\)-valued \((p,0)\)-form \(\lambda\) is identified in the same way with a semibasic \((p,0)\)-form on \(L^{\times}\), which we also denote by \(\lambda\), so that
\[
\ell_a^*\lambda=a^{-1}\lambda.
\]
The connection, expressed in this ``upstairs formalism'', becomes
\[
\nabla \lambda=d\lambda+\gamma\wedge\lambda.
\]
Compute
\begin{align*}
\nabla\Braket{\nabla\lambda|\lambda}
&=
\frac{\sqrt{-1}^{(p+1)^2}}{2^{p+1}}
\nabla\left(e^{-\phi/2}\nabla\lambda\wedge\bar\lambda\right),
\\
&=
\frac{\sqrt{-1}^{(p+1)^2}}{2^{p+1}}
e^{-\phi/2}\left(\nabla^2\lambda\wedge\bar\lambda
+(-1)^{p+1}\nabla\lambda\wedge\nabla\bar\lambda\right).
\end{align*}
The second term is
\[
(-1)^{p+1}\Braket{\nabla\lambda|\nabla\lambda}.
\]
The first is
\begin{align*}
\frac{\sqrt{-1}^{(p+1)^2}}{2^{p+1}}
e^{-\phi/2}\nabla^2\wedge\lambda\wedge\bar\lambda
&=
\frac{\sqrt{-1}^{p^2+2p+1}}{2^{p+1}}
e^{-\phi/2}\nabla^2\wedge\lambda\wedge\bar\lambda,
\\
&=
-(-1)^{p+1}
\frac{\sqrt{-1}^{p^2+1}}{2^{p+1}}
e^{-\phi/2}\nabla^2\wedge\lambda\wedge\bar\lambda,
\\
&=
-(-1)^{p+1}
\left(\frac{\sqrt{-1}}{2}\nabla^2\right)
\wedge
\frac{\sqrt{-1}^{p^2}}{2^p}
e^{-\phi/2}\lambda\wedge\bar\lambda,
\\
&=
-(-1)^{p+1}
\Braket{\lambda|\frac{\sqrt{-1}}{2}\nabla^2|\lambda}
\end{align*}
Put back together,
\[
(-1)^{p+1}\nabla\Braket{\nabla\lambda|\lambda}
=
\Braket{\nabla\lambda|\nabla\lambda}
-
\Braket{\lambda|\frac{\sqrt{-1}}{2}\nabla^2|\lambda}.
\]
If \(\phi\) is smooth enough, then by Stokes's theorem, our integrals agree.
Demailly's original proof \cite{Demailly:2002} explains that we can always approximate \(\phi\) by something smooth with suitable bounds. 
We prefer not to repeat the remainder of his argument.
\end{proof}
Suppose that \(\vb{I} \subset T^*M\) is a holomorphic subbundle and let \(\vb{V}=\vb{I}^{\perp} \subset TM\).
Let \(p\) be the rank of \(\vb{I}\) and let \(\lambda=\lambda_{\vb{I}}\) be the holomorphic section of \(\Lm{p}{T^*M} \otimes \det{\vb{I}}^*\) given by
\[
\lambda\of{v_1,v_2,\dots,v_p}
=
\pr{v_1+\vb{V}} \wedge \pr{v_2+\vb{V}} \wedge \dots \wedge \pr{v_p+\vb{V}} \in \det{\pr{TM/\vb{V}}}=\det{\vb{I}}^*.
\]
\begin{corollary}\label{corollary:bracket.pseudo}
Suppose that \(\vb{I} \subset T^*M\) is a holomorphic vector subbundle on a compact \Kaehler manifold \(M\).
Then either \(\det \vb{I}\) is not pseudoeffective or \(\vb{V}=\vb{I}^{\perp} \subset TM\) is bracket closed and the associated differential form \(\lambda_{\vb{I}}\) of \(\vb{I}\) defined above is parallel in any singular Hermitian metric on \(\det \vb{I}\) with nonnegative curvature.
\end{corollary}
\begin{proof}
Let \(L=\det{\vb{I}}^*\) and apply theorem~\vref{theorem:Demailly.bracket}:
\[
\int_M 
\isesq{\lambda}{\frac{\sqrt{-1}}{2 \pi}\nabla^2_L}{\lambda}
\wedge e^{\Omega} 
=
\frac{1}{\pi}
\int
\sesq{\nabla \lambda}{\nabla \lambda} \wedge e^{\Omega}
\ge 0.
\]
If the line bundle \(\det{\vb{I}} = L^*\) is pseudoeffective, so admits a singular Hermitian metric with nonnegative curvature (in the sense of currents) then in that singular Hermitian metric, \(\nabla^2_L \le 0\) so the integral turns out not positive.
Therefore in that singular Hermitian metric, \(\vb{V}=\vb{I}^{\perp}\) is bracket closed.
\end{proof}

\section{Pseudoeffectivity in parabolic geometries}%
\label{section:pseudoeff.parabolic}
Take a Cartan geometry \(\G\to M\) modelled on a complex homogeneous space \((X,G)\) with \(X=G/H\).
To any holomorphic \(H\)-module \(V\), associate the holomorphic vector bundle \(\vb{V}:= \prodquot{\G}{V}{H} \to M\) whose sections are the holomorphic \(H\)-equivariant maps \(\G\to V\).
We use the same symbol \(\vb{V}\) for the associated vector bundle on the model \(X\).
\begin{corollary}\label{corollary:Q.cone.pseudoeffective}
Suppose that \((X,G)\) is a flag variety and \(P \to\G\to M\) is a minimal holomorphic \((X,G)\)-geometry on a smooth projective variety \(M\).
If \(I\) is a holomorphic \(P\)-module whose character is dominant, then the determinant \(\det \vb{I}\) of the associated vector bundle is pseudoeffective.
\end{corollary}
\begin{proof}
By lemma~\vref{lemma:quotient.modules}, the characters of the quotient \(P\)-modules of the module \(\LieGN\) span the dominant \(P\)-characters.
Hence the determinant line bundles on \(M\) associated to the induced quotient bundles of \(T^*M\) span the associated line bundles of the dominant \(P\)-characters.
By theorem~\vref{theorem:Campana.Peternell}, those determinant line bundles are pseudoeffective.
\end{proof}
\begin{corollary}\label{corollary:pseudo.cotangent}
Suppose that \((X,G)\) is a flag variety and \(P \to\G\to M\) is a minimal holomorphic \((X,G)\)-geometry on a smooth projective variety \(M\).
If \(p\) is a positive integer and \(I \subset \LieGN^{\otimes p}\) is a \(P\)-submodule, then the determinant line bundle \(\det \vb{I}\) associated to the vector bundle \(\vb{I} \subset T^*M^{\otimes p}\) is a pseudoeffective line bundle.
\end{corollary}
\begin{lemma}
Suppose that \(\pr{X,G}\) is a flag variety, \(P \to\G\to M\) is a minimal holomorphic \((X,G)\)-geometry on a smooth projective variety \(M\), and \(\chi\) is a dominant \(P\)-character.
Let \(L_{\chi} \to M\) be the associated line bundle with total space \(L_{\chi} = \amal{\G}{\chi}{\C}\).
Then either \(L_{\chi}\) is trivial or \(0=\cohomology{0}{M,L_{\chi}}\).
\end{lemma}
\begin{proof}
By lemma~\vref{lemma:pseudoeffective.trivial}, since \(-L_{\chi}\) is pseudoeffective,  \(L_{\chi}\) is trivial or has no nonzero holomorphic sections.
\end{proof}
A Cartan geometry \(\G\to M\) modelled on a homogeneous space \(X=G/P\) is \emph{bracket closed} if, for every \(P\)-submodule \(V \subset \LieG/\LieP\), the corresponding vector subbundle \(\vb{V} \subset TM\) is bracket closed.
\begin{corollary}\label{corollary:brackets}
Every minimal holomorphic parabolic geometry on any smooth projective variety is bracket closed.
\end{corollary}
\begin{proof}
Every smooth projective variety is a compact K\"ahler manifold.
Suppose that the model is \((X,G)\) with \(X=G/P\).
Take a \(P\)-submodule \(V \subset \LieG/\LieP\) and let \(\vb{V} \subset TM\) be the associated vector bundle.
Let \(I \subset \LieGN\) be \(I=V^{\perp}\) with associated vector bundle \(\vb{I} \subset T^*M\).
By corollary~\vref{corollary:bracket.pseudo}, either \(\det\vb{I}\) is not pseudoeffective or \(\vb{V}\) is bracket closed.
By lemma~\vref{lemma:Knapp.dominant}, the character \(\chi=\chi_{I}\) is dominant.
By corollary~\vref{corollary:pseudo.cotangent}, \(\det\vb{I}\) is pseudoeffective.
\end{proof}
Bracket closure ensures that any splitting of the model \(X=X_1 \times \dots \times X_k\) and \(G=G_1 \times \dots \times G_k\) gives a splitting of the universal covering space of the manifold \(M\), and a corresponding local splitting of the manifold
\[
TM = T_1 \oplus \dots \oplus T_k
\]
given by tangent bundles of foliations, one for each factor in the model, with \(\rank{T_i}=\dim{X_i}\).

\section{Structure equations of a parabolic geometry}%
\label{section:structure.equations}
\subsection{Regularity}\label{section:regularity}
A parabolic geometry \(\G\to M\) gives rise to a filtration 
\[
T_r M \subset \dots \subset T_{-r}M= TM
\] 
of the tangent bundle \(TM=\vb{V}\), corresponding to the filtration of \(V=\LieG/\LieP\),
with associated graded \(\graded{*}{TM}\).
There are two notions of bracket available here: if we take local sections \(u\) of \(\graded{i}{TM}\) and \(v\) of \(\graded{j}{TM}\), these are sections of \(\amal{\G}{P}{\graded{*}{V}}\), i.e. as \(P\)-equivariant maps
\[
u,v\colon\G\to\graded{*}{V},
\]
and these then have an algebraic bracket \(\lb{u}{v}\) coming simply from the Lie algebra \(\LieG\).
We can also locally pick vector fields \(U,V\) representing \(u,v\), sections of \(TM_i\) and \(TM_j\), and then compute vector field Lie bracket \(\lb{U}{V}\), and take the quotient in \(\graded{*}{TM}\), the \emph{Langlands bracket}.
A parabolic geometry is \emph{regular} if the algebraic and Langlands brackets agree.
The proof of theorem~\vref{theorem:comini} is trivial: the Langlands bracket vanishes by corollary~\vref{corollary:brackets}.
\emph{Danger:} regularity is not necessarily preserved when dropping.
Dropping is the main tool we have in studying holomorphic parabolic geometries.
Therefore we have to be careful not to assume regularity whenever possible.
\begin{example}
Flatness is preserved when dropping.
Flatness implies regularity.
So when we study flat holomorphic parabolic geometries, we can make use of regularity hypotheses, hence make use of theorem~\vref{theorem:comini}.
Hence we see classicality of the drop in corollary~\vref{cor:Carlson}.
\end{example}
\subsection{Structure equations}
Throughout this section, pick 
\begin{itemize}
\item
a complex semisimple Lie group \(G\) and 
\item
a flag variety \(X=G/P\) and
\item 
a Chevalley basis and
\item
a complex manifold \(M\) and 
\item
a holomorphic \((X,G)\)-geometry \(\G\to M\).
\end{itemize}
We provide explicit structure equations for parabolic geometries, more abstract than in the concrete computations of Cartan \cite{Cartan:136,Cartan:136bis,Cartan:174,Cartan:1992}, but with one system of equations for all parabolic geometries.
These equations are identical for holomorphic parabolic geometries and for real parabolic geometries with model \((X,G)\) where \(G\) is the split real form of a semisimple Lie group.
Consider the \(1\)-forms \(\om{\alpha}\) on \(G\) dual to the vectors \(\XX{\alpha}\) of the Chevalley basis. 
Use the Killing form to extend \(\alpha\) from the Cartan subalgebra \(\LieH\) to \(\LieG\), by splitting \(\LieG = \LieH \oplus \LieH^{\perp}\), and taking \(\alpha=0\) on \(\LieH^{\perp}\). 
The \(1\)-forms \(\om{\alpha}, \alpha\) span \(\LieG^*\).
The Cartan connection identifies each tangent space of \(\G\) with \(\LieG\), so pulls back \(\om{\alpha}, \alpha\) to define \(1\)-forms, given by the same symbols, on \(\G\).
Let
\(
\om{\beta \gamma} := \om{\beta} \wedge \om{\gamma},
\)
etc. and compute
\begin{align}
d \om{\aa} &= - \aa \wedge \om{\aa} - \frac{1}{2}
\sum_{\beta + \gamma=\alpha} N_{\beta \gamma} \om{\beta\gamma}
+
\frac{1}{2}
\sum_{\bb,\cc<0}
k^{\alpha}_{\beta \gamma} \om{\beta\gamma},\label{eqn:domega}
\\
d \alpha &=
-
\sum_{\beta} \frac{\KillingForm{\alpha}{\beta}}
{\KillingSquare{\beta}} \om{\beta,-\beta}
+
\frac{1}{2}
\sum_{\beta, \gamma<0}
\ell^{\alpha}_{\beta \gamma} \om{\beta\gamma},\label{eqn:dalpha}
\end{align}
with sums over all roots.
The quantities \(k^{\alpha}_{\beta\gamma}\) and \(\ell^{\alpha}_{\beta\gamma}\) are the components of the curvature of the \((X,G)\)-geometry.
We define \(k^{\alpha}_{\beta \gamma}=0\) and \(\ell^{\alpha}_{\beta \gamma}=0\) if either of \(\beta\) or \(\gamma\) is not in \(\ncptNegRts\).
Note that
\(
\ell^{\alpha+\varepsilon}_{\beta \gamma}
=
\ell^{\alpha}_{\beta \gamma}
+
\ell^{\varepsilon}_{\beta \gamma}.
\)

Let
\begin{align*}
\nabla k^{\alpha}_{\beta \gamma} 
:= &
\begin{dcases}
dk^{\aa}_{\bb \cc}
+
k^{\aa}_{\bb \cc}\pr{\aa-\bb-\cc}
-\ell^{\aa}_{\bb \cc} \delta_{\alpha\ge 0} \om{\alpha}
\\
+\sum_{\dd\ge 0} \pr{
N_{\cc \dd} k^{\aa}_{\bb,\dd+\cc} + N_{\dd \bb} k^{\aa}_{\cc,\bb+\dd} + N_{\dd,\aa-\dd} k^{\aa-\dd}_{\bb\cc} 
}
\om{\dd}
\end{dcases}
\\
+&
\begin{dcases}
-\ell^{\aa}_{\bb \cc} \delta_{\alpha<0} \om{\alpha}
\\
+\sum_{\ee<0}
\pr{
N_{\bb\cc}
k^{\aa}_{\ee,\bb+\cc}
+
N_{\ee,\aa-\ee}
k^{\aa-\ee}_{\bb\cc}
+
\sum_{\dd \in \ncptNegRts}
k^{\aa}_{\dd\ee}
k^{\dd}_{\bb\cc}
}
\om{\ee}.
\end{dcases}
\end{align*}
and let
\begin{align*}
\nabla \ell^{\alpha}_{\beta \gamma}
:= &
d \ell^{\alpha}_{\beta \gamma}
-
\ell^{\alpha}_{\beta \gamma} \pr{\beta + \gamma}
-
2
\sum_{\varepsilon}
\frac{\KillingForm{\alpha}{\varepsilon}}{\KillingSquare{\varepsilon}}
k^{\varepsilon}_{\beta \gamma} \om{-\varepsilon}
\\
&-
\sum_{\sigma<0, \varepsilon\ge 0}
\pr{
\delta_{\varepsilon+\beta=\sigma} \ell^{\alpha}_{\sigma \gamma} N_{\varepsilon \beta} 
-
\delta_{\varepsilon+\gamma=\sigma} \ell^{\alpha}_{\sigma \beta} N_{\varepsilon \gamma} 
}
\om{\varepsilon}
\\
&+
\sum_{\sigma, \varepsilon<0}
\ell^{\alpha}_{\varepsilon\sigma} 
\pr{
k^{\varepsilon}_{\beta\gamma}
-
\delta_{\beta+\gamma=\varepsilon}
N_{\beta \gamma}
}\om{\sigma}.
\end{align*}

\begin{proposition}
To each choice of \(\alpha \in \Roots, \beta, \gamma, \sigma<0\), there is a unique holomorphic function \(k^{\alpha}_{\beta \gamma \sigma} \colon\G\to \C\) with cyclic sums vanishing:
\[
0=
k^{\alpha}_{\beta \gamma \sigma}
+
k^{\alpha}_{\sigma \beta \gamma }
+
k^{\alpha}_{\gamma \sigma \beta }
\]
so that
\begin{equation}\label{equation:nabla.k}
\nabla k^{\aa}_{\bb \gamma} = \sum_{\sigma<0} k^{\alpha}_{\beta \gamma \sigma} \om{\sigma}
\end{equation}
and a unique holomorphic function \(\ell^{\alpha}_{\beta \gamma \sigma} \colon\G\to \C\) with cyclic sums vanishing
\[
0 = 
\ell^{\alpha}_{\beta \gamma \varepsilon}
+
\ell^{\alpha}_{\varepsilon \beta \gamma}
+
\ell^{\alpha}_{\gamma \varepsilon \beta}
\]
so that
\[
\nabla \ell^{\alpha}_{\beta \gamma} = \sum_{\sigma<0} \ell^{\alpha}_{\beta \gamma \sigma} \om{\sigma}.
\]
\end{proposition}
\begin{proof}
Let
\[
t^{\alpha}_{\beta \gamma}
:=
k^{\alpha}_{\beta \gamma} - \delta_{\beta+\gamma=\alpha} N_{\beta \gamma}.
\]
Note that \(t^{\aa}_{\bb\cc}+t^{\aa}_{\cc\bb}=0\) and
\[
d\om{\alpha} = - \alpha \wedge \om{\alpha} + \frac{1}{2} t^{\alpha}_{\beta \gamma}
\om{\beta \gamma},
\]
with Einstein summation convention.
Compute
\begin{align*}
0
=&
d^2\om{\alpha},
\\
=&
d\pr{-\alpha \wedge \om{\alpha} + \frac{1}{2} t^{\alpha}_{\beta \gamma}
\om{\beta \gamma}}
,
\\
=&
-d\alpha \wedge \om{\alpha}
+
\alpha \wedge d\om{\alpha}
+
\frac{1}{2} dt^{\alpha}_{\beta \gamma} \wedge \om{\beta \gamma}
+
\frac{1}{2} t^{\alpha}_{\beta \gamma} 
\pr{
d \om{\beta} \wedge \om{\gamma}
-
\om{\beta} \wedge d \om{\gamma}
},
\\
=&
-\pr{
-\sum_{\beta} \frac{\KillingForm{\alpha}{\beta}}{\KillingSquare{\beta}} \om{\beta,-\beta}
+
\frac{1}{2}\ell^{\alpha}_{\beta \gamma} \om{\beta \gamma}
} \wedge
\om{\alpha}
+
\alpha 
\wedge 
\pr{
-\alpha \wedge \om{\alpha}
+
\frac{1}{2}t^{\alpha}_{\beta \gamma} \om{\beta \gamma}}
\\&+
\frac{1}{2} dt^{\alpha}_{\beta \gamma} \wedge \om{\beta \gamma}
+
\frac{1}{2} t^{\alpha}_{\beta \gamma} 
\pr{
\pr{-\beta \wedge \om{\beta} + \frac{1}{2} t^{\beta}_{\varepsilon \sigma} \om{\varepsilon \sigma}
}
\wedge \om{\gamma}
-
\om{\beta} \wedge
\pr{-\gamma \wedge \om{\gamma} + \frac{1}{2} t^{\gamma}_{\varepsilon \sigma} \om{\varepsilon \sigma}
}
},
\\
=&
\sum_{\beta} \frac{\KillingForm{\alpha}{\beta}}{\KillingSquare{\beta}} \om{\alpha,\beta,-\beta}
-
\frac{1}{2}\ell^{\alpha}_{\beta \gamma} \om{\alpha \beta \gamma}
+
\frac{1}{2}t^{\alpha}_{\beta \gamma} \alpha \wedge \om{\beta \gamma}
+
\frac{1}{2} dt^{\alpha}_{\beta \gamma} \wedge \om{\beta \gamma}
\\&-\frac{1}{2} t^{\alpha}_{\beta \gamma} 
\pr{\beta+\gamma} \wedge \om{\beta \gamma} 
+ \frac{1}{4} t^{\alpha}_{\beta \gamma} 
\pr{t^{\beta}_{\varepsilon \sigma} \om{\gamma}- t^{\gamma}_{\varepsilon \sigma} \om{\beta}} \wedge \om{\varepsilon \sigma},
\\
=&
\frac{1}{2} 
\pr{
dt^{\alpha}_{\beta \gamma} 
+
t^{\alpha}_{\beta \gamma} \pr{\alpha-\beta-\gamma}
-
\ell^{\alpha}_{\beta \gamma} \om{\alpha}
}
\wedge \om{\beta \gamma}
+
\sum_{\beta} \frac{\KillingForm{\alpha}{\beta}}{\KillingSquare{\beta}} \om{\alpha,\beta,-\beta}
+ \frac{1}{4} t^{\alpha}_{\beta \gamma} 
\pr{t^{\beta}_{\varepsilon \sigma} \om{\gamma}- t^{\gamma}_{\varepsilon \sigma} \om{\beta}} \wedge \om{\varepsilon \sigma},
\\
=&
\frac{1}{2} 
\pr{
dt^{\alpha}_{\beta \gamma} 
+
t^{\alpha}_{\beta \gamma} \pr{\alpha-\beta-\gamma}
-
\ell^{\alpha}_{\beta \gamma} \om{\alpha}
}
\wedge \om{\beta \gamma}
+
\sum_{\beta\gamma} \frac{\KillingForm{\alpha}{\beta}}{\KillingSquare{\bb}} \delta_{\beta+\gamma=0}\om{\alpha\beta\gamma}
+ \frac{1}{4} t^{\alpha}_{\varepsilon \sigma}
\pr{t^{\varepsilon}_{\beta \gamma} \om{\sigma}- t^{\sigma}_{\beta \gamma} \om{\varepsilon}} \wedge \om{\beta \gamma},
\\
=&
\frac{1}{2} 
\pr{
dt^{\alpha}_{\beta \gamma} 
+
t^{\alpha}_{\beta \gamma} \pr{\alpha-\beta-\gamma}
-
\ell^{\alpha}_{\beta \gamma} \om{\alpha}
+ \frac{1}{2} t^{\alpha}_{\varepsilon \sigma}
\pr{t^{\varepsilon}_{\beta \gamma} \om{\sigma}- t^{\sigma}_{\beta \gamma} \om{\varepsilon}}
}
\wedge \om{\beta \gamma}
\\
&+
\frac{1}{2} \sum_{\beta\gamma} 
\pr{
\frac{\KillingForm{\alpha}{\beta}}{\KillingSquare{\bb}}
-
\frac{\KillingForm{\alpha}{\gamma}}{\KillingSquare{\gamma}}
}
 \delta_{\beta+\gamma=0}\om{\alpha\beta\gamma}
\\
=&
\frac{1}{2} 
\curly{
dt^{\alpha}_{\beta \gamma} 
+
t^{\alpha}_{\beta \gamma} \pr{\alpha-\beta-\gamma}
+
\pr{
\pr{
\frac{\KillingForm{\alpha}{\beta}}{\KillingSquare{\bb}}
-
\frac{\KillingForm{\alpha}{\gamma}}{\KillingSquare{\cc}}
}
 \delta_{\beta+\gamma=0}
-
\ell^{\alpha}_{\beta \gamma}} \om{\alpha}
+ \frac{1}{2} t^{\alpha}_{\varepsilon \sigma}
\pr{t^{\varepsilon}_{\beta \gamma} \om{\sigma}- t^{\sigma}_{\beta \gamma} \om{\varepsilon}}
}
\wedge \om{\beta \gamma}
\end{align*}
So we let
\[
D t^{\alpha}_{\beta \gamma}
=
dt^{\alpha}_{\beta \gamma} 
+
t^{\alpha}_{\beta \gamma} \pr{\alpha-\beta-\gamma}
+
\pr{
\pr{
\frac{\KillingForm{\alpha}{\beta}}{\KillingSquare{\bb}}
-
\frac{\KillingForm{\alpha}{\gamma}}{\KillingSquare{\cc}}
}
 \delta_{\beta+\gamma=0}
-
\ell^{\alpha}_{\beta \gamma}} \om{\alpha}
+ \frac{1}{2} t^{\alpha}_{\varepsilon \sigma}
\pr{t^{\varepsilon}_{\beta \gamma} \om{\sigma}- t^{\sigma}_{\beta \gamma} \om{\varepsilon}}
\]
and then
\[
0 = D t^{\alpha}_{\beta \gamma} \wedge \om{\beta \gamma}.
\]
By Cartan's lemma \cite{Sternberg:1983} p.~18 Theorem 4.4,
\[
D t^{\alpha}_{\beta \gamma} = t^{\alpha}_{\beta \gamma \sigma} \om{\sigma}
\]
for unique holomorphic functions \(t^{\alpha}_{\beta \gamma \sigma} \colon\G\to \C\)
with
\[
0=
t^{\alpha}_{\beta \gamma \sigma} 
+
t^{\alpha}_{\sigma \beta \gamma}
+
t^{\alpha}_{ \gamma \sigma \beta}.
\]

To isolate the semibasic contribution, we first expand
\begin{align*}
Dt^{\aa}_{\bb \cc}
=&
dk^{\aa}_{\bb \cc}
+
k^{\aa}_{\bb \cc}\pr{\aa-\bb-\cc}
-\ell^{\aa}_{\bb \cc}\om{\alpha}
-\frac{1}{2}\delta_{\dd+\ee=\aa} N_{\dd \ee} 
\pr{
k^{\dd}_{\bb \cc} \om{\ee}
-
k^{\ee}_{\bb \cc} \om{\dd}
}
\\
&+\frac{1}{2}k^{\aa}_{\dd \ee}\pr{k^{\dd}_{\bb \cc} \om{\ee} - k^{\ee}_{\bb \cc} \om{\dd}}
\\
&+\frac{1}{2}
\underbrace{
k^{\aa}_{\dd \ee} N_{\bb \cc}
\pr{
\delta_{\bb+\cc=\ee} \om{\dd} - \delta_{\bb+\cc=\dd} \om{\ee}
}
}_{(1)}
\\
&+
\underbrace{
\pr{
\frac{\KillingForm{\aa}{\bb}}{\KillingSquare{\bb}}
-
\frac{\KillingForm{\aa}{\cc}}{\KillingSquare{\cc}}
}\delta_{\bb+\cc=0} \om{\aa}
+
\frac{1}{2}\delta_{\dd+\ee=\aa}N_{\bb \cc}N_{\dd\ee}
\pr{
\delta_{\bb+\cc=\dd}\om{\ee}
- 
\delta_{\bb+\cc=\ee}\om{\dd}
}
}_{(2)}
\end{align*}
The term marked (2) will drop out of the expression \(Dt^{\alpha}_{\beta \gamma} \wedge \om{\beta \gamma}\), without altering the result, because we can pick the model geometry and arrange \(k=\ell=0\) and still satisfy the equation \(0=Dt^{\alpha}_{\beta \gamma} \wedge \om{\beta \gamma}\) without altering (2).
Besides the term (1), the rest of the terms vanish when \(\beta\) or \(\gamma\) are not in \(\ncptNegRts\).
Consider (1), wedged with \(\om{\bb \cc}\):
\begin{align*}
\frac{1}{2} k^{\aa}_{\dd \ee} N_{\bb \cc} 
\pr{
\delta_{\bb+\cc=\ee} \om{\dd} - \delta_{\bb+\cc=\dd} \om{\ee}
} \wedge \om{\bb \cc}
=&
k^{\aa}_{\dd, \bb+\cc} N_{\bb \cc} 
\om{\dd \bb \cc} 
\end{align*} 
This expression vanishes unless \(\dd<0\) and \(\bb+\cc<0\), which forces either \(\bb\) or \(\cc\) in \(\ncptNegRts\), or both:
\begin{align*}
\frac{1}{2} k^{\aa}_{\dd \ee} N_{\bb \cc} 
\pr{
\delta_{\bb+\cc=\ee} \om{\dd} - \delta_{\bb+\cc=\dd} \om{\ee}
} \wedge \om{\bb \cc}
=&
\sum_{\bb, \cc, \dd<0}
k^{\aa}_{\dd, \bb+\cc} N_{\bb \cc} 
\om{\dd \bb \cc}
\\
&+
\underbrace{
\sum_{\cc, \dd<0}^{\bb\ge 0}
k^{\aa}_{\dd, \bb+\cc} N_{\bb \cc} 
\om{\dd \bb \cc}
}_{(a)}
+
\underbrace{
\sum_{\bb, \dd<0}^{\cc\ge 0}
k^{\aa}_{\dd, \bb+\cc} N_{\bb \cc} 
\om{\dd \bb \cc}
}_{(b)}
,
\\
\intertext{and swap \(\bb \leftrightarrow \dd\) in (a), \(\cc \leftrightarrow  \dd\) in (b) to yield}
=&
\sum_{\bb, \cc, \dd<0}
k^{\aa}_{\dd, \bb+\cc} N_{\bb \cc} 
\om{\dd \bb \cc}
\\
&+
\sum_{\bb, \cc<0}^{\dd\ge 0}
k^{\aa}_{\bb, \dd+\cc} N_{\dd \cc} 
\om{\bb \dd \cc}
+
\sum_{\bb, \cc<0}^{\dd\ge 0}
k^{\aa}_{\cc, \bb+\dd} N_{\bb \dd} 
\om{\cc \bb \dd}
,
\\
=&
\sum_{\bb, \cc<0}
\kappa^{\aa}_{\bb \cc} \wedge \om{\bb \cc}
\end{align*}
where
\[
\kappa^{\aa}_{\bb \cc}
=
\sum_{\ee<0}
k^{\aa}_{\ee, \bb+\cc}
N_{\bb\cc}
\om{\ee}
+
\sum_{\dd \in \ParaRts}
\pr{
k^{\alpha}_{\bb, \cc+\dd} N_{\cc \dd} 
+
k^{\alpha}_{\cc, \bb+\dd} N_{\dd \bb} 
}\om{\dd}.
\]

To sum up, for any \(\alpha \in \Roots\) and \(\beta, \gamma<0\), let
\begin{align*}
\nabla k^{\alpha}_{\beta \gamma} 
=&
dk^{\aa}_{\bb \cc}
+
k^{\aa}_{\bb \cc}\pr{\aa-\bb-\cc}
-\ell^{\aa}_{\bb \cc} \om{\alpha}
-\frac{1}{2}\delta_{\dd+\ee=\aa} N_{\dd \ee} 
\pr{
k^{\dd}_{\bb \cc} \om{\ee}
-
k^{\ee}_{\bb \cc} \om{\dd}
}
\\
&+\frac{1}{2}k^{\aa}_{\dd \ee}\pr{k^{\dd}_{\bb \cc} \om{\ee} - k^{\ee}_{\bb \cc} \om{\dd}}
\\
&+\sum_{\ee<0}
k^{\aa}_{\ee, \bb+\cc}
N_{\bb\cc}
\om{\ee}
+
\sum_{\dd\ge 0}
\pr{
k^{\alpha}_{\bb, \cc+\dd} N_{\cc \dd} 
+
k^{\alpha}_{\cc, \bb+\dd} N_{\dd \bb} 
}\om{\dd}.
\end{align*}
By Cartan's lemma \cite{Sternberg:1983} p.~18 Theorem 4.4,
for any \(\alpha \in \Roots\) and \(\beta, \gamma, \sigma<0\), there exists a unique holomorphic function \(k^{\alpha}_{\beta \gamma \sigma} \colon\G\to\C\), with cyclic sums vanishing:
\[
0=
k^{\alpha}_{\beta \gamma \sigma}
+
k^{\alpha}_{\sigma \beta \gamma }
+
k^{\alpha}_{\gamma \sigma \beta }
\]
so that
\[ 
\nabla k^{\aa}_{\bb \gamma} = \sum_{\sigma<0} k^{\alpha}_{\beta \gamma \sigma} \om{\sigma}.
\]
It is convenient to split into semibasic and nonsemibasic terms as:
\begin{align*}
\nabla k^{\alpha}_{\beta \gamma} 
=&
\begin{dcases}
dk^{\aa}_{\bb \cc}
+
k^{\aa}_{\bb \cc}\pr{\aa-\bb-\cc}
-\ell^{\aa}_{\bb \cc} \delta_{\alpha\ge 0} \om{\alpha}
\\
+\sum_{\dd\ge 0} \pr{
N_{\cc \dd} k^{\aa}_{\bb,\dd+\cc} + N_{\dd \bb} k^{\aa}_{\cc,\bb+\dd} + N_{\dd,\aa-\dd} k^{\aa-\dd}_{\bb\cc} 
}
\om{\dd}
\end{dcases}
\\
+&
\begin{dcases}
-\ell^{\aa}_{\bb \cc} \delta_{\alpha<0} \om{\alpha}
\\
+\sum_{\ee<0}
\pr{
N_{\bb\cc}
k^{\aa}_{\ee,\bb+\cc}
+
N_{\ee,\aa-\ee}
k^{\aa-\ee}_{\bb\cc}
+
\sum_{\dd<0}
k^{\aa}_{\dd\ee}
k^{\dd}_{\bb\cc}
}
\om{\ee}.
\end{dcases}
\end{align*}

Similarly let
\begin{align*}
\nabla \ell^{\alpha}_{\beta \gamma}
=&
d \ell^{\alpha}_{\beta \gamma}
-
\ell^{\alpha}_{\beta \gamma} \pr{\beta + \gamma}
-
2
\sum_{\varepsilon}
\frac{\KillingForm{\alpha}{\varepsilon}}{\KillingSquare{\varepsilon}}
k^{\varepsilon}_{\beta \gamma} \om{-\varepsilon}
\\
&-
\sum_{\sigma<0, \varepsilon\ge 0}
\pr{
\delta_{\varepsilon+\beta=\sigma} \ell^{\alpha}_{\sigma \gamma} N_{\varepsilon \beta} 
-
\delta_{\varepsilon+\gamma=\sigma} \ell^{\alpha}_{\sigma \beta} N_{\varepsilon \gamma} 
}
\om{\varepsilon}
\\
&+
\sum_{\sigma, \varepsilon<0}
\ell^{\alpha}_{\varepsilon\sigma} 
\pr{
k^{\varepsilon}_{\beta\gamma}
-
\delta_{\beta+\gamma=\varepsilon}
N_{\beta \gamma}
}\om{\sigma}.
\end{align*}

Write
\(
d\alpha = \frac{1}{2}T^{\alpha}_{\beta \gamma} \om{\bb \cc},
\)
where
\[
T^{\aa}_{\bb \cc} = 
\pr{
\frac{\KillingForm{\aa}{\cc}}{\KillingSquare{\cc}}
-
\frac{\KillingForm{\aa}{\bb}}{\KillingSquare{\bb}}
}\delta_{\bb+\cc=0}
+
\ell^{\aa}_{\bb \cc}.
\]
Take exterior derivative:
\begin{align*}
0 
&=
2 \, d^2\aa,
\\
&=
dT^{\aa}_{\bb \cc} \wedge \om{\bb \cc}
+
T^{\aa}_{\bb \cc} 
\pr{
d\om{\bb} \wedge \om{\cc}
-
\om{\bb} \wedge d\om{\cc}
},
\\
&=
dT^{\aa}_{\bb \cc} \wedge \om{\bb \cc}
+
2 \, T^{\aa}_{\bb \cc} 
d\om{\bb} \wedge \om{\cc}
,
\\
&=
d\ell^{\aa}_{\bb \cc} \wedge \om{\bb \cc}
+
2 \, 
\pr{
	\pr{
		\frac{\KillingForm{\aa}{\cc}}{\KillingSquare{\cc}}
		-
		\frac{\KillingForm{\aa}{\bb}}{\KillingSquare{\bb}}
	}
	\delta_{\bb+\cc=0}
+
\ell^{\aa}_{\bb \cc}
}
d\om{\bb} \wedge \om{\cc}
,
\\
&=
d\ell^{\aa}_{\bb \cc} \wedge \om{\bb \cc}
\\
&\quad +
2 \pr{\pr{
\frac{\KillingForm{\aa}{\cc}}{\KillingSquare{\cc}}
-
\frac{\KillingForm{\aa}{\bb}}{\KillingSquare{\bb}}
}\delta_{\bb+\cc=0}
+\ell^{\aa}_{\bb \cc}} 
\pr{
-\beta \wedge \om{\bb} - \frac{1}{2}\pr{\delta_{\dd+\ee=\bb}N_{\dd\ee}-k^{\bb}_{\dd\ee}}\om{\dd \ee}
}
 \wedge \om{\cc}
,
\\
&=
d\ell^{\aa}_{\bb \cc} \wedge \om{\bb \cc}
-2
	\pr{
		\frac{\KillingForm{\aa}{\cc}}{\KillingSquare{\cc}}
		-
		\frac{\KillingForm{\aa}{\bb}}{\KillingSquare{\bb}}
	}
	\delta_{\bb+\cc=0}
	\beta \wedge \om{\bb \cc}
\\
&\quad
-
	\pr{
		\frac{\KillingForm{\aa}{\cc}}{\KillingSquare{\cc}}
		-
		\frac{\KillingForm{\aa}{\bb}}{\KillingSquare{\bb}}	
	}
	\delta_{\bb+\cc=0}
	\pr{\delta_{\dd+\ee=\bb}N_{\dd\ee}-k^{\bb}_{\dd\ee}}\om{\dd \ee \cc} 
\\
&\quad
-2\ell^{\aa}_{\bb \cc} 
	\beta \wedge \om{\bb\cc}
-\ell^{\aa}_{\bb \cc}
\pr{\delta_{\dd+\ee=\bb}N_{\dd\ee}-k^{\bb}_{\dd\ee}}\om{\dd \ee \cc}.
\end{align*}
In the flat model geometry, we have \(k=\ell=0\) and therefore have
\[
0=
-2
	\pr{
		\frac{\KillingForm{\aa}{\cc}}{\KillingSquare{\cc}}
		-
		\frac{\KillingForm{\aa}{\bb}}{\KillingSquare{\bb}}
	}
	\delta_{\bb+\cc=0}
	\beta \wedge \om{\bb\cc}
-
	\pr{
		\frac{\KillingForm{\aa}{\cc}}{\KillingSquare{\cc}}
		-
		\frac{\KillingForm{\aa}{\bb}}{\KillingSquare{\bb}}
	}
	\delta_{\bb+\cc=0}
	\delta_{\dd+\ee=\bb}N_{\dd\ee}\om{\dd \ee \cc}.
\]
But then the same equation holds in every \((X,G)\)-geometry.
Indeed for the first term, in every \((X,G)\)-geometry,
\begin{align*}
-2
	\pr{
		\frac{\KillingForm{\aa}{\cc}}{\KillingSquare{\cc}}
		-
		\frac{\KillingForm{\aa}{\bb}}{\KillingSquare{\bb}}
	}
	\delta_{\bb+\cc=0}
	\beta \wedge \om{\bb\cc}
&=
4
		\frac{\KillingForm{\aa}{\bb}}{\KillingSquare{\bb}}
	\beta \wedge \om{\bb,-\bb},
\\
\intertext{and we swap \(\bb\leftrightarrow-\bb\),}
&=
-4
		\frac{\KillingForm{\aa}{\bb}}{\KillingSquare{\bb}}
	\beta \wedge \om{\bb,-\bb},
\end{align*}
hence vanishes.
The other term is more difficult, but we already know it must vanish.
Therefore
\begin{align*}
0
&=
d\ell^{\aa}_{\bb \cc} \wedge \om{\bb \cc}
+
	\pr{
		\frac{\KillingForm{\aa}{\cc}}{\KillingSquare{\cc}}
		-
		\frac{\KillingForm{\aa}{\bb}}{\KillingSquare{\bb}}
	}
	\delta_{\bb+\cc=0}
	k^{\bb}_{\dd\ee}\om{\dd \ee \cc} 
\\
&\quad
-2\ell^{\aa}_{\bb \cc} 
	\beta \wedge \om{\bb \cc}
-\ell^{\aa}_{\bb \cc}
\pr{\delta_{\dd+\ee=\bb}N_{\dd\ee}-k^{\bb}_{\dd\ee}}\om{\dd \ee \cc},
\\
&=
d\ell^{\aa}_{\bb \cc} \wedge \om{\bb \cc}
+
	\pr{
		\frac{\KillingForm{\aa}{\ee}}{\KillingSquare{\ee}}
		-
		\frac{\KillingForm{\aa}{\dd}}{\KillingSquare{\dd}}
	}
	\delta_{\dd+\ee=0}
	k^{\dd}_{\bb\cc}\om{\bb \cc \ee} 
\\
&\quad
-\ell^{\aa}_{\bb \cc} 
	\beta \wedge \om{\bb \cc}
-\ell^{\aa}_{\cc \bb} 
	\cc \wedge \om{\cc \bb}
\\
&\quad
-2\delta_{\dd\ge 0, \ee<0} \ell^{\aa}_{\bb \cc}
\delta_{\dd+\ee=\bb}N_{\dd\ee}\om{\dd \ee \cc}
-\delta_{\dd, \ee<0} \ell^{\aa}_{\bb \cc}
\delta_{\dd+\ee=\bb}N_{\dd\ee}\om{\dd \ee \cc}
\\
&\quad
+\ell^{\aa}_{\bb \cc}k^{\bb}_{\dd\ee}\om{\dd \ee \cc},
\\
&=
d\ell^{\aa}_{\bb \cc} \wedge \om{\bb \cc}
+
	\pr{
		\frac{\KillingForm{\aa}{\ee}}{\KillingSquare{\ee}}
		-
		\frac{\KillingForm{\aa}{\dd}}{\KillingSquare{\dd}}
	}
	\delta_{\dd+\ee=0}
	k^{\dd}_{\bb\cc}\om{\bb \cc \ee} 
\\
&\quad
-\ell^{\aa}_{\bb \cc} \pr{\bb+\cc} \wedge \om{\bb \cc}
\\
&\quad
-\delta_{\dd\ge 0, \bb<0} \ell^{\aa}_{\ee \cc}
\delta_{\dd+\bb=\ee}N_{\dd\bb}\om{\dd \bb \cc}
\\
&\quad
-\delta_{\dd\ge 0, \cc<0} \ell^{\aa}_{\ee \bb}
\delta_{\dd+\cc=\ee}N_{\dd\cc}\om{\dd \cc \bb}
\\
&\quad
-\delta_{\bb, \cc<0} \ell^{\aa}_{\dd \ee}
\delta_{\bb+\cc=\dd}N_{\bb\cc}\om{\bb \cc \ee}
\\
&\quad
+\ell^{\aa}_{\dd \ee}k^{\dd}_{\bb\cc}\om{\bb \cc \ee},
\\
&=
d\ell^{\aa}_{\bb \cc} \wedge \om{\bb \cc}
+
	\pr{
		\frac{\KillingForm{\aa}{\ee}}{\KillingSquare{\ee}}
		-
		\frac{\KillingForm{\aa}{\dd}}{\KillingSquare{\dd}}
	}
	\delta_{\dd+\ee=0}
	k^{\dd}_{\bb\cc}\om{\ee \bb \cc} 
\\
&\quad
-\ell^{\aa}_{\bb \cc} \pr{\bb+\cc} \wedge \om{\bb \cc}
\\
&\quad
-\delta_{\dd\ge 0, \bb<0} \ell^{\aa}_{\ee \cc}
\delta_{\dd+\bb=\ee}N_{\dd\bb}\om{\dd \bb \cc}
\\
&\quad
+\delta_{\dd\ge 0, \cc<0} \ell^{\aa}_{\ee \bb}
\delta_{\dd+\cc=\ee}N_{\dd\cc}\om{\dd \bb \cc}
\\
&\quad
+\ell^{\aa}_{\dd \ee}
\pr{
k^{\dd}_{\bb\cc}
-\delta_{\bb, \cc<0} 
\delta_{\bb+\cc=\dd}N_{\bb\cc}}\om{\ee \bb \cc},
\\
&=
\nabla \ell^{\aa}_{\bb \cc} \wedge \om{\bb \cc}.
\end{align*}
By Cartan's lemma \cite{Sternberg:1983} p.~18 Theorem 4.4,
there are unique holomorphic functions \(\ell^{\alpha}_{\beta \gamma \varepsilon} \colon\G\to \C\) so that
\(
\nabla \ell^{\alpha}_{\beta \gamma} = \ell^{\alpha}_{\beta \gamma \varepsilon} \om{\varepsilon},
\)
with \(\ell^{\alpha}_{\beta \gamma \varepsilon}\) with
\[
0 = 
\ell^{\alpha}_{\beta \gamma \varepsilon}
+
\ell^{\alpha}_{\varepsilon \beta \gamma}
+
\ell^{\alpha}_{\gamma \varepsilon \beta}.
\]
\end{proof}

\section{Bracket closed geometries}%
\label{section:bracket.closed.geometries}
Throughout this section, pick 
\begin{itemize}
\item
a complex semisimple Lie group \(G\) and
\item
a flag variety \(X=G/P\) and 
\item
a Chevalley basis and 
\item
a complex manifold \(M\) and 
\item
a bracket closed holomorphic \((X,G)\)-geometry \(\pi \colon\G\to M\).
\end{itemize}
\begin{lemma}\label{lemma:bracket.closed.equations}
A parabolic geometry is bracket closed just when, for each saturated set \(\Gamma \subset \ncptNegRts\), if \(\alpha\in\ncptNegRts\backslash\Gamma\) and \(\beta, \gamma\in \Gamma\) then
\[
k^{\aa}_{\bb \cc}=\delta_{\aa=\bb+\cc}N_{\bb \cc}.
\]
\end{lemma}
\begin{proof}
Let \(\bar\Gamma:=\ncptNegRts\backslash \Gamma\).
By the Frobenius theorem, the subbundle \(\vb{V}_\G \subset T\G\) cut out by the equations \(\om{\alpha}\) for \(\alpha \in \bar\Gamma\) is bracket closed just when each \(d\om{\alpha}\) lies in the ideal generated by the \(1\)-forms \(\om{\beta}\), for \(\beta \in\bar\Gamma\).
From equation~\vref{eqn:domega}, this is equivalent to \(k^{\alpha}_{\beta \gamma}=\delta_{\beta+\gamma=\alpha}N_{\beta \gamma}\) for every \(\beta, \gamma \in \Gamma\).
Since \(\Gamma \subset \Roots\) is saturated, its root spaces sum to a \(P\)-submodule  \(V_{\Gamma} \subset \LieG\), with quotient \(P\)-module \(V \subset \LieG/\LieP\) having an associated vector bundle \(\vb{V}_M \subset TM\), the quotient under the map \(T\G \to TM\) which is the derivative of \(\G\to M\).
Thus \(\vb{V}_M\) is bracket closed just when \(\vb{V}_\G\) is.
\end{proof}
\begin{corollary}
For every bracket closed parabolic geometry, if \(\alpha < \beta < 0\) and \(\alpha < \cc < 0\) are roots, then
\[
k^{\aa}_{\bb \cc}=\delta_{\aa=\bb+\cc}N_{\bb \cc}.
\]
\end{corollary}

\subsection{Torsion components}
For each flag variety \((X,G)\) and \((X,G)\)-geometry \(\G\to M\) and integer \(j \ge 1\), the \emph{\(j\)-torsion} is
\[
\nator^j :=
\sumOver{-j\le\Pheight{\aa}<0}{\nator_{\aa} \omega^{\aa}}
\]
where
\[
\nator_{\aa}
:= 
\sum_{\bb<\aa}	k^{\bb}_{\aa \bb}.
\]
\begin{lemma}\label{lemma:nator.derivative}
For each root \(\aa<0\), the semibasic \(1\)-form
\[
\nabla \nator_{\aa} := 
	\sum_{\bb<\aa}
	\nabla k^{\bb}_{\aa \bb}
\]
satisfies
\[
\nabla \nator_{\aa} 
=
d\nator_{\aa}
-\aa \nator_{\aa}
+2
\frac{\KillingForm{\chi_{>\left|\Pheight{\aa}\right|}}{\aa}}{\KillingSquare{\aa}}
\om{-\aa}
+\sum_{\Pheight{\dd}=0} N_{\dd\aa} \nator_{\aa+\dd} \om{\dd}
+\dots 
\]
writing \(\dots\) for terms \(\om{\dd}\) for \(\dd< 0\) or \(0<\dd<-\aa\).
\end{lemma}
\begin{proof}
The covariant derivative of \(\nator_{\aa}\) is:
\begin{align*}
\nabla \nator_{\aa}
&=
\sum_{\bb<\aa}
	\nabla
		k^{\bb}_{\aa \bb},
\\
&=
\sum_{\bb<\aa} \biggl(
	d
		k^{\bb}_{\aa \bb}
	+ k^{\bb}_{\aa \bb} \pr{\bb-\aa-\bb} 
	\\
	& \quad 
	+ \sum_{\dd \ge 0}
	\pr{
	N_{\dd \aa} k^{\bb}_{\bb,\aa+\dd}
	+
	N_{\bb \dd} k^{\bb}_{\aa,\dd+\bb}
	+
	N_{\dd, \bb-\dd} k^{\bb-\dd}_{\aa\bb}	
	}\om{\dd} \biggr) + \dots
\end{align*}
We split the sum over \(\dd \ge 0\) into its \(3\) terms, which we denote by \(1,2,3\) respectively.
We split each of these into sums as follows: each has a factor \(k^{\sigma}_{\mu \nu}\), and we split each into four cases, which we denote with a subscript \(+\) if \(\mu>\sigma\), \(-\) otherwise, and a second subscript \(+\) if \(\nu>\sigma\), \(-\) otherwise, giving \(12\) sums in all, which we denote by \(1_{++}, 1_{+-}\), and so on, as follows:
\[
\begin{array}{*5{>{\displaystyle}l}}
\toprule
& ++ & +- & -+ & -- 
\\ 
\midrule 
1
&
\sumOver{\bb<\aa \\ \bb, \aa+\dd > \bb}{N_{\dd \aa} k^{\bb}_{\bb,\aa+\dd}}
&
\sumOver{\bb<\aa \\ \bb > \bb \\ \aa+\dd \le \bb}{N_{\dd \aa} k^{\bb}_{\bb,\aa+\dd}}
&
\sumOver{\bb<\aa \\ \bb \le \bb \\ \aa+\dd > \bb}{N_{\dd \aa} k^{\bb}_{\bb,\aa+\dd}}
&
\sumOver{\bb<\aa \\ \bb, \aa+\dd \le \bb}{N_{\dd \aa} k^{\bb}_{\bb,\aa+\dd}}
\\[40pt]
2
&
\sumOver{\bb<\aa \\ \aa, \dd+\bb > \bb}{N_{\bb \dd} k^{\bb}_{\aa,\dd+\bb}}
&
\sumOver{\bb<\aa \\ \aa > \bb \\ \dd+\bb \le \bb}{N_{\bb \dd} k^{\bb}_{\aa,\dd+\bb}}
&
\sumOver{\bb<\aa \\ \aa \le \bb \\ \dd+\bb > \bb}{N_{\bb \dd} k^{\bb}_{\aa,\dd+\bb}}
&
\sumOver{\bb<\aa \\ \aa, \dd+\bb \le \bb}{N_{\bb \dd} k^{\bb}_{\aa,\dd+\bb}}
\\[40pt]
3
&
\sumOver{\bb<\aa \\ \aa, \bb > \bb-\dd}{N_{\dd, \bb-\dd} k^{\bb-\dd}_{\aa\bb}}
&
\sumOver{\bb<\aa \\ \aa > \bb-\dd \\ \bb \le \bb-\dd}{N_{\dd, \bb-\dd} k^{\bb-\dd}_{\aa\bb}}
&
\sumOver{\bb<\aa \\ \aa \le \bb-\dd \\ \bb > \bb-\dd}{N_{\dd, \bb-\dd} k^{\bb-\dd}_{\aa\bb}}
&
\sumOver{\bb<\aa \\ \aa, \bb \le \bb-\dd}{N_{\dd, \bb-\dd} k^{\bb-\dd}_{\aa\bb}}
\\ \bottomrule
\end{array}
\]

Because the geometry is bracket closed, by lemma~\vref{lemma:bracket.closed.equations}, for \(\bb < \mu, \nu < 0\), 
\[
k^{\bb}_{\mu \nu}=\delta_{\bb=\mu+\nu}N_{\mu \nu}.
\]
We will apply this to various terms as follows.

In the term
\[
1_{++}=
\sumOver{\bb<\aa \\ \dd \ge 0 \\ \bb, \aa+\dd > \bb}{N_{\dd \aa} k^{\bb}_{\bb,\aa+\dd}}
\om{\dd},
\]
we cannot have \(\bb > \bb\), so there are no such terms.

In the term
\[
1_{+-}=
\sumOver{\bb<\aa \\ \dd \ge 0 \\ \bb > \bb \\ \aa+\dd \le \bb}{N_{\dd \aa} k^{\bb}_{\bb,\aa+\dd}}\om{\dd},
\]
we cannot have \(\bb > \bb\), so there are no such terms.

In the term
\[
1_{-+}=
\sumOver{\bb<\aa \\ \dd \ge 0 \\ \bb \le \bb \\ \aa+\dd > \bb}{N_{\dd \aa} k^{\bb}_{\bb,\aa+\dd}}
\om{\dd}
\]
to get \(k^{\bullet}_{\bullet,\aa+\dd}\) not zero, we need \(\aa+\dd<0\).
We are computing modulo \(\om{\dd}\) except for \(\Pheight{\dd}=0\) or \(\dd\ge-\aa\), i.e. \(\aa+\dd\ge0\), but such terms yield zeroes, so:
\[
1_{-+}=
\sumOver{\bb<\aa \\ \Pheight{\dd}=0 \\ \aa+\dd > \bb}{N_{\dd \aa} k^{\bb}_{\bb,\aa+\dd}}
\om{\dd}
\]
Note that \(\aa+\dd > \bb\) just precisely when \(\aa > \bb\):
\begin{align*}
1_{-+}
&=
\sumOver{\bb<\aa+\dd \\ \Pheight{\dd}=0}{N_{\dd \aa} k^{\bb}_{\bb,\aa+\dd}}
\om{\dd},
\\
&=
\sum_{\Pheight{\dd}=0}{N_{\dd \aa} \nator_{\aa+\dd}
\om{\dd}}.
\end{align*}

In the term
\[
1_{--}=
\sumOver{\bb<\aa \\ \dd \ge 0 \\ \bb, \aa+\dd \le \bb}{N_{\dd \aa} k^{\bb}_{\bb,\aa+\dd}}
\om{\dd}
\]
\(\aa+\dd\le \bb<\aa\), but \(\dd\ge 0\), a contradiction, so there are no such terms.

Consider the term
\[
2_{++}=\sumOver{\bb<\aa \\ \dd \ge 0 \\ \aa, \dd+\bb > \bb}{N_{\bb \dd} k^{\bb}_{\aa,\dd+\bb}} \om{\dd}.
\]
By lemma~\vref{lemma:bracket.closed.equations},
\begin{align*}
k^{\bb}_{\aa,\dd+\bb}
&=
N_{\aa,\dd+\bb} \delta_{\aa+\dd+\bb=\bb},
\\
&=
N_{\aa,\dd+\bb} \delta_{\dd=-\aa},
\end{align*}
and thus
\[
2_{++}
=
\sum_{\bb<\aa}{N_{\bb,-\aa}N_{\aa,\bb-\aa} \om{-\aa}}.
\]

In the term
\[
2_{+-}=
\sumOver{\bb<\aa \\ \dd \ge 0 \\ \aa > \bb \\ \dd+\bb \le \bb}{N_{\bb \dd} k^{\bb}_{\aa,\dd+\bb}} \om{\dd}
\]
the requirement that \(0 \le \dd\) and that \(\dd+\bb \le \bb\) ensures, taking grades, that \(\Pheight{\dd}=0\).
So
\[
2_{+-}=
\sumOver{\bb<\aa \\ \Pheight{\dd}=0}{N_{\bb \dd} k^{\bb}_{\aa,\dd+\bb}} \om{\dd}.
\]
We can't simplify this term any further.

In the term
\[
2_{-+}=
\sumOver{\bb<\aa \\ \dd \ge 0 \\ \aa \le \bb \\ \dd+\bb > \bb}{N_{\bb \dd} k^{\bb}_{\aa,\dd+\bb}}\om{\dd},
\]
\(\aa \le \bb\), but \(\bb < \aa\), so there are no such terms.

In the term
\[
2_{--}=\sumOver{\bb<\aa \\ \dd \ge 0 \\ \aa, \dd+\bb \le \bb}{N_{\bb \dd} k^{\bb}_{\aa,\dd+\bb}}\om{\dd},
\]
\(\aa \le \bb\), but \(\bb < \aa\), so there are no such terms.

Take the term
\[
3_{++}=
\sumOver{\bb<\aa \\ \dd \ge 0 \\ \aa, \bb > \bb-\dd}{N_{\dd, \bb-\dd} k^{\bb-\dd}_{\aa\bb}}
\om{\dd}.
\]
Again by our lemma~\vref{lemma:bracket.closed.equations},
\begin{align*}
k^{\bb-\dd}_{\aa\bb}
&=
N_{\aa\bb} \delta_{\aa+\bb=\bb-\dd},
\\
&=
N_{\aa\bb} \delta_{\dd=-\aa},
\end{align*}
and thus
\[
3_{++}
=
\sumOver{\bb<\aa}{N_{-\aa, \aa+\bb}N_{\aa\bb} \om{-\aa}}.
\]

Take the term
\[
3_{+-}=
\sumOver{\bb<\aa \\ \dd \ge 0 \\ \aa > \bb-\dd \\ \bb \le \bb-\dd}{N_{\dd, \bb-\dd} k^{\bb-\dd}_{\aa\bb}}
\om{\dd}.
\]
From \(\bb \le \bb-\dd\) we see that \(\dd\) has grade zero, i.e. is a compact root.
\[
3_{+-}=
\sumOver{\bb<\aa \\ \Pheight{\dd}=0}{N_{\dd, \bb-\dd} k^{\bb-\dd}_{\aa\bb}}
\om{\dd}.
\]
Let \(\cc:=\bb-\dd\).
Note that if \(\bb\) and \(\cc\) are not both roots, then the associated term vanishes.
Moreover \(\bb<\aa\) just when \(\cc<\aa\).
So we can write \(\bb=\cc+\dd\), and write
\begin{align*}
3_{+-}
&=
\sumOver{\cc<\aa \\ \Pheight{\dd}=0}{N_{\dd\cc} k^{\cc}_{\aa,\cc+\dd}}
\om{\dd},
\\
\intertext{and if we now formally replace the symbol \(\cc\) with the symbol \(\bb\)}
&=
\sumOver{\bb<\aa \\ \Pheight{\dd}=0}{N_{\dd\bb} k^{\bb}_{\aa,\bb+\dd}}
\om{\dd},
\\
&=
-\sumOver{\bb<\aa \\ \Pheight{\dd}=0}{N_{\bb\dd} k^{\bb}_{\aa,\bb+\dd}}
\om{\dd},
\\
&=-2_{+-}.
\end{align*}
So \(3_{+-}\) cancels \(2_{+-}\) in \(\nabla \nator_{\aa}\).

Take the term
\[
3_{-+}=
\sumOver{\bb<\aa \\ \dd \ge 0 \\ \aa \le \bb-\dd \\ \bb > \bb-\dd}{N_{\dd, \bb-\dd} k^{\bb-\dd}_{\aa\bb}}
\om{\dd}.
\]
Since \(\bb<\aa\) while \(\aa \le \bb-\dd\), and both of \(\bb\) and \(\bb-\dd\) are roots, \(\aa \le \bb-\dd \le \bb\), a contradiction, so there are no such terms.

Finally, take the term
\[
3_{--}=
\sumOver{\bb<\aa \\ \dd \ge 0 \\ \aa, \bb \le \bb-\dd}{N_{\dd, \bb-\dd} k^{\bb-\dd}_{\aa\bb}}
\om{\dd}.
\]
Then \(\bb \le \bb-\dd\) forces \(\Pheight{\dd}\le 0\), so that \(\Pheight{\dd}=0\).
But then \(\bb < \aa\) so \(\bb-\dd<\aa\).
This contradicts \(\bb-\dd\ge \aa\), so there are no such terms.

Summing up, writing \(\dots\) for terms \(\om{\dd}\) for \(\dd<0\) or \(0<\dd<-\aa\),
\begin{align*}
\nabla \nator_{\aa}
&=
d\nator_{\aa}
-\aa \nator_{\aa}
+\sum_{\bb<\aa}%
{
\pr%
{
N_{\bb,-\aa}N_{\aa,\bb-\aa} 
+
N_{-\aa, \aa+\bb}N_{\aa\bb}
}
\om{-\aa}
}
+
\sum_{\Pheight{\dd}=0} N_{\dd\aa} \nator_{\aa+\dd} \om{\dd}
+ \dots
\end{align*}

By lemma~\vref{lemma:Jacobi.mystery}
\begin{align*}
\nabla \nator_{\aa}
&=
d\nator_{\aa}
-\aa \nator_{\aa}
-2\sum_{\bb<\aa}%
{
\frac{\KillingForm{\bb}{\aa}}{\KillingSquare{\aa}}
\om{-\aa}
}
+
\sum_{\Pheight{\dd}=0} N_{\dd\aa} \nator_{\aa+\dd} \om{\dd}
+ 
\dots,
\\
&=
d\nator_{\aa}
-\aa \nator_{\aa}
-2
\frac{\KillingForm{
\sum_{\bb<\aa}\bb
}
{\aa}}{\KillingSquare{\aa}}
\om{-\aa}
+\sum_{\Pheight{\dd}=0} N_{\dd\aa} \nator_{\aa+\dd} \om{\dd}
+ \dots,
\\
&=
d\nator_{\aa}
-\aa \nator_{\aa}
+2
\frac{\KillingForm{\chi_{>\left|\Pheight{\aa}\right|}}{\aa}}{\KillingSquare{\aa}}
\om{-\aa}
+\sum_{\Pheight{\dd}=0} N_{\dd\aa} \nator_{\aa+\dd} \om{\dd}
+ \dots.
\end{align*}
\end{proof}

\section{The boosh}%
\label{section:boosh}
\subsection{Notation}
Throughout this section, pick 
\begin{itemize}
\item
a complex semisimple Lie group \(G\) and
\item
a flag variety \(X=G/P\) and 
\item
a Chevalley basis and 
\item
a complex manifold \(M\) and 
\item
a bracket closed holomorphic \((X,G)\)-geometry \(\pi \colon\G\to M\).
\end{itemize}
Let
\[
T:=\bigoplus_{\aa} \LieG_{\aa},
\]
the sum over \(\aa>0\) submaximal, and recall the definitions:
\[
\LieGZ:=\bigoplus_{\aa} \LieG_{\aa}
\]
the sum over \(P\)-compact roots,
\[
\LieG_{\text{max}}:=\bigoplus_{\aa} \LieG_{\aa},
\]
the sum over \(\aa>0\) maximal,
\[
\LieP':=\LieGZ\oplus\LieG_{\text{max}}, 
\]
and \(P':=e^{\LieP'}\).
The \(P'\)-module \(\LieG\) contains \(T\) as a \(P'\)-submodule.
Let \(V':= \LieG/T\) be the quotient \(P'\)-module.

\subsection{Definition}
Let \(\G_0:=\G\).
Let \(\G_1\subseteq\G\) be the set of points of \(\G\) at which the \(1\)-torsion vanishes.
Inductively, let \(\G_j \subseteq\G_{j-1}\) be the set of points at which the \(j\)-torsion vanishes.
The \emph{boosh} of \(\G\to M\) is
\[
\G':=\bigcap_j\G_j,
\]
i.e. the smallest of these \(\G_j\).
Pull back the Cartan connection \(\omega\) on \(\G\) to \(\G'\), and compose with the projection \(\LieG\to V'=\LieG/T\) as above to form a \(1\)-form \(\omega' \in \nForms{1}{\G'} \otimes^{P'}V'\).

\subsection{Generalized Cartan geometries}
The boosh is not a Cartan geometry in the conventional sense; we need to use the more general notion of Cartan geometry \cite{McKay:2023} p.~229.
An \emph{infinitesimal model} \((V,H)\) in this broader sense is a complex Lie group \(H\) and a finite dimensional holomorphic \(H\)-module \(V\) containing \(\LieH\) as an \(H\)-submodule.
Every homogeneous space \((X,G)\) has infinitesimal model \((\LieG,H)\) where \(H:=G^{x_0}\) is, as usual, the stabilizer of a point.
A \((V,H)\)-geometry, or \emph{Cartan geometry infinitesimally modelled on \((V,H)\)} is a right principal \(H\)-bundle \(\G\to M\) with an \(H\)-equivariant \(V\)-valued \(1\)-form \(\omega\), the \emph{Cartan connection}, giving a linear isomorphism of all tangent spaces of \(\G\) with \(V\), and agreeing with the Maurer--Cartan form on the fibers.
Clearly every \((X,G)\)-geometry is a \((V,H)\)-geometry for \((V,H)=(\LieG,H)\).
If \(V\) comes equipped with an \(H\)-invariant Lie algebra structure, we can still make sense of curvature, torsion and flatness; otherwise, curvature, torsion and flatness are not defined.
Unfortunately, there is no natural choice of Lie algebra structure for the infinitesimal model of the boosh.

\subsection{The boosh as Cartan geometry}
\begin{lemma}\label{lemma:normalize}
The boosh \(\G'\subset\G\) of any holomorphic bracket-closed parabolic geometry is a holomorphic principal right \(P'\)-subbundle, on which \(\omega^{\aa}\) is semibasic for any \(\aa>0\) not maximal in its irreducible factor of the root system of \(G\).
Thus \(\omega'\) is a Cartan connection modelled on the infinitesimal model \((V',P')\).
\end{lemma}
\begin{proof}
At each point \(p\in\G_j\), for each root \(\aa\) of grade \(\Pheight{\aa}=-j\),  \(\nator_{\aa}=0\), so on \(T_p\G_{j-1}\),
\[
\nabla \nator_{\aa} = 
2
\frac{\KillingForm{\chi_{\ge j+1}}{\aa}}{\aa^2} \omega^{-\aa} + \dots
\]
where the \(\dots\) are semibasic terms, by induction.
It is convenient to switch the sign on \(\aa\) here: replacing \(\aa\) by \(-\aa\), so \(0<\Pheight{\aa}\le j\).
By lemma~\vref{lemma:not.perp.roots},
\[
\KillingForm{\chi_{\ge j+1}}{\aa}>0
\]
precisely when \(\Pheight{\aa}=j\).
By induction, if \(0<\Pheight{\aa}<j\), we already have \(\om{\aa}\) semibasic, but there is at least one root \(\aa\) with \(\Pheight\aa=j\), since \(j\) is not maximal.
So the equations of \(\G_j\) have full rank at every point of \(\G_j\), so \(\G_j \subset\G\) is a complex submanifold of codimension equal to the number of roots of grade \(j\).
Since there are no \(\om{\aa}\) for roots of grade zero in this expression, and no \(1\)-forms \(\aa\), \(\G_j\) is invariant under the action of \(\LieGZ\).
Similarly \(\G_j\) is invariant under \(\LieG_{\bb}\) for \(\bb>\aa\).
By induction, \(\G'\) is invariant under \(\LieGZ\) and \(\LieG'\).
Because \(\GZ\) and \(G'\) are connected, \(\G'\) is invariant under \(\GZ\) and \(G'\).
By dimension count and transversality, \(\G'\) is a principal right \(\GZ G'\)-subbundle.
Inductively, on each \(\G_j\), every \(\omega^{\aa}\) is now semibasic for \(0<\Pheight\aa\) submaximal.
\end{proof}
\subsection{Structure equations of the boosh so far}
Write \(\hat\aa\) to mean that \(\aa\) is maximally graded for the flag variety.
Define
\begin{align*}
\nabla'\om\aa&=d\om{\aa}+\aa\wedge\om{\aa}
-\sum_{\Pheight{\sigma}=0}N_{\alpha+\sigma,-\sigma}\omega^{-\sigma,\alpha+\sigma},\qquad\aa<0,\\
\nabla'\om\aa&=d\om{\aa}+\aa\wedge\om{\aa}
+
\sum_{\Pheight{\bb}=0=\Pheight{\cc}}^{\bb+\cc=\aa}N_{\bb\cc}\omega^{\bb\cc}
+
\sum_{\hat\bb,\hat\cc}^{\hat\bb-\hat\cc=\aa}N_{\hat\bb,-\hat\cc}\omega^{\hat\bb,-\hat\cc}
,\qquad\Pheight{\aa}=0,\\
\nabla'\om{\hat\aa} &= d\om{\hat\aa}+\hat\aa \wedge \om{\hat\aa}+\sum_{\hat\bb + \cc=\hat\aa}^{\Pheight{\cc}=0} N_{\hat\bb \cc} \om{\hat\bb\cc},\\
\nabla'\aa&=d\aa+2
\sum_{\hat\bb} 
\frac{\KillingForm{\alpha}{\hat\bb}}{\KillingSquare{\hat\bb}} 
\om{\hat\bb,-\hat\bb}.
\end{align*}
The structure equations of the boosh are
\begin{align*}
\nabla'\om{\aa} &=
\sum_{\aa\le\bb<0,\cc<\aa}k^{\alpha}_{\bb\gamma}\om{\bb\gamma}-
\sum_{\Pheight{\beta}>0}^{\aa=\bb+\cc}N_{\beta\cc}\omega^{\bb\cc}
,\qquad\aa<0,
\\
\nabla'\om{\aa}&=
\frac{1}{2}
\sum_{\bb,\cc<0}K^{\alpha}_{\bb\cc}\om{\bb,\cc},\qquad\Pheight{\aa}=0,
\\
\nabla'\om{\hat\aa} &= 
\frac{1}{2}
\sum_{\bb,\cc<0}
K^{\hat\aa}_{\beta \gamma} \om{\beta\gamma}, 
\\
\nabla'\alpha &=
\frac{1}{2}
\sum_{\bb, \cc<0}
L^{\aa}_{\bb \cc} \om{\bb\cc}.
\end{align*}
\subsection{More Lie brackets}
The \(P'\)-action on \(\LieG/\LieP\) factors through \(P'\to G_0\), , as \(\rtspMax{\LieG}\) acts trivially, so the boosh induces a \(G_0\)-structure on \(M\).
So every \(P'\)-submodule \(V \subset \LieG/\LieP\) is precisely \(V=\pr{\LieG/\LieP}_{\Gamma}\) where \(\Gamma\subset\ncptNegRts\) is any \(\GZ\)-saturated subset.
\begin{proposition}
Suppose that \(M\) is a smooth projective variety and \(\G\to M\) is a minimal holomorphic parabolic geometry with boosh \(\G'\to M\).
Take a \(G_0\)-submodule \(V\subseteq\LieG/\LieP\).
Then \(V\) is an \(P'\)-submodule, hence has associated holomorphic distribution \(\vb{V}\subseteq TM\); this distribution is bracket closed.
\end{proposition}
\begin{proof}
The \(G_0\)-submodule \(V\subseteq\LieG/\LieP\) has a \(G_0\)-invariant complement \(V_{\perp}\) since \(G_0\) is reductive.
Hence we have a torsion-free quotient module \((\LieG/\LieP)^*\to Q:=V^{\perp}=V_{\perp}^*\).
So the vector bundle \(\vb{V}\) is both a subsheaf and a quotient sheaf of \(TM\), and so \(\vb{I}:=\vb{V}^{\perp}\) is therefore both a subsheaf and a quotient sheaf of \(\nForms{1}{M}\).
By theorem~\vref{theorem:Campana.Peternell}, \(\det\vb{Q}\) is pseudoeffective, i.e. the determinant \(\det \vb{I}\) is pseudoeffective.
By corollary~\vref{corollary:bracket.pseudo}, \(V\) is bracket closed, i.e. the tangent bundle of a unique holomorphic foliation of \(M\), i.e. \(\vb{I}\) is \(d\)-closed.
As above, if we let \(\bar\Gamma:=\ncptNegRts-\Gamma\):
\[
k^{\aa}_{\bb\cc} = N_{\bb\cc} \delta_{\alpha=\beta+\gamma},
\]
for \(\alpha\in\bar\Gamma\)  and \(\beta, \gamma\in\Gamma\).
\end{proof}
Split up the \(P\)-negative roots into minimal collections so that each collection \(\Gamma\) is the set of roots whose root vectors span a \(G_0\)-submodule of \(\LieG/\LieP\). 
Each collection has root vectors spanning a \(G_0\)-submodule of \(\LieG/\LieP\).
(See pictures of these minimal collections \cite{McKay:2020}).
This splits \(\LieG/\LieP\) into a maximal direct sum of \(G_0\)-submodules
\[
\LieG/\LieP=\bigoplus_{\Gamma}V_{\Gamma}.
\]
The smooth projective variety \(M\) has correspondingly split tangent bundle 
\[
T=\bigoplus_{\Gamma}T\vb{F}_{\Gamma}
\]
from the various foliations \(\vb{F}_{\Gamma}\) with \(T\vb{F}_{\Gamma}=\vb{V}_{\Gamma}\).
\subsection{Structure equations of the boosh}
For each root \(\aa<0\), let \(\Gamma:=\Gamma_{\aa}\) be the minimal collection of roots \(\bb<0\) so that \(V_{\Gamma}\) is a \(G_0\)-module.
Incorporating these additional bracket closure equations, the structure equations of the boosh are, for some holomorphic functions \(K^{\aa}_{\bb\cc},L^{\aa}_{\bb\cc}\): 
\begin{align*}
\nabla'\om{\aa} &=
\frac{1}{2}\sum_{\bb,\cc\in\Gamma_{\aa}}K^{\alpha}_{\bb\gamma}\om{\bb\gamma}
+
\sum_{\bb\notin\Gamma_{\aa},\cc\in\Gamma_{\aa}}K^{\alpha}_{\bb\gamma}\om{\bb\gamma}
,\qquad\aa<0,
\\
\nabla'\om{\aa}&=
-
\sum_{\hat\bb,\hat\cc}^{\hat\bb-\hat\cc=\aa}N_{\hat\bb,-\hat\cc}\omega^{\hat\bb,-\hat\cc}
+
\frac{1}{2}
\sum_{\bb,\cc<0}K^{\alpha}_{\bb\cc}\om{\bb,\cc},\qquad\Pheight{\aa}=0,
\\
\nabla'\om{\hat\aa} &= 
\frac{1}{2}
\sum_{\bb,\cc<0}
K^{\hat\aa}_{\beta \gamma} \om{\beta\gamma}, 
\\
\nabla'\alpha &=
\frac{1}{2}
\sum_{\bb, \cc<0}
L^{\aa}_{\bb \cc} \om{\bb\cc}.
\end{align*}
\begin{example}
Pick some negative grades, and pick out the root vectors of those grades to generate a \(G_0\)-submodule \(V\subseteq\LieG/\LieP\).
We can already read off from the structure equations above that the associated holomorphic distribution \(\vb{V}\) is the tangent bundle of a unique holomorphic foliation of \(M\), without using the previous proposition.
\end{example}

\subsection{Quotienting down the boosh}
Let \(\bar{\G}':=\G'/Z'\) where \(Z'\subseteq P'\) is the center of the unipotent radical of \(P'\).
We have principal bundles
\[
\begin{tikzcd}
\bar{\G}'\arrow[d]&P'/Z'\arrow[l]\\
M
\end{tikzcd}
\]
and
\[
\begin{tikzcd}
\G'\arrow[d]&Z'\arrow[l]\\
\bar{\G}'.
\end{tikzcd}
\]
The \(1\)-forms \(\omega^\aa\) for \(\aa<0\) are semibasic for \(\G'\to M\).
The \(1\)-forms \(\omega^\aa\) for \(\Pheight{\aa}=0\) and the \(1\)-forms \(\aa\) are semibasic for \(\G'\to\bar{\G}'\).
So at each point of \(\G'\), these \(1\)-forms project to \((1,0)\)-forms on the tangent space at the underlying point of \(\bar{\G}'\).
\subsection{Reducing structure group of the boosh}
Note that \(Z'\subseteq P'\) is abelian, connected and simply connected \cite{Knapp:2002} p.~393 Proposition~7.31, hence biholomorphic to complex Euclidean space, so smoothly contractible. 

Take a \(C^\infty\) principal bundle \(G\to E\to M\).
Suppose that the structure group \(G\) has finitely many components.
Take an embedded Lie subgroup \(H\subseteq G\).
Then \(G\to E\to M\) admits a \(C^\infty\) retraction to an \(H\)-subbundle if and only if \(H\) contains a \(C^\infty\) maximal compact subgroup of \(G\)  \cite{Hilgert.Neeb:2012} p.~, \cite{Steenrod:1999} \S 12.14.

Therefore, returning to our bracket closed geometry, there is a global \(C^\infty\)-smooth section \(\bar{\G}'\xrightarrow{s}\G'\), i.e. a reduction of \(\G'\) to a principal right \(P'/Z'\)-bundle.
When we pullback the \(1\)-forms \(\omega^\aa\) for \(\aa<0\), \(\omega^\aa\) for \(\Pheight{\aa}=0\) and \(\aa\), they agree with these \((1,0)\)-projections, so they are \((1,0)\)-forms on \(\bar{\G}'\).
But we have no control on whether the \(1\)-forms \(\omega^{\hat\aa}\), pulled back to \(\bar{\G}'\), are \((1,0)\) or \((0,1)\) or a sum of \((1,0)\) and \((0,1)\) say
\[
\omega^{\hat\aa}=(\omega^{\hat\aa})^{1,0}+(\omega^{\hat\aa})^{0,1}.
\]
\subsection{The Atiyah class of the boosh}\label{section:boosh.Atiyah}
On \(\bar{\G}'\), the \(1\)-forms \(\omega^\aa\) for \(\Pheight{\aa}=0\) and the \(1\)-forms \(\aa\) together form the components of a \((1,0)\)-connection.
The Atiyah class of this connection is the image in Dolbeault cohomology of its \((1,1)\)-part \cite{Atiyah:1957}.
From the structure equations, the components of the Atiyah class are
\begin{align*}
\bar\partial\om{\aa} &= 
-
\sum_{\hat\bb,\hat\cc}^{\hat\bb-\hat\cc=\aa}N_{\hat\bb,-\hat\cc}(\om{\hat\bb})^{0,1}\wedge\om{-\hat\cc}, \text{ if } \Pheight{\aa}=0,
\\
\bar\partial\alpha &=
-
2
\sum_{\hat\bb} 
\frac{\KillingForm{\alpha}{\hat\bb}}{\KillingSquare{\hat\bb}} 
(\om{\hat\bb})^{0,1}\wedge\om{-\hat\bb}
\end{align*}
Note that these expressions do not contain any contribution from the curvature of the parabolic geometry.
Also note that we have no control over the possible values of the \((0,1)\)-forms here, because they can vary arbitrarily depending on the tangent space to the choice of reduction of structure group.
More discussion of the Atiyah class in terms of these structure equations, see \cite{McKay:2022}.

\section{Chern classes on the boosh}%
\label{section:Chern.boosh}
The \emph{characteristic ring} of a holomorphic Cartan geometry \(\G\to M\) is the ring of Dolbeault cohomology classes in \(H^{*,*}(M)\) consisting of all polynomials in all Chern classes of all associated holomorphic vector bundles of the bundle \(\G\to M\).
In particular, the characteristic ring of any holomorphic Cartan geometry \(\G\to M\) contains the Chern classes of the tangent bundle of \(M\).

Take a flag variety \((X,G)\).
Denote by \(P\subseteq G\) the \(G\)-stabilizer of some point \(x_0\in X\). 
Take an associated cominuscule subvariety \((\breve{X},\breve{G})\) through \(x_0\).
Denote by \(\breve{P}\subseteq\breve{G}\) the \(\breve{G}\)-stabilizer of \(x_0\in X\). 
Let \(G'\subseteq G\) be the subgroup acting on \(X\) preserving \(\breve{X}\subseteq X\) and let \(P'\subseteq G'\) be the stabilizer of \(x_0\).
Each element of \(G'\) thus acts as an automorphism of \(\breve{X}\) given by a unique element of \(\breve{G}\).
Thus we have a  holomorphic Lie group morphism \(G'\to\breve{G}\), hence a morphism \(P'\to\breve{P}\) of stabilizer subgroups.
Recall that \(G'\to\breve{G}\) is onto \cite{McKay:2020}.

Each \(G'\)-invariant vector bundle \(\vb{V}\to\breve{X}\) is associated to a unique \(P'\)-module \(V\).
In particular, the morphism \(P'\to\breve{P}\) allows us to take \(\breve{P}\)-modules and construct out of them \(G'\)-invariant vector bundles.
Take a basis \(\breve{\Delta}\) of simple roots \(\alpha\) of \(\breve{P}\).
Associate to each character \(\chi\) of \(\breve{P}\) the line bundle \(\vb{L}_{\chi}\) and hence its first Chern class.
Recall that \(P, P'\) and \(\breve P\) have the same characters, so the same line bundles.

On a compact complex manifold \(M\), take a holomorphic bracket closed parabolic geometry \(\G\to M\) with model \((X,G)\).
As above, denote its boosh by \(\G'\to M\).
Hence, to each \(P'\)-module, construct a holomorphic vector bundle \(\vb{V}\to M\) by \(\vb{V}:=\amal{\G'}{P'}{V}\).
In particular, we have line bundles \(\vb{L}_{\chi}\to M\).

\begin{theorem}\label{theorem:Bracket.closed}
On a compact complex manifold \(M\), take a holomorphic bracket closed parabolic geometry \(\G\to M\) with model \((X,G)\).
The map 
\[
(\vb{V}\to\breve{X})\mapsto (\vb{V}\to M)
\]
taking \(G'\)-invariant holomorphic vector bundles to holomorphic vector bundles extends uniquely to a complex linear map
\[
H^{*,*}(\breve{X}) \to H^{*,*}(M),
\]
taking the Dolbeault cohomology of \(\breve{X}\) to that of \(M\).
Its image contains the characteristic ring of the parabolic geometry.
In particular, any polynomial 
\[
P(c_1,c_2,\dots,c_n)
\]
in Chern classes of those associated vector bundles which vanishes on \(\breve{X}\) (for example, has degree \((p,p)\) with \(p>\dimC{\breve{X}}\)) vanishes on \(M\).
Hence the numerical dimension of any holomorphic line bundle associated to the Cartan geometry is at most the dimension of \(\breve{X}\).
\end{theorem}
\begin{proof}
For every \(P'\)-module \(V\), we can compute its Chern classes as polynomials in the Atiyah class of the bundle \(\G'\to M\).
The Atiyah class of \(G'\to\breve X\) has precisely the same expression as that of \(\G'\to M\) from section~\vref{section:boosh.Atiyah}, in any reduction of structure group to a maximal compact subgroup of \(P'\)  \cite{McKay:2011}, except that the \((1,0)\)-forms and \((0,1)\)-forms appearing on that maximal compact subgroup are \(\C\)-linearly independent.
Hence any relation on \(\breve X\) implies one on \(\G'\).

The characters of the Cartan subgroup generate the characteristic ring of the parabolic geometry \cite{McKay:2011}.
Pick a character \(\chi\) of \(P'\).
Since \(P'\) has the same Cartan subgroup as \(P\), \(\chi\) is a character of \(P\).
By lemma~\vref{lemma:chi.j}, up to rational multiples, we can subtract off some multiples of characters arising from \(\breve{G}\)-invariant line bundles on \(\breve{X}\) to arrange that \(\chi\) is perpendicular to all \(P\)-maximal roots \(\hat\bb\).

Expand \(\chi\) in a basis of simple roots \(\alpha_i\), say
\[
\chi=\sum n_i\alpha_i.
\]
Denote also by \(\chi\) the \(1\)-form on \(\G\) given by the same expression using the basis of forms on \(\G\) described in section~\vref{section:structure.equations}.
The Chern class of \(\chi\) is represented in Dolbeault cohomology by the same polynomial in the components of the Atiyah class, after reduction of the structure group of \(\G'\) to a maximal compact subgroup of the boosh structure group \(P'\), by \(\partial\chi\) as above.
Expand into the structure equations of the boosh:
\[
\bar\partial \alpha=
-
2
\sum_{\hat\bb} 
\frac{\KillingForm{\alpha}{\hat\bb}}{\KillingSquare{\hat\bb}} 
\om{\hat\bb,-\hat\bb}
\]
Hence the wedge products arising are those \(\om{\hat\bb,-\hat\bb}\) for the maximal roots \(\hat\bb\).
The inner products \(\KillingForm{\chi}{\hat\bb}\) vanish.
Hence \(\bar\partial\chi\) has vanishing \((1,1)\)-part on any reduction of structure group to a maximal compact subgroup.
So the associated line bundle \(\vb{L}_\chi\to M\) has a holomorphic connection and vanishing \(c_1\).
Therefore characteristic ring is generated by the first Chern classes of characters \(\chi\) associated to \(\breve{G}\)-invariant line bundles on \(\breve{X}\).
\end{proof}

\begin{proposition}
The characteristic ring of any cominuscule variety is generated by the first Chern classes of the tangent bundles of its irreducible factors.
\end{proposition}
\begin{proof}
Since \((\breve{X},\breve{G})\) is cominuscule, we can split it into a product of irreducible cominuscules
\begin{align*}
\breve{X}&=\breve{X}_1\times\dots\times\breve{X}_k,\\
\breve{G}&=\breve{G}_1\times\dots\times\breve{G}_k,\\
\breve{P}&=\breve{P}_1\times\dots\times\breve{P}_k.
\end{align*}
Each \(\breve\LieG_i\) is \(1\)-graded, i.e. over its Levi factor \(G_{i,0}\):
\[
\breve{\LieG}_i=\breve{\LieG}_{i,-1}\oplus\breve{\LieG}_{i,0}\oplus\breve{\LieG}_{i,1}
\]
with \(T\breve{X}_{i,x_0}=\breve{\LieG}_{i,-1}\).
The first Chern class of \(\breve{X}_i\) is therefore
\[
c_1=\frac{\sqrt{-1}}{2\pi}\sum_{\alpha} d\alpha,
\]
the sum over the roots \(\alpha\) of grade \(-1\) of \(\breve{\LieG}_i\).
This sum is not zero, because there are negative roots (i.e. \(\dim\breve{X}_i>0\) by definition), so \(c_1\) spans the unique linear subspace of the dual \(\breve{\LieH}_i^*\) of the Cartan subalgebra of \(\breve{G}_i\) which is perpendicular to the roots of \(\breve{\LieG}_{i,0}\).
This same linear subspace contains the first Chern class of every line bundle on \(\breve{X}_i\), and so of every line bundle on \(X\) which is associated to a \(1\)-dimensional \(\breve{P}_i\)-module.
\end{proof}
\begin{corollary}
Suppose that \(M\) is a smooth projective variety admitting a minimal holomorphic parabolic geometry.
Suppose that the model has \(k\) irreducible factors
\[
X=X_1\times X_2\times\dots\times X_k.
\]
The subring of the characteristic class ring of \(M\) in Dolbeault cohomology generated by the invariant vector bundles on the model is generated in degree \((1,1)\) by \(k\) commuting \((1,1)\)-forms, the first Chern forms of the line bundles associated to the canonical bundles of the associated cominuscules \(\breve X_i\) of the factors \(X_i\) of the model.
\end{corollary}
Take a flag variety \((X,G)\).
As above, denote the stabilizer of a point \(x_0\in X\) as \(P\subseteq G\), and \(\GZ\subseteq P\) a maximal reductive Levi factor.
Recall that the the \emph{tally} \(N_X\) of an irreducible flag variety \((X,G)\) is the infimum number of elements in a nonempty set \(\Gamma\) of roots for which the sum \(V_{\Gamma}\subset V=\LieG/\LieP\) of root spaces is a \(\GZ\)-module.
\begin{theorem}
Suppose that \(M\) is a smooth projective variety admitting a minimal parabolic geometry.
Suppose that the model \((X,G)\) splits into \(k\) irreducible flag varieties
\begin{align*}
X&=X_1\times X_2\times\dots\times X_k,\\
X&=G_1\times G_2\times\dots\times G_k,
\end{align*}
with tallies
\[
N_1,N_2,\dots,N_k
\]
respectively.
The characteristic ring of the holomorphic parabolic geometry is generated by \(k\) first Chern classes \(\alpha_i=c_1(L_i)\) of holomorphic line bundles, each satisfying \(\alpha_i^{N_i+1}=0\).
In particular, any polynomial in Chern classes of associated holomorphic vector bundles of the parabolic geometry, with degree more than \(N_i\) in the Chern classes of invariant vector bundles associated to the factor \(X_i\) vanishes.
Hence the numerical dimension of any holomorphic line bundle associated to the Cartan geometry is at most the sum of the tallies of the model.
\end{theorem}
\begin{proof}
Each of our foliations \(\vb{V}:=\vb{V}_{\Gamma}\subseteq TM\) has Atiyah class the image in \(\cohomology{1}{M,T^*\otimes\vb{V}^*\otimes \vb{V}}\) of a class in \(\cohomology{1}{M,\vb{V}^*\otimes\vb{V}^*\otimes\vb{V}}\) \cite{Beauville:2000} p. 3, Lemma 3.1; see \cite{McKay:2022b} for a proof more in the spirit of our work here.
Hence polynomials in Chern classes of \(\vb{V}\) of total degree more than the rank of \(\vb{V}\) (i.e. the number of elements of \(\Gamma\)) vanish.

The set \(\Gamma\) is the set of weights of an \(H\)-module \(V=V_\Gamma\subseteq\LieG/\LieP\), i.e. of negative grade.
Splitting \(TM\) according to the splitting of the model, and thereby splitting \(\vb{V}\), we can assume that \(\Gamma\) is a set of roots of a single factor \(X_i\) of the model \(X\).
The first Chern class of \(\vb{V}\) is represented by
\[
\bar\partial\chi_\Gamma=\sum_{\alpha\in\Gamma} \bar\partial\alpha.
\]
This sum is over weights of negative grade, so \(\chi\) is negative grade.
The minimal roots of that factor \(X_i\) are also negative, so the associated \(P\)-module, a quotient of \(\LieG/\LieP\), also has negative grade, so first Chern class
\[
\bar\partial\chi_i=\sum_{\alpha\in\Delta_i} \bar\partial\alpha.
\]
By the above, \(c_1(\vb{V})=c_1(\vb{L}_\chi)\) is a multiple of \(c_1(\vb{L}_{\chi_i})\), and must be a nonzero multiple.
So the characteristic ring is also generated by the classes \(c_1(\vb{V})\) of the tangent bundles of the foliations.
\end{proof}
\begin{example}
Take as model
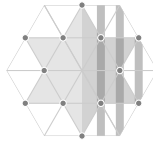
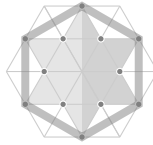
\[
\drawroots{G}{2}
\]
with associated cominuscule
\[
\begin{tikzpicture}[baseline=-.5]
\begin{rootSystem}{G}
\roots
\parabolic{2}
\draw[/root system/grading] (hex cs:x=-1,y=2) -- (hex cs:x=1,y=1) -- (hex cs:x=2,y=-1) -- (hex cs:x=1,y=-2) -- (hex cs:x=-1,y=-1) -- (hex cs:x=-2,y=1) -- cycle;
\end{rootSystem}
\end{tikzpicture}
\]
The tally is \(1\), since the middle vertical line in 
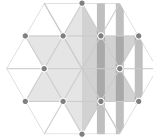
\[
\drawroots{G}{2}
\]
has one root on it.
Hence any minimal holomorphic geometry with this model has numerical dimension at most \(1\).
\end{example}
\begin{example}\label{ex:E.8}
The \(E_8\)-flag variety \dynkin E{******x*} has Hasse diagram:
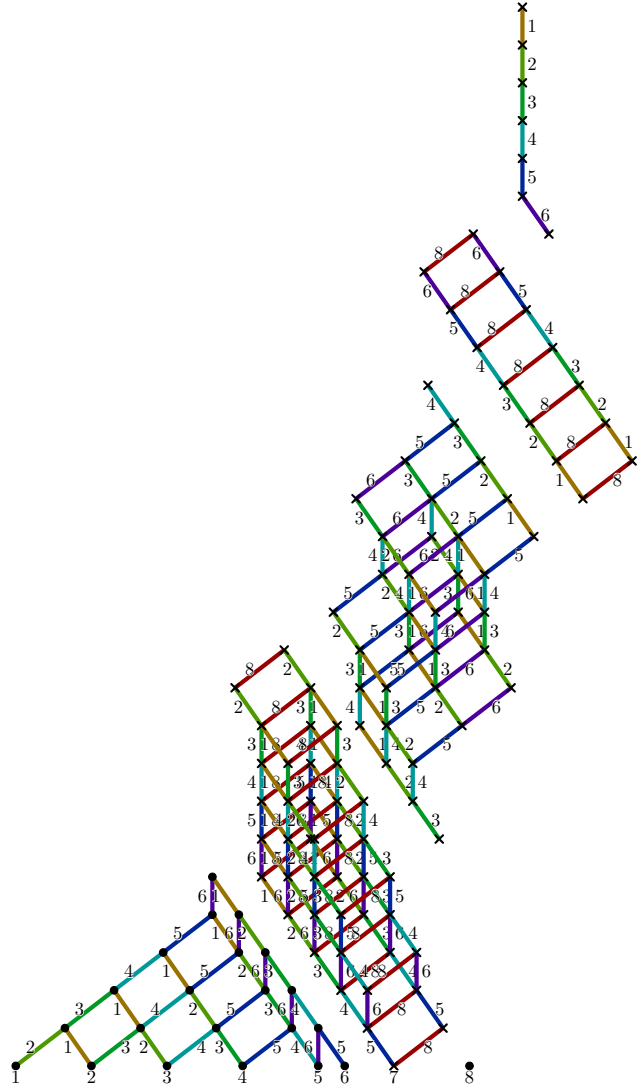
\begin{center}
\hasse[three D=false] E{******x*}
\end{center}
See \cite{McKay:2020} for a complete discussion of Hasse diagrams of flag varieties.
The connected components that have crosses on their vertices represent the irreducible invariant subsheaves of the associated graded of the tangent bundle of the flag variety.
In this example, their dimensions (down from the top) are \(7,14,35,42\).
So the dimension of the flag variety is \(7+14+35+42=98\).
The tally is the smallest of these dimensions: \(7\), from the component of the highest root.
Therefore every smooth projective variety \(M\) bearing a minimal geometry with this model has numerical dimension at most \(7\) and dimension \(98\).
The associated cominuscule variety is projective space \(\Proj{7}\).
Therefore 
\[
c_k=\binom{8}{k}c_1^k
\]
for \(k=1,2,\dots,7\), while \(c_k=0\) for \(k>7\).
\end{example}
\begin{example}
Every adjoint variety and every complete flag variety \(G/B\) has tally \(1\), so every minimal geometry with such a model has numerical dimension at most \(1\).
\end{example}
Roughly speaking, minimal parabolic geometries only occur on smooth projective varieties with large numerical codimension.
This is the only result known that demonstrates that smooth projective varieties without rational curves typically have no holomorphic parabolic geometries.
\section{Weyl structures and the canonical bundle}%
\label{section:Weyl}
Take a flag variety \((X,G)\).
Pick a point \(x_0\in X\) and let \(H:=G^{x_0}\).
A \emph{Weyl structure} for an \((X,G)\)-geometry \(\G\to M\) on a complex manifold \(M\) is a holomorphic reduction of structure group \(\G_0\subset\G\), i.e.~a holomorphic principal \(G_0\)-subbundle,
We summarize some results from \cite{McKay2013}, with an outline of the proof.
\begin{lemma}\label{lemma:canon.conn}
Take a complex manifold with a holomorphic parabolic geometry.
Every holomorphic connection on the canonical bundle determines a Weyl structure and determines a holomorphic connection on 
\begin{itemize}
\item
the tangent bundle and
\item
the boosh and
\item
the bundle of the parabolic geometry.
\end{itemize}
\end{lemma}
\begin{proof}
Take a flag variety \((X,G)\), say \(X=G/P\).
Pick a Cartan subgroup of \(G\) lying in \(P\).
Let \(\delta\) be half the sum of the \(P\)-noncompact positive roots.
Suppose that \(P \to\G\to M\) is a holomorphic \((X,G)\)-geometry, with Cartan connection \(\omega \in \amal{\nForms{1}{\G}}{P}{\LieG}\).
Then \(\delta \circ \omega \in \nForms{1}{\G}\), and we also denote this \(1\)-form on \(\G\) as \(\delta\).
Any holomorphic connection \(\nabla\) on the canonical bundle of \(M\) can be written as a \(1\)-form \(\phi\) on the total space \(\kappa_M^{\times}\) of the associated principal bundle.
Pick a nonzero complex linear volume form \(\xi\) on \(\LieG/\LieP\), i.e. \(\xi \in \Lmtop{\LieG/\LieP}^*-0\).
Let \(\bar\omega:= \omega+\LieP\).
At each point \(e \in\G\), \(\xi \circ (\bar\omega,\bar\omega,\dots,\bar\omega)\) is a semibasic top degree form, so the pullback of a unique element \(\xi_e \in \Lmtop{T_m M}^*-0\).
Via this bundle map \(e \in\G\mapsto \xi_e\in\kappa_M^{\times}\), we can pullback the connection \(1\)-form \(\phi\), and denote the pullback also as \(\phi\).
On the fibers of \(\G\to\kappa_M^{\times}\), \(\phi=\delta\), since both are characters of the structure group scaling the same under the Cartan subgroup: as half the sum of the positive roots.
Therefore, on \(\G\), \(\phi=\delta + \sum t_{\alpha}\omega^{\alpha}\), a sum over \(P\)-noncompact negative roots, for some holomorphic functions \(\G\xrightarrow{t_{\alpha}}\C\).
The subset \(\G_0\subset\G\) on which \(t_{\alpha}=0\) for all \(\alpha\) is a principal right \(\GZ\)-subbundle of \(\G\).
Splitting \(\LieG\) as a \(\GZ\)-module: \(\LieG=\LieGX \oplus \LieGZ \oplus \LieGN\), the Cartan connection splits \(\omega=\omX + \omZ + \omN\).
The \(1\)-form \(\omZ\) is a holomorphic connection on \(\G_0\).
It imposes a holomorphic connection on every quotient bundle, in particular on the tangent bundle, and on \(\G\) and on \(\G'\).
\end{proof}
\section{Bundles on tori}%
\label{section:bundles.on.tori}
\begin{lemma}[Biswas and G\'omez \cite{Biswas/Gomez:2008} p. 41 theorem 4.1]\label{lemma:BG}
Take a holomorphic principal bundle \(G \to P \to T\) on a complex torus \(T\), with structure group a linear algebraic group \(G\).
Suppose that the identity component \(G_0 \subset G\) has semisimple Levi quotient \(1 \to G_s \to G_0 \to G_{ss} \to 1\).
The following are equivalent:
\begin{enumerate}
\item
the bundle admits a holomorphic connection
\item
the bundle is pseudostable with vanishing first and second Chern classes
\item
after perhaps lifting to a finite unramified covering torus, the pullback bundle admits a holomorphic reduction of structure group to the Borel subgroup of \(G_{ss}\), so that the line bundle associated to any character of the Borel subgroup has vanishing first Chern class
\item
\begin{itemize}
\item
the bundle is pseudostable and
\item
the bundle admits a unique holomorphic flat connection and
\item
every holomorphic section of any associated vector bundle is parallel for that holomorphic flat connection and
\item
every holomorphic connection on the bundle is parallel in that holomorphic flat connection.
\end{itemize}
\end{enumerate}
\end{lemma}
\begin{proof}
Take a holomorphic principal bundle \(G \to P \to T\) on a complex torus \(T\). Take any finite unramified covering space \(\hat{T} \to T\) and pull back the holomorphic vector fields on \(T\) to see that the unramified covering space is also a torus.
If we can prove our result for \(\hat{T}\), then by the uniqueness of the flat connection, it is invariant under the deck transformations of \(\hat{T} \to T\), so descends.
Therefore it suffices to prove the result for some, hence any, finite unramified covering space.
If \(G\) has a normal subgroup \(G_0\subseteq G\) of finite index, then \(\hat{T}:=P/G_0\to T\) is a finite unramified covering space.
So we can assume that \(G=G_0\) is a connected linear algebraic group.
Every connected linear algebraic group has a Levi quotient.

Any holomorphic principal bundle \(G\to P\to T\) with connected complex linear algebraic structure group on any complex torus \(T=\C[n]/\Lambda\) admits a holomorphic connection just when it admits a unique holomorphic flat connection, and the bundle is then pseudostable and the flat connection is categorical: if we apply a morphism \(G \to G'\) of connected linear algebraic groups, then the flat connection is taken to the flat connection \cite{Biswas/Gomez:2008} p. 41 theorem 4.1.
So we can assume that \(G=G_{ss}\).
Any morphism from a trivial line bundle to an associated vector bundle is just a section, so by the categorical nature of the connection, every section is parallel.
See \cite{Biswas.Dumitrescu:2017} p. 5 section 3.1 and p. 6 lemma 3.1 for more detail.
\end{proof}
\begin{lemma}[Biswas \& McKay \cite{Biswas.McKay:2024} lemma 14.3, corollary 14.4]\label{lemma:tori}
Suppose that \(M\) is a smooth projective variety bearing a minimal holomorphic Cartan geometry.
Take a smooth projective variety \(Z\) with trivial canonical bundle.
If there is a finite meromorphic map \rationalMap{Z}{M} then \(Z\) has a finite unramified covering by an abelian variety, and the map extends to a holomorphic map \(Z\to M\).
\end{lemma}
\subsection{The Iitaka fibration has torus fibers}
\begin{theorem}
Suppose that \(M\) is a smooth complex projective variety with a minimal holomorphic parabolic geometry, modelled on a flag variety \((X,G)\).
Then the canonical bundle is numerically effective, i.e.
\[
\int_Z c_1(K_M)^{\dim Z}\ge 0
\]
for every subvariety \(Z\) of \(M\).
If the canonical bundle of \(M\) is big, and \(M\) is not a counterexample to the abundance conjecture, then the canonical bundle of \(M\) is ample and the model \((X,G)\) is cominuscule.
\end{theorem}
\begin{proof}
By the cone theorem \cite{Cascini:2013,Mori:1982} and \cite{Kollar/Mori:1998} p. 22, if \(M\) is a smooth projective variety, in the absence of rational curves, the canonical bundle is numerically effective, i.e.~\(\int_C c_1(K_M)\ge 0\) for any irreducible curve \(C\) in \(M\).
By Kleiman's criterion \cite{Lazarsfeld:2004} p.~53 theorem~1.4.9, \(\int_Z c_1(K_M)^{\dim Z}\ge 0\) for any subvariety \(Z\) of \(M\).

Suppose that the canonical bundle of \(M\) is big, i.e.~\(c_1(\kappa_M)^n>0\) where \(n:=\dim M\). 
So there is a integer \(p>0\) so that, mapping each point of \(m\) to the hyperplane of sections of \(\kappa^{\otimes p}\) vanishing at \(m\), inside the projective space of sections, gives a rational map which is an isomorphism away from some subvariety.
If \(M\) is not a counterexample to the abundance conjecture then the canonical bundle of \(M\) is also semiample, i.e.~some power is spanned by global sections.
So this rational map is a birational morphism.
Exceptional fibers are covered by rational curves \cite{Kawamata:1991} Theorem 2.
But there are no rational curves in \(M\) so the Iitaka fibers are points: the canonical bundle is ample.
Apply theorem~\vref{theorem:Bracket.closed}.
\end{proof}
\begin{theorem}\label{thm:Iitaka.fibers}
Take a nonuniruled smooth projective variety \(M\) with a holomorphic parabolic geometry \(P\to\G_P\to M\) with model \((X,G)\).
Assume that the canonical bundle of \(M\) is semiample (i.e.~ that some positive power of the canonical bundle of \(M\) is spanned by global sections), which would follow from the abundance conjecture.
Then the generic Iitaka fiber of \(M\) admits an unramified covering by an abelian variety.
The restriction to the generic fiber of the induced holomorphic principal \(G\)-bundle \(\G_G:=\amal{\G_P}{P}{G}\) admits a flat holomorphic connection.
\end{theorem}
\begin{proof}
Since some positive power \(p\cb{M}\) of the canonical bundle of \(M\) is spanned by global sections, \(M\) admits a regular Iitaka fibration \(M\to S\) and some positive power of \(\cb{M}\), which we can assume is also \(p\cb{M}\), is pulled back from a line bundle \(L\to S\) \cite[p.~129, Theorem 2.1.26]{Lazarsfeld:2004}.
Take a generic fiber \(M_s\).
So \(p\left.\cb{M}\right|_{M_s}\) is trivial.
By the Iitaka fibration theorem \cite{Fujino:2020} p. 20, \(M_s\) has Kodaira dimension zero.
By Mori's cone theorem \cite{Cascini:2013,Mori:1982} and \cite{Kollar/Mori:1998} p. 22, each fiber \(M_s\) has numerically effective canonical bundle.

Take a nonempty open set \(U_S\subseteq S\) of the smooth points of the base \(S\).
By Sard's lemma, we can pick \(U_S\) that \(M\to S\) is a submersion over \(U_S\).
We can further refine the choice of \(U_S\) so that both \(L\) and \(\cb{S}\) admit a holomorphic connection over \(U_S\).
Hence all powers of \(L\) have a holomorphic connection on \(U_S\).
On \(U:=\pi^{-1}U_S\subseteq M\), \(p\cb{U}\) admits a holomorphic connection, so \(\cb{U}\) admits a holomorphic connection.

By lemma~\vref{lemma:canon.conn}, that holomorphic connection gives us a holomorphic connection on \(TU\), and on \(\left.\G'\right|_U\).
Hence we have a holomorphic connection on \(\left.\cb{M}\right|_U\) and on the pullback bundle \(\pi^*\left.\cb{S}\right|_U\), hence on the relative canonical bundle of \(M\to S\) over \(U_S\), hence on the canonical bundle of the generic fiber \(M_s\).
Existence of a holomorphic connection on a line bundle is equivalent to vanishing of the first Chern class.
So \(c_1(TM_s)=0\) on the generic Iitaka fiber \(M_s\).
By lemma~\vref{lemma:tori}, the generic fiber admits an unramified covering by an abelian variety.
Apply lemma~\vref{lemma:BG} to get a flat connection on \(\G\) over each fiber.
\end{proof}

We prove theorem~\vref{thm:Iitaka.group.scheme}.
\begin{proof}
The argument is similar to that of Jahnke \& Radloff \cite{jahnke2009holomorphicnormalprojectiveconnections} p.~6, proposition~2.6, and to an argument which we quote intact from our previous paper \cite{Biswas.McKay:2024} p.~29.

Since the canonical bundle of \(M\) is spanned by global sections, \(M\) admits a regular Iitaka fibration \(M\to S\) and some tensor power \(p\cb{M}\) with \(p\ge 1\) is pulled back from a line bundle \(L\to S\) \cite[p.~129, Theorem 2.1.26]{Lazarsfeld:2004}.
Hence that power is trivial on the fibers.

By theorem~\vref{thm:Iitaka.fibers}, the generic fiber is an abelian variety, and the Cartan geometry bundle admits a holomorphic connection on that fiber.
All fibers of \(M \to S\) have the same dimension \cite[p.~610, Theorem 2]{Kawamata:1991}.
We will need to check that \(S\) is smooth.

Take any resolution of singularities \(S_0\to S\) \cite{Kollar:2007}.
The rational map \(\rationalMap{M}{S_0}\) becomes regular on some variety \(M_0\) birational to \(M\) \cite{Kollar:2007} p.~135 Corollary~3.18, and we can replace \(M_0\) by a resolution of singularities to arrange that \(M_0\) is smooth, to get a dominant morphism \(M_0\to S_0\) birational to \(M \to S\) with smooth \(M_0\) and \(S_0\).
The morphism \(M_0\to S_0\) has generic fiber connected, so connected fibers.

Since the generic fiber of the Iitaka fibration is an abelian variety with large fundamental group, after replacing \(M\) with a finite unramified covering space, there is some resolution of singularities \(S_0 \to 
S\) over which \(M \to S\) is birational to an abelian group scheme \cite[p.~198, Theorem 6.3]{Kollar:1993}:
\[
\begin{tikzcd}
\dot{M} \arrow[->]{d} \arrow[<->, dashed,densely dashed]{r} & M_0 \arrow[->]{r}\arrow[->]{d} & M \arrow[->]{d}\\
\dot{S} \arrow[<->, dashed,densely dashed]{r} & S_0 \arrow[->]{r} & S
\end{tikzcd}
\]
By lemma~\vref{lemma:dev.large}, rational maps to \(M\) extend to holomorphic maps, since \(M\) contains no rational curves.
Extend the rational map \(\rationalMap{\dot{M}}{M}\) to a holomorphic map \(\dot{M} \to M\).
Up to birational isomorphism, there is a section \(\dot{S}\to\dot{M}\) \cite[Theorem 6.3]{Kollar:1993}, and so a rational section \(\rationalMap{S}{M}\).
Extend the rational section to a holomorphic map \(S\to M\), and hence a holomorphic section.
This section maps curves in \(S\) to isomorphic curves in \(M\); but \(M\) contains no rational curves, and so \(S\) contains no rational curves.
Hence meromorphic maps to \(S\) extend to be holomorphic.
Holomorphically extend \(\rationalMap{\dot{S}}{S}\) and \(\rationalMap{M}{S}\):
\[
\begin{tikzcd}
\dot{M} \arrow[->,shift left]{r}\arrow[->]{d} & M \arrow[->]{d} \arrow[->, dashed,densely dashed,shift left]{l} \\
\dot{S} \arrow[->,shift left]{r} & S\arrow[->, dashed,densely dashed,shift left]{l}
\end{tikzcd}
\]

By Zariski's main theorem (see \cite[p.~280, Corollary 11.4]{Hartshorne:1977}), the map \(\dot{M} \to M\) has connected fibers, as does \(\dot{S} \to S\).
Hence \(\dot{M} \to \dot{S}\) is also an Iitaka fibration.
So the generic fiber has torsion canonical bundle.
The generic fiber \(\dot{M}_{\dot{s}}\) maps holomorphically and birationally to a fiber \(M_s\) of \(M\to S\).
Since \(M_s\) is an abelian variety, there is a global nowhere vanishing holomorphic section of the canonical bundle on \(M_s\), which pulls back to a holomorphic section of the canonical bundle on \(\dot{M}_{\dot{s}}\).
The vanishing locus of this section is a divisor representing the first Chern class of the canonical bundle.
Since the canonical bundle of \(\dot{M}_{\dot{s}}\) is torsion, the section is also globally nonvanishing, and the canonical bundle of \(\dot{M}_{\dot{s}}\) is trivial.
Moroever, \(\dot{M}_{\dot{s}}\to M_s\) is a local biholomorphism, hence a finite covering map.
By lemma~\vref{lemma:tori}, \(\dot{M}_{\dot{s}}\) is also an abelian variety.
Since \(\dot{M}_{\dot{s}}\to M_s\) is birational, this finite covering map is a biholomorphism, hence a biregular morphism of smooth projective varieties.
The generic fiber of \(\dot{M}\to \dot{S}\) is an abelian variety mapping biholomorphically to a fiber of \(M \to S\).

On every fiber \(\dot{M}_{\dot{s}}\), this map \(\dot{M}_{\dot{s}}\to M_s\) is regular and birational with exceptional locus covered in rational curves.
But neither fiber contains rational curves.
Therefore every fiber \(\dot{M}_{\dot{s}}\) of \(\dot{M}\to \dot{S}\) is mapped isomorphically to a fiber of \(M\to S\).

Suitable powers of the canonical bundles of \(\dot{M}\) and \(M\) are spanned by global sections, pulled back from sections of line bundles on \(\dot{S}\) and \(S\) respectively.
All sections on \(M\) and on \(\dot{M}\) are pulled back from \(S\) and \(\dot{S}\) respectively.
So the vanishing locus of any such section is a union of fibers.

The holomorphic map \(\dot{M} \to M\) is an isomorphism except on the ramification divisor, an effective divisor given by the vanishing of the determinant of the derivative of \(\dot{M} \to M\).
Applying the derivative of \(\dot{M} \to M\) to sections of the canonical bundle on \(M\), we get sections of the canonical bundle on \(\dot{M}\), with divisor the sum of the pullback canonical divisor and the ramification divisor.
But the sections vanish on unions of fibers.
Hence the ramification divisor of \(\dot{M}\to M\) is the preimage of some effective divisor on \(\dot{S}\).

Away from the image in \(S\) of the support of that divisor in \(\dot{S}\), our section of \(M \to S\) strikes a smooth fiber in a single smooth point, since \(\dot{M} \to M\) is biholomorphic on those fibers.
All fibers have intersection number \(1\) with the section, and so all fibers are smooth near the image of the section, and the image of the section is smooth.
The variety \(S\) is smooth, since the image of the section is isomorphic to \(S\).

Since the fibers of \(M \to S\) are all of the same dimension, at any point of any fiber at which the reduction of the fiber is smooth, we can pick a transverse complex submanifold, a local holomorphic section of \(M \to S\).
All nearby fibers intersect the local section transversely: the generic point of every fiber of \(M\to S\) is reduced.

If some fiber \(M_s\) of \(M\to S\) is reducible, its preimage in \(\dot{M}\) is reducible.
As the fibers of \(\dot{M}\to M\) are connected, the preimage in \(\dot{M}\) of \(M_s\) is a connected reducible subvariety.
Since the fibers of \(\dot{M} \to \dot{S}\) are irreducible, the preimage of \(M_s\) in \(\dot{M}\) is a union of fibers over a connected reducible subvariety of \(\dot{S}\).
This subvariety of \(\dot{S}\) maps to a point in \(S\).
Because the subvariety is connected, all of the fibers over the subvariety map to the same component of \(M_s\).
But \(\dot{M} \to M\) is surjective, since it is holomorphic and birational.
Therefore all fibers of \(M \to S\) are irreducible.

The fibers \(\dot{M}_{\dot{s}}\) of \(\dot{M} \to \dot{S}\) all have the same images in homology inside \(M\), all nonzero as they are algebraic cycles.
So the derivative of the holomorphic map \(\dot{M}_{\dot{s}} \to M\) drops rank on a ramification divisor.
All holomorphic sections of a suitable positive power of the canonical bundle of \(M\) are pulled back from sections of a suitable line bundle on \(S\).
Find a section of that suitable line bundle, nonzero at \(s\), and pullback to get a section of \(\cb{M}\) nowhere zero near \(M_s\).
Take a local holomorphic section of the canonical bundle of \(S\) defined and nonzero at \(s\), and divide the two sections to get a holomorphic section of the relative canonical bundle of \(M \to S\) defined near \(M_s\).
Plugging the derivative of \(\dot{M}_{\dot{s}} \to M\) into this section gives a section of the canonical bundle of the abelian variety \(\dot{M}_{\dot{s}}\), not vanishing somewhere and therefore not vanishing anywhere as the canonical bundle of an abelian variety is trivial.
Therefore \(\dot{M}_{\dot{s}} \to M\) is a local biholomorphism on every fiber, taking each fiber \(\dot{M}_{\dot{s}}\) to a fiber \(M_s\).
The map \(\dot{M} \to M\) is birational, so the map \(\dot{M}_{\dot{s}} \to M_s\) on each fiber is a biholomorphism to its image.

Take a local holomorphic section of \(M\to S\) transverse to the image of a fiber \(\dot{M}_{\dot{s}}\).
The section strikes the generic nearby fiber of \(M \to S\) in a single smooth point, since \(\dot{M}\to 
M\) is biholomorphic on generic fibers.
Generic fibers have intersection number \(1\) with the section, and so all fibers have intersection \(1\) with the section, so are smooth and reduced near the image of the section.
Hence all fibers of \(M \to S\) are everywhere smooth and reduced and irreducible, and so the maps \(\dot{M}_{\dot{s}} \to M_s\) are all isomorphisms.

Since \(S\) has big canonical bundle, its Iitaka fibration \(S\to\Ii{S}\) is birational, with exceptional fibers covered by rational curves \cite{Kawamata:1991} p.~610 Theorem~2.
But \(S\) contains no rational curves, so there are no exceptional fibers.
Since rational maps to \(S\) extend to become holomorphic, the Iitaka fibration \(S\to\Ii{S}\) has a holomorphic inverse.
The exceptional locus of the canonical morphism is then also covered by rational curves, so is empty, so \(S\) has ample canonical bundle.

The Chern classes of vector bundles are topological, and \(M\to S\) is a \(C^{\infty}\)-trivial torus bundle.
So the Chern classes of the tangent bundle of \(M\) are the pullbacks of the Chern classes of the tangent bundle of \(S\).
Hence the Chern classes \(c_k(S,TS)\) satisfy the same equations as \(c_k(M,TM)\).
\end{proof}

\section{Examples}%
\label{section:examples}
\begin{example}
Consider the \(G=G_2\)-adjoint flag variety \(X=X^5=\dynkin[parabolic=1]G2\), with roots 
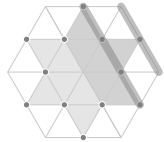
\[
\drawroots{G}{1}
\]
with tally \(1\) and associated cominuscule \(\breve{X}=\Proj{1}\):
\[
\begin{tikzpicture}[baseline=-.5]
\begin{rootSystem}{G}
\roots
\parabolic{1}
\draw[/root system/grading,line width=.3cm] (hex cs:x=1,y=1) -- (hex cs:x=-1,y=-1);
\end{rootSystem}
\end{tikzpicture}
\]
Suppose that \(M=M^5\) is a smooth projective \(5\)-fold bearing a holomorphic \((X,G)\)-geometry.
Up to isomorphism, either 
\begin{itemize}
\item
\(M\) contains a rational curve, in which case \(M=X\) with its standard flat holomorphic \((X,G)\)-geometry \cite{Biswas.McKay:2016} or 
\item
\(M\) is not uniruled, in which case \(M\) contains no rational curve \cite{Biswas.McKay:2016}. The holomorphic parabolic geometry is bracket closed, so the ``contact'' \(4\)-planes are bracket closed.
The Kodaira dimension is:
\begin{enumerate}
\item[\(-\infty\)] 
It would follow from the abundance and minimal model conjectures that \(M^5\) is uniruled, hence \(M=X\), returning to the previous case.
\item[\(0\)]
It would follow from the abundance conjecture that \(M^5\) is a Calabi--Yau manifold.
Up to finite \'etale covering, \(M\) is therefore an abelian variety with a translation invariant \((X,G)\)-geometry by theorem~\vref{thm:cpt.Kaehler.c.1.0}; these are easily classified.
\item[\(1\)] 
The complex manifold \(M\) is a smooth projective variety and has a rational map to a curve with positive canonical bundle.
It would follow from the abundance conjecture that \(M\) (after perhaps replacing by a finite covering space) is an abelian group scheme over a smooth curve of genus at least \(2\), and every fiber is a \(4\)-torus.
The \(5\)-fold \(M\) carries a holomorphic foliation with \(4\)-dimensional leaves and a transverse foliation with \(1\)-dimensional leaves.
The universal covering space \(\tilde{M}\) splits into a product of complex manifolds \(\tilde{M}=Y^4\times Y^1\) where \(Y^1\) is the unit disk or \(\C{}\) and the \(4\)-dimensional leaves are the images of \(Y^4\times *\) \cite{Brunella/Pereira/Touzet:2006}.
\end{enumerate}
\end{itemize}
\end{example}
\begin{example}
Consider the \(G=G_2\)-flag variety \(X=X^5=\dynkin[parabolic=2]G2\), famous from Cartan's \emph{five variable paper} \cite{Cartan:30}, with roots
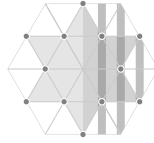
\[
\drawroots{G}{2}
\]
tally \(1\), and associated cominuscule \(\breve{X}=\Proj{2}\):
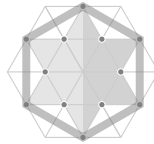
\[
\begin{tikzpicture}[baseline=-.5]
\begin{rootSystem}{G}
\roots
\parabolic{2}
\draw[/root system/grading] (hex cs:x=-1,y=2) -- (hex cs:x=1,y=1) -- (hex cs:x=2,y=-1) -- (hex cs:x=1,y=-2) -- (hex cs:x=-1,y=-1) -- (hex cs:x=-2,y=1) -- cycle;
\end{rootSystem}
\end{tikzpicture}
\]
Suppose that \(M=M^5\) is a smooth projective \(5\)-fold bearing a holomorphic \((X,G)\)-geometry.
Up to isomorphism, either 
\begin{itemize}
\item
\(M\) contains a rational curve, in which case \(M=X\) with its standard flat holomorphic \((X,G)\)-geometry \cite{Biswas.McKay:2016} or 
\item
\(M\) is not uniruled, in which case \(M\) contains no rational curve \cite{Biswas.McKay:2016}.
The Kodaira dimension is:
\begin{enumerate}
\item[\(-\infty\)] 
It would follow from the abundance and minimal model conjectures that \(M^5\) is uniruled, hence \(M=X\), returning to the previous case.
\item[\(0\)]
It would follow from the abundance conjecture that \(M^5\) is a Calabi--Yau manifold.
Suppose so.
Up to finite \'etale covering, \(M\) is an abelian variety with a translation invariant \((X,G)\)-geometry by theorem~\vref{thm:cpt.Kaehler.c.1.0}; these are easily classified.
\item[\(1\)]
The \(5\)-fold \(M\) carries transverse holomorphic foliations with \(2\)-dimensional, \(1\)-dimensional, and \(2\)-dimensional leaves.
The universal covering space \(\tilde{M}\) splits into a product of complex manifolds \(\tilde{M}=Y^4\times Y^1\) where \(Y^1\) is the unit disk and the \(4\)-dimensional leaves of the sum of the two rank \(2\) foliations are the images of \(Y^4\times *\) \cite{Brunella/Pereira/Touzet:2006}.
The complex manifold \(M\) is a smooth projective variety with a canonical Iitaka map to a curve of genus \(2\) or more.
The Chern classes satisfy \(0=c_1^2=c_2=c_3=c_4=c_5\).
It would follow from the abundance conjecture that the typical fiber is a \(4\)-torus.
\end{enumerate}
\end{itemize}
\end{example}
\begin{example}
The \(F_4\)-flag variety \dynkin F{***x} has Hasse diagram:
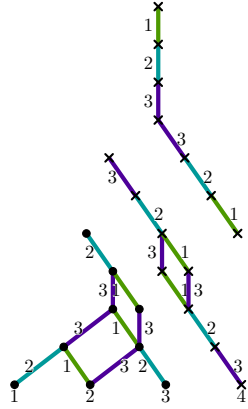
\begin{center}
\hasse[three D=false] F{***x}
\end{center}
This flag variety was the first that Cartan wrote about, and first written about by Cartan \cite{Cartan1893,Ishikawa.Machida:2025,Nurowski:2025} and \cite{Yamaguchi:1993} p.~482.
Counting the crosses, the model is \(15\)-dimensional, and has an \(8\)-dimensional subbundle of the tangent bundle (the component with \(8\) crosses).
Its associated cominuscule subvariety is the \(7\)-dimensional smooth quadric hypersurface \dynkin B{x***} with Hasse diagram
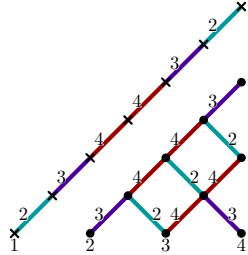
\begin{center}
\hasse B{x***}
\end{center}
Suppose that \(M\) is a smooth projective variety and has a minimal geometry with this model.
Counting the crosses, \(\dimC{M}=15\). 
Assuming the abundance conjecture, the manifold \(M\) is an abelian group scheme over a basis \(M\to\bar M\) with \(\dimC{\bar M}\le 7\).
Since the Chern classes are topological, if \(\dimC{\bar{M}}=7\) then \(\bar M\) is ample and satisfies the Chern class identities of the smooth \(7\)-dimensional quadric hypersurface in projective space.

Yamaguchi states (without proof) \cite{Yamaguchi:1993} p.~483 that if a complex manifold \(M\) of dimension \(15\) bears a rank \(8\) holomorphic subbundle of the tangent bundle, with a certain infinitesimal nondegeneracy condition (too complicated to make explicit here), then it arises from a unique regular normal geometry with this model, and is locally isomorphic to the model.
Our work above strengthens his result: if \(M\) is, in addition, a smooth projective variety, then \(M=X\) with its standard rank \(8\) subbundle.
If a smooth projective variety \(M\) of dimension \(15\) bears a geometry with this model, and either (i) it contains a rational curve or (ii) its associated rank \(8\) subbundle is not bracket closed, then \(M=X\) is the model.
Otherwise, up to finite \'etale covering, \(M\) is an abelian variety with translation invariant geometry (nowhere flat) or \(M\) is an abelian group scheme, with base of dimension at most \(7\) and nowhere flat geometry.
\end{example}
\begin{example}
We return to the example of the \(E_8\)-flag variety \dynkin E{******x*} from page~\vpageref{ex:E.8}.
Suppose that \(M\) is a smooth projective variety and has a minimal geometry with this model.
Counting the crosses, \(\dimC{M}=98\). 
Assuming the abundance conjecture, up to finite \'etale covering, the manifold \(M\) is an abelian group scheme over a smooth base \(M\to\bar M\) with \(\dimC{\bar M}\le 7\) and \(\bar M\) has ample canonical bundle and satisfies the Chern class identities of \(\Proj{7}\).
Suppose that \(\dimC{\bar M}=7\).
By the Bogomolov--Miyaoka--Yau inequality \cite{Greb.Kebekus.Taji:2018}, \(\bar{M}\) is a ball quotient.
So \(M\) is an abelian group scheme over a compact smooth quotient of the \(7\)-dimensional ball.
The fundamental group of \(\bar{M}^7\) is acting as automorphisms of the ball together with morphisms of that fundamental group to \(\SL{2}\times\SL{7}\), say thought of as left and right \(\SL{2,L}\times\SL{7,R}\), acting via irreducible representations on \(\C^7_R\), \(\C^2_L\otimes\C^7_R\), \(\C^{35}_R\) and \(\C^2_L\otimes\C^{21}_R\).
\end{example}

\begingroup
\NewDocumentCommand\LastFlag{smmmm}%
{
	\multirow{-#2}{*}%
		{%
			\(#3\)%
		}%
		\Flag{#4}{#5}%
		\IfBooleanF{#1}{\hline}%
}
\NewDocumentCommand\Flag{mm}{&#1&#2\\*}
\begin{longtabl}{@{}R>{\columncolor[gray]{.93}}>{$}r<{$}L@{}}%
\caption{The flag varieties whose associated cominuscule variety is a projective space}%
\label{table:Projective.Space.Associated.Cominiscules}\\
\toprule
&G/P&\breve{G}/\breve{P}\\
\midrule
\endfirsthead
\multicolumn{3}{c}{\dots continued}\\
\toprule
&G/P&\breve{G}/\breve{P}\\
\midrule
\endhead
\midrule
\multicolumn{2}{c}{continued \dots}\\
\endfoot
\bottomrule
\endlastfoot
\Flag{\dynkin {A}{xo.o}}{\dynkin {A}{xo.o}}
\LastFlag{2}{A_r}{\dynkin {A}{o.ox}}{\dynkin {A}{o.ox}}
\Flag{\dynkin{B}{ox}}{\dynkin{A}{x}}
\Flag{\dynkin{B}{oxo}}{\dynkin{A}{x}}
\Flag{\dynkin{B}{oxo.ooo}}{\dynkin{A}{x}}
\Flag{\dynkin{B}{x*x}}{\dynkin{A}{x*}}
\Flag{\dynkin{B}{x**x}}{\dynkin{A}{x**}}
\Flag{%
\begin{dynkinDiagram}{B}{x*.**x}
\end{dynkinDiagram}%
}{%
\begin{dynkinDiagram}{A}{x*.*}
\dynkinBrace[r-1]{1}{3}
\end{dynkinDiagram}%
}%
\Flag{%
\begin{dynkinDiagram}{B}{x*.**x*}
\end{dynkinDiagram}%
}{%
\begin{dynkinDiagram}{A}{x*.*}
\dynkinBrace[r-2]{1}{3}
\end{dynkinDiagram}%
}%
\Flag{%
\begin{dynkinDiagram}{B}{x*.*xo.ooo}
\dynkinBrace[1\le\ell\le r-1]{1}{3}
\end{dynkinDiagram}%
}{%
\begin{dynkinDiagram}{A}{x*.*}
\dynkinBrace[\ell]{1}{3}
\end{dynkinDiagram}%
}%
\Flag{\dynkin{B}{**x}}{\dynkin{A}{**x}}
\Flag{\dynkin{B}{**xo}}{\dynkin{A}{**x}}
\LastFlag{11}{B_r}{\dynkin{B}{**xo.oo}}{\dynkin{A}{**x}}
\Flag{\dynkin C{xo}}{\dynkin A{x}}
\Flag{\dynkin C{xoo}}{\dynkin A{x}}
\Flag{\dynkin C{xo.ooo}}{\dynkin{A}{x}}%
\Flag{\dynkin{D}{oxo.oooo}}{\dynkin{A}{x}}
\Flag{%
\begin{dynkinDiagram}{D}{x*.*xo.ooo}
\dynkinBrace[\ell]{1}{3}
\end{dynkinDiagram}%
}{%
\begin{dynkinDiagram}[parabolic=1]{A}{}
\dynkinBrace[\ell]{1}{4}
\end{dynkinDiagram}%
}%
\Flag{\dynkin D{x*.*xx}}{%
\begin{dynkinDiagram}{A}{x*.*}
\dynkinBrace[r-2]{1}{3}
\end{dynkinDiagram}%
}%
\Flag{%
\begin{dynkinDiagram}{D}{x*.*x*}
\end{dynkinDiagram}%
}{%
\begin{dynkinDiagram}[parabolic=1]{A}{}
\dynkinBrace[r-1]{1}{4}
\end{dynkinDiagram}%
}%
\LastFlag{5}{D_r}{%
\begin{dynkinDiagram}{D}{x*.**x}
\end{dynkinDiagram}%
}{%
\begin{dynkinDiagram}[parabolic=1]{A}{}
\dynkinBrace[r-1]{1}{4}
\end{dynkinDiagram}%
}%
\Flag{\dynkin[ordering=Carter]E{oooxoo}}{\dynkin A{x}}
\Flag{\dynkin[ordering=Carter]E{oox*oo}}{\dynkin A{x*}}
\Flag{\dynkin[ordering=Carter]E{ox**xo}}{\dynkin A{x**}}
\Flag{\dynkin[ordering=Carter]E{x***xo}}{\dynkin A{x***}}
\LastFlag{5}{E_6}{\dynkin[ordering=Carter]E{****xo}}{\dynkin A{x****}}
\Flag{\dynkin[ordering=Carter]E{oooooox}}{\dynkin A{x}}
\Flag{\dynkin[ordering=Carter]E{ooooox*}}{\dynkin A{x*}}
\Flag{\dynkin[ordering=Carter]E{oooxo**}}{\dynkin A{x**}}
\Flag{\dynkin[ordering=Carter]E{oox*x**}}{\dynkin A{x***}}
\Flag{\dynkin[ordering=Carter]E{oox****}}{\dynkin A{x****}}
\Flag{\dynkin[ordering=Carter]E{ox**x**}}{\dynkin A{x****}}
\Flag{\dynkin[ordering=Carter]E{x***x**}}{\dynkin A{x*****}}
\LastFlag{8}{E_7}{\dynkin[ordering=Carter]E{****x**}}{\dynkin A{x******}}
\Flag{\dynkin[ordering=Carter]{E}{xooooooo}}{\dynkin{A}{x}}
\Flag{\dynkin[ordering=Carter]{E}{*xoooooo}}{\dynkin{A}{x*}}
\Flag{\dynkin[ordering=Carter]{E}{**xooooo}}{\dynkin{A}{x**}}
\Flag{\dynkin[ordering=Carter]{E}{***xoooo}}{\dynkin{A}{X***}}
\Flag{\dynkin[ordering=Carter]{E}{****xooo}}{\dynkin{A}{x****}}
\Flag{\dynkin[ordering=Carter]{E}{*****xxo}}{\dynkin{A}{x*****}}
\Flag{\dynkin[ordering=Carter]{E}{*****x*x}}{\dynkin{A}{x******}}
\Flag{\dynkin[ordering=Carter]{E}{******xo}}{\dynkin{A}{x******}}
\LastFlag{9}{E_8}{\dynkin[ordering=Carter]{E}{*****x**}}{\dynkin{A}{x*******}}
\Flag{\dynkin[ordering=Carter]F{xooo}}{\dynkin A{x}}
\Flag{\dynkin[ordering=Carter]F{*xoo}}{\dynkin A{x*}}
\LastFlag{3}{F_4}{\dynkin[ordering=Carter]F{**xo}}{\dynkin A{x**}}
\Flag{\dynkin[parabolic=1]{G}{2}}{\dynkin[parabolic=1]{A}{1}}%
\Flag{\dynkin[parabolic=2]{G}{2}}{\dynkin[parabolic=2]{A}{2}}%
\LastFlag*{3}{G_2}{\dynkin[parabolic=3]{G}{2}}{\dynkin[parabolic=1]{A}{1}}%
\end{longtabl}
\endgroup

\begin{example}
In table~\vref{table:Projective.Space.Associated.Cominiscules}, we denote a flag variety by its Dynkin diagram, but with \(\dynkin A{o}\) to mark a root which can be either \(\dynkin A{x}\) or \(\dynkin A{*}\).
The table lists all effective flag varieties whose associated cominuscule flag variety is projective space, say \(\breve X=\Proj{d}\).
(Among the various flag varieties of a given rank, most of them fall into this category.)
Suppose that \(M\) is a smooth projective variety and has a minimal geometry with this model.
Assuming the abundance conjecture, up to finite \'etale covering, the manifold \(M\) is an abelian group scheme over a smooth base \(M\to\bar M\), and \(\bar{M}\) has dimension at most that of the projective space \(\breve X=\Proj{d}\) and \(\bar M\) has ample canonical bundle and no rational curves and satisfies the Chern class identities of \(\Proj{d}\).
Suppose that \(\dimC{\bar M}=\dimC{\breve X}\). 
By the Bogomolov--Miyaoka--Yau inequality \cite{Greb.Kebekus.Taji:2018}, \(\bar{M}\) is a ball quotient.
So \(M\) is an abelian group scheme over a compact smooth quotient of the \(d\)-dimensional ball.
Note that there are choices of flag variety model \((X,G)\) for which the associated cominuscule \((\breve X,\breve G)\) has \(\breve X=\Proj{1}\), in which case \(M\) could an abelian group scheme over a curve \(\bar M\), and we don't even yet know if there is one such \(M\) bearing a holomorphic \((X,G)\)-geometry.
\end{example}

\section{Conjecture on Fujiki maniolds}
A \emph{Fujiki space} is a compact connected complex space whose reduction is the meromorphic (or equivalently, holomorphic) image of a compact K\"ahler manifold \cite{Grauert:1994} p. 336 section 3.4, \cite{Abramovich/Karu/Matsuki/Wlodarczyk:2002} \S{}1.2.4.

\begin{conjecture} 
If a Fujiki manifold admits a holomorphic Cartan geometry, then it admits a flat holomorphic Cartan geometry with the same model, and the Fujiki manifold is K\"ahler.
\end{conjecture}

The non-Fujiki spaces \(G/\Gamma\), where \(G\) is a complex semisimple Lie group and \(\Gamma \subset G\) is a cocompact lattice, admit holomorphic affine connections, but not flat ones \cite{Dumitrescu:2009}.

\section{Conclusion}
Our main result is that any smooth projective variety admitting a holomorphic parabolic geometry, modulo dropping, is a torus with translation invariant geometry (classified!), an ample smooth projective variety with cominuscule geometry (classified!), or (after perhaps replacing by a finite covering space) an abelian group scheme over an ample smooth projective variety.
This last possibility remains mysterious.
Worse: we only know how to construct examples of holomorphic parabolic geometries on certain very special abelian group schemes, and the only holomorphic parabolic geometries we can construct on any abelian group scheme at the moment are projective connections \cite{Jahnke/Radloff:2015}.
The base of any abelian group scheme is mapped to the moduli space of abelian varieties by taking the fiber as variety.
So our main result builds a bridge between Cartan geometries and moduli spaces of abelian varieties.
\bibliographystyle{amsplain}
\bibliography{projective-parabolic}

\providecommand{\bysame}{\leavevmode\hbox to3em{\hrulefill}\thinspace}
\providecommand{\MR}{\relax\ifhmode\unskip\space\fi MR }
\providecommand{\MRhref}[2]{%
  \href{http://www.ams.org/mathscinet-getitem?mr=#1}{#2}
}
\providecommand{\href}[2]{#2}
\begin{thebibliography}{100}

\bibitem{Abramovich/Karu/Matsuki/Wlodarczyk:2002}
Dan Abramovich, Kalle Karu, Kenji Matsuki, and Jaros{\l}aw W{\l}odarczyk,
  \emph{Torification and factorization of birational maps}, J. Amer. Math. Soc.
  \textbf{15} (2002), no.~3, 531--572 (electronic). \MR{1896232}

\bibitem{Atiyah:1957}
M.~F. Atiyah, \emph{Complex analytic connections in fibre bundles}, Trans.
  Amer. Math. Soc. \textbf{85} (1957), 181--207. \MR{MR0086359 (19,172c)}

\bibitem{Barlet.Magnusson:2019}
Daniel Barlet and J\'on Magn\'usson, \emph{Complex analytic cycles. {I}---basic
  results on complex geometry and foundations for the study of cycles},
  Grundlehren der mathematischen Wissenschaften [Fundamental Principles of
  Mathematical Sciences], vol. 356, Springer, Cham; Soci\'et\'e{}
  Math\'ematique de France, Paris, [2019] \copyright 2019, Translated from the
  2014 French original [3307241] by Alan Huckleberry. \MR{4294879}

\bibitem{Baston/Eastwood:1989}
Robert~J. Baston and Michael~G. Eastwood, \emph{The {P}enrose transform},
  Oxford Mathematical Monographs, The Clarendon Press Oxford University Press,
  New York, 1989, Its interaction with representation theory, Oxford Science
  Publications. \MR{MR1038279 (92j:32112)}

\bibitem{Baum/Bott:1972}
Paul Baum and Raoul Bott, \emph{Singularities of holomorphic foliations}, J.
  Differential Geometry \textbf{7} (1972), 279--342. \MR{0377923 (51 \#14092)}

\bibitem{Beauville:2000}
Arnaud Beauville, \emph{Complex manifolds with split tangent bundle}, Complex
  analysis and algebraic geometry, de Gruyter, Berlin, 2000, pp.~61--70.
  \MR{1760872}

\bibitem{Besse:1987}
Arthur~L. Besse, \emph{Einstein manifolds}, Ergebnisse der Mathematik und ihrer
  Grenzgebiete (3) [Results in Mathematics and Related Areas (3)], vol.~10,
  Springer-Verlag, Berlin, 1987. \MR{88f:53087}

\bibitem{Biswas.Dumitrescu:2017}
Indranil Biswas and Sorin Dumitrescu, \emph{{Holomorphic Cartan geometries on
  complex tori}}, ArXiv e-prints (2017), 1--7.

\bibitem{Biswas/Dumitrescu:2023}
\bysame, \emph{Holomorphic projective connections on compact complex
  threefolds}, Math. Z. \textbf{304} (2023), no.~2, Paper No. 27, 32.
  \MR{4588174}

\bibitem{Biswas/Gomez:2008}
Indranil Biswas and Tom{\'a}s~L. G{\'o}mez, \emph{Connections and {H}iggs
  fields on a principal bundle}, Ann. Global Anal. Geom. \textbf{33} (2008),
  no.~1, 19--46. \MR{2369185}

\bibitem{Biswas.McKay:2010}
Indranil Biswas and Benjamin McKay, \emph{Holomorphic {C}artan geometries,
  {C}alabi-{Y}au manifolds and rational curves}, Differential Geom. Appl.
  \textbf{28} (2010), no.~1, 102--106. \MR{2579385}

\bibitem{Biswas.McKay:2016}
\bysame, \emph{Holomorphic {C}artan geometries and rational curves}, Complex
  Manifolds \textbf{3} (2016), 145--168, arXiv:1005.1472. \MR{3489612}

\bibitem{Biswas.McKay:2024}
\bysame, \emph{Locally homogeneous holomorphic geometric structures on
  projective varieties}, SIGMA Symmetry Integrability Geom. Methods Appl.
  \textbf{20} (2024), Paper No. 030, 33. \MR{4728452}

\bibitem{Borel:1991}
Armand Borel, \emph{Linear algebraic groups}, second ed., Graduate Texts in
  Mathematics, vol. 126, Springer-Verlag, New York, 1991. \MR{92d:20001}

\bibitem{Boucksom/Demailly/Paun/Peternell:2013}
S{\'e}bastien Boucksom, Jean-Pierre Demailly, Mihai P{\u{a}}un, and Thomas
  Peternell, \emph{The pseudo-effective cone of a compact {K}\"ahler manifold
  and varieties of negative {K}odaira dimension}, J. Algebraic Geom.
  \textbf{22} (2013), no.~2, 201--248. \MR{3019449}

\bibitem{Brieskorn.Knoerrer:1986}
Egbert Brieskorn and Horst Kn\"orrer, \emph{Plane algebraic curves}, Modern
  Birkh\"auser Classics, Birkh\"auser/Springer Basel AG, Basel, 1986,
  Translated from the German original by John Stillwell, [2012] reprint of the
  1986 edition. \MR{2975988}

\bibitem{Brunella/Pereira/Touzet:2006}
Marco Brunella, Jorge~Vit\'orio Pereira, and Fr\'ed\'eric Touzet,
  \emph{K\"ahler manifolds with split tangent bundle}, Bull. Soc. Math. France
  \textbf{134} (2006), no.~2, 241--252. \MR{2233706}

\bibitem{BW:2022}
Jarosław Buczy\'nski and Jarosław~A. Wi\'sniewski, \emph{Algebraic torus
  actions on contact manifolds}, J. Differential Geom. \textbf{121} (2022),
  no.~2, 227--289, With an appendix by Andrzej Weber. \MR{4466670}

\bibitem{Campana:1992}
F.~Campana, \emph{Connexit\'e rationnelle des vari\'et\'es de {F}ano}, Ann.
  Sci. \'Ecole Norm. Sup. (4) \textbf{25} (1992), no.~5, 539--545. \MR{1191735}

\bibitem{Campana/Peternell:2011}
Fr{\'e}d{\'e}ric Campana and Thomas Peternell, \emph{Geometric stability of the
  cotangent bundle and the universal cover of a projective manifold}, Bull.
  Soc. Math. France \textbf{139} (2011), no.~1, 41--74, With an appendix by
  Matei Toma. \MR{2815027 (2012e:14031)}

\bibitem{Cap/Slovak:2009}
Andreas {\v{C}}ap and Jan Slov{\'a}k, \emph{Parabolic geometries. {I}},
  Mathematical Surveys and Monographs, vol. 154, American Mathematical Society,
  Providence, RI, 2009, Background and general theory. \MR{2532439
  (2010j:53037)}

\bibitem{Carlson.Toledo:2014}
James~A. Carlson and Domingo Toledo, \emph{Compact quotients of non-classical
  domains are not {K}\"ahler}, Hodge theory, complex geometry, and
  representation theory, Contemp. Math., vol. 608, Amer. Math. Soc.,
  Providence, RI, 2014, pp.~51--57. \MR{3205511}

\bibitem{Cartan1893}
\'E Cartan, \emph{\"uber die einfachen transformationsgruppen}, Leipz. Ber.
  \textbf{45} (1893), 395--420.

\bibitem{Cartan1904}
\'Elie Cartan, \emph{Sur la structure des groupes infinis de transformation},
  Ann. Sci. \'Ecole Norm. Sup. (3) \textbf{21} (1904), 153--206. \MR{1509040}

\bibitem{Cartan1905}
\bysame, \emph{Sur la structure des groupes infinis de transformation (suite)},
  Ann. Sci. \'Ecole Norm. Sup. (3) \textbf{22} (1905), 219--308. \MR{1509054}

\bibitem{Cartan1909}
Elie Cartan, \emph{Les groupes de transformations continus, infinis, simples},
  Ann. Sci. \'Ecole Norm. Sup. (3) \textbf{26} (1909), 93--161. \MR{1509105}

\bibitem{Cartan:30}
\'Elie Cartan, \emph{Les syst\`emes de {P}faff \`a cinq variables et les
  \'equations aux d\'eriv\'ees partielle du second ordre}, Ann. {\'E}c. Norm.
  \textbf{27} (1910), 109--192, Also in \cite{Cartan:II}, pp. 927--1010.

\bibitem{Cartan:136}
\bysame, \emph{Sur la g\'eometrie pseudo--conforme des hypersurfaces de
  l'espace de deux variables complexes, {I}}, Annali di Matematica Pura ed
  Applicata \textbf{11} (1932), 17--90, Also in \cite{Cartan:II}, pp.
  1232--1305.

\bibitem{Cartan:136bis}
\bysame, \emph{Sur la g\'eometrie pseudo--conforme des hypersurfaces de
  l'espace de deux variables complexes, {II}}, Annali Sc. Norm. Sup. Pisa
  \textbf{1} (1932), 333--354, Also in \cite{Cartan:III2}, pp. 1217--1238.

\bibitem{Cartan:174}
\bysame, \emph{La geometria de las ecuaciones diferencials de tercer orden},
  Rev. Mat. Hispano-Amer. \textbf{4} (1941), 1--31, also in {\OE}uvres
  {C}ompl{\`e}tes, Partie III, Vol. 2, 174, p. 1535--1566.

\bibitem{Cartan:II}
\bysame, \emph{{\OE}uvres compl\`etes. {P}artie {II}}, second ed., \'Editions
  du Centre National de la Recherche Scientifique (CNRS), Paris, 1984,
  Alg\`ebre, syst\`emes diff\'erentiels et probl\`emes d'\'equivalence.
  [Algebra, differential systems and problems of equivalence]. \MR{85g:01032b}

\bibitem{Cartan:III2}
\bysame, \emph{{\OE}uvres compl\`etes. {P}artie {III}. {V}ol. 2}, second ed.,
  \'Editions du Centre National de la Recherche Scientifique (CNRS), Paris,
  1984, G\'eom\'etrie diff\'erentielle. Divers. [Differential geometry.
  Miscellanea], With biographical material by Shiing-shen Chern, Claude
  Chevalley and J. H. C. Whitehead. \MR{85g:01032d}

\bibitem{Cartan:1992}
\bysame, \emph{Le\c cons sur la g\'eom\'etrie projective complexe. {L}a
  th\'eorie des groupes finis et continus et la g\'eom\'etrie diff\'erentielle
  trait\'ees par la m\'ethode du rep\`ere mobile. {L}e\c cons sur la th\'eorie
  des espaces \`a connexion projective}, Les Grands Classiques
  Gauthier-Villars. [Gauthier-Villars Great Classics], \'Editions Jacques
  Gabay, Sceaux, 1992, Reprint of the editions of 1931, 1937 and 1937.
  \MR{1190006 (93i:01030)}

\bibitem{Cascini:2013}
Paolo Cascini, \emph{Rational curves on complex manifolds}, Milan J. Math.
  \textbf{81} (2013), no.~2, 291--315. \MR{3129787}

\bibitem{Cutkosky:2018}
Steven~Dale Cutkosky, \emph{Introduction to algebraic geometry}, Graduate
  Studies in Mathematics, vol. 188, American Mathematical Society, Providence,
  RI, 2018. \MR{3791837}

\bibitem{StGervais:2016}
Henri~Paul de~Saint-Gervais, \emph{Uniformization of {R}iemann surfaces},
  Heritage of European Mathematics, European Mathematical Society (EMS),
  Z\"urich, 2016, Revisiting a hundred-year-old theorem, Translated from the
  2010 French original [MR2768303] by Robert G. Burns, The name of Henri Paul
  de Saint-Gervais covers a group composed of fifteen mathematicians:
  Aur\'elien Alvarez, Christophe Bavard, Fran\c cois B\'eguin, Nicolas
  Bergeron, Maxime Bourrigan, Bertrand Deroin, Sorin Dumitrescu, Charles
  Frances, \'Etienne Ghys, Antonin Guilloux, Frank Loray, Patrick
  Popescu-Pampu, Pierre Py, Bruno S\'evennec and Jean-Claude Sikorav.
  \MR{3494804}

\bibitem{deBarre:2001}
Olivier Debarre, \emph{Higher-dimensional algebraic geometry}, Universitext,
  Springer-Verlag, New York, 2001. \MR{1841091}

\bibitem{Demailly:2002}
Jean-Pierre Demailly, \emph{On the {F}robenius integrability of certain
  holomorphic {$p$}-forms}, Complex geometry ({G}\"ottingen, 2000), Springer,
  Berlin, 2002, pp.~93--98. \MR{1922099 (2003f:32029)}

\bibitem{Demailly:2012}
\bysame, \emph{Analytic methods in algebraic geometry}, Surveys of Modern
  Mathematics, vol.~1, International Press, Somerville, MA; Higher Education
  Press, Beijing, 2012. \MR{2978333}

\bibitem{Demailly:2014}
\bysame, \emph{{On the cohomology of pseudoeffective line bundles}}, ArXiv
  e-prints (2014), 1--39.

\bibitem{Dumitrescu:2001c}
Sorin Dumitrescu, \emph{Structures g\'eom\'etriques holomorphes sur les
  vari\'et\'es complexes compactes}, Ann. Sci. \'Ecole Norm. Sup. (4)
  \textbf{34} (2001), no.~4, 557--571. \MR{MR1852010 (2003e:32034)}

\bibitem{Dumitrescu:2009}
\bysame, \emph{Une caract\'erisation des vari\'et\'es complexes compactes
  parall\'elisables admettant des structures affines}, C. R. Math. Acad. Sci.
  Paris \textbf{347} (2009), no.~19-20, 1183--1187. \MR{2567000}

\bibitem{Dumitrescu:2010}
\bysame, \emph{Connexions affines et projectives sur les surfaces complexes
  compactes}, Math. Z. \textbf{264} (2010), no.~2, 301--316. \MR{2574978}

\bibitem{Dumitrescu/McKay:2016}
Sorin Dumitrescu and Benjamin McKay, \emph{Symmetries of holomorphic geometric
  structures on tori}, Complex Manifolds \textbf{3} (2016), 1--15. \MR{3451271}

\bibitem{Fujiki:1982}
Akira Fujiki, \emph{On the {D}ouady space of a compact complex space in the
  category {${\mathscr C}$}}, Nagoya Math. J. \textbf{85} (1982), 189--211.
  \MR{648422}

\bibitem{Fujino:2020}
Osamu Fujino, \emph{Iitaka conjecture---an introduction}, SpringerBriefs in
  Mathematics, Springer, Singapore, 2020. \MR{4177797}

\bibitem{Fulton/Harris:1991}
William Fulton and Joe Harris, \emph{Representation theory}, Graduate Texts in
  Mathematics, vol. 129, Springer-Verlag, New York, 1991, A first course,
  Readings in Mathematics. \MR{93a:20069}

\bibitem{Grauert:1994}
H.~Grauert, Th. Peternell, and R.~Remmert (eds.), \emph{Several complex
  variables. {VII}}, Encyclopaedia of Mathematical Sciences, vol.~74,
  Springer-Verlag, Berlin, 1994, Sheaf-theoretical methods in complex analysis,
  A reprint of \emph{Current problems in mathematics. Fundamental directions.
  Vol. 74} (Russian), Vseross. Inst. Nauchn. i Tekhn. Inform. (VINITI), Moscow.
  \MR{1326617}

\bibitem{Greb.Kebekus.Taji:2018}
Daniel Greb, Stefan Kebekus, and Behrouz Taji, \emph{Uniformisation of
  higher-dimensional minimal varieties}, Algebraic geometry: {S}alt {L}ake
  {C}ity 2015, Proc. Sympos. Pure Math., vol. 97.1, Amer. Math. Soc.,
  Providence, RI, 2018, pp.~277--308. \MR{3821153}

\bibitem{Griffiths.Robles.Toledo:2014}
Phillip Griffiths, Colleen Robles, and Domingo Toledo, \emph{Quotients of
  non-classical flag domains are not algebraic}, Algebr. Geom. \textbf{1}
  (2014), no.~1, 1--13. \MR{3234111}

\bibitem{Gunning:1978}
R.~C. Gunning, \emph{On uniformization of complex manifolds: the role of
  connections}, Mathematical Notes, vol.~22, Princeton University Press,
  Princeton, NJ, 1978. \MR{505691}

\bibitem{Hartshorne:1977}
Robin Hartshorne, \emph{Algebraic geometry}, Springer-Verlag, New
  York-Heidelberg, 1977, Graduate Texts in Mathematics, No. 52. \MR{0463157}

\bibitem{Helgason:1978}
Sigurdur Helgason, \emph{Differential geometry, {L}ie groups, and symmetric
  spaces}, Graduate Studies in Mathematics, vol.~34, American Mathematical
  Society, Providence, RI, 2001, Corrected reprint of the 1978 original.
  \MR{MR1834454 (2002b:53081)}

\bibitem{Hilgert.Neeb:2012}
Joachim Hilgert and Karl-Hermann Neeb, \emph{Structure and geometry of {L}ie
  groups}, Springer Monographs in Mathematics, Springer, New York, 2012.
  \MR{3025417}

\bibitem{HKLR:1987}
N.~J. Hitchin, A.~Karlhede, U.~Lindstr\"om, and M.~Ro\v{c}ek,
  \emph{Hyper-{K}\"ahler metrics and supersymmetry}, Comm. Math. Phys.
  \textbf{108} (1987), no.~4, 535--589. \MR{877637}

\bibitem{Humphreys:1978}
James~E. Humphreys, \emph{Introduction to {L}ie algebras and representation
  theory}, Graduate Texts in Mathematics, vol.~9, Springer-Verlag, New
  York-Berlin, 1978, Second printing, revised. \MR{499562}

\bibitem{matsumura2025compactkahlermanifoldsnef}
Shin ichi Matsumura, Juanyong Wang, Xiaojun Wu, and Qimin Zhang, \emph{Compact
  {K}\"ahler manifolds with nef anti-canonical bundle}, 2025.

\bibitem{Iitaka:1972}
Shigeru Iitaka, \emph{Genus and classification of algebraic varieties. {I}},
  S\=ugaku \textbf{24} (1972), no.~1, 14--27. \MR{569689}

\bibitem{Iitaka:1982}
\bysame, \emph{Algebraic geometry}, Graduate Texts in Mathematics, vol.~76,
  Springer-Verlag, New York-Berlin, 1982, An introduction to birational
  geometry of algebraic varieties, North-Holland Mathematical Library, 24.
  \MR{637060}

\bibitem{Ishikawa.Machida:2025}
Goo Ishikawa and Yoshinori Machida, \emph{Prolongation of $(8,15)$-distribution
  of type $f_4$ by singular curves}, Symmetry, Integrability and Geometry:
  Methods and Applications (2025).

\bibitem{Jahnke/Radloff:2004}
Priska Jahnke and Ivo Radloff, \emph{Threefolds with holomorphic normal
  projective connections}, Math. Ann. \textbf{329} (2004), no.~3, 379--400.
  \MR{2127983 (2005m:14068)}

\bibitem{jahnke2009holomorphicnormalprojectiveconnections}
Priska Jahnke and Ivo Radloff, \emph{Holomorphic normal projective connections
  on projective manifolds}, 2009.

\bibitem{Jahnke/Radloff:2015}
Priska Jahnke and Ivo Radloff, \emph{Projective uniformization, extremal
  {C}hern classes and quaternionic {S}himura curves}, Math. Ann. \textbf{363}
  (2015), no.~3-4, 753--776. \MR{3412342}

\bibitem{Joyce:2000}
Dominic~D. Joyce, \emph{Compact manifolds with special holonomy}, Oxford
  Mathematical Monographs, Oxford University Press, Oxford, 2000. \MR{1787733}

\bibitem{Kawamata:1991}
Yujiro Kawamata, \emph{On the length of an extremal rational curve}, Invent.
  Math. \textbf{105} (1991), no.~3, 609--611. \MR{1117153}

\bibitem{Klingler:1998}
Bruno Klingler, \emph{Structures affines et projectives sur les surfaces
  complexes}, Ann. Inst. Fourier (Grenoble) \textbf{48} (1998), no.~2,
  441--477. \MR{1625606 (99c:32038)}

\bibitem{Klingler:2001}
\bysame, \emph{Un th\'eor\`eme de rigidit\'e non-m\'etrique pour les
  vari\'et\'es localement sym\'etriques hermitiennes}, Comment. Math. Helv.
  \textbf{76} (2001), no.~2, 200--217. \MR{1839345 (2002k:53076)}

\bibitem{Knapp:2002}
Anthony~W. Knapp, \emph{Lie groups beyond an introduction}, second ed.,
  Progress in Mathematics, vol. 140, Birkh\"auser Boston Inc., Boston, MA,
  2002. \MR{MR1920389 (2003c:22001)}

\bibitem{Kobayashi:1984}
Shoshichi Kobayashi, \emph{Holomorphic projective structures and invariant
  distances}, Several complex variables ({H}angzhou, 1981), Birkh\"auser
  Boston, Boston, MA, 1984, pp.~67--71. \MR{897583}

\bibitem{Kobayashi/Nagano:1964}
Shoshichi Kobayashi and Tadashi Nagano, \emph{A theorem on filtered {L}ie
  algebras and its applications}, Bull. Amer. Math. Soc. \textbf{70} (1964),
  401--403. \MR{0162892}

\bibitem{KobayashiOchiai:1980}
Sh{\^o}shichi Kobayashi and Takushiro Ochiai, \emph{Holomorphic projective
  structures on compact complex surfaces}, Math. Ann. \textbf{249} (1980),
  no.~1, 75--94. \MR{81g:32021}

\bibitem{Kobayashi1981}
\bysame, \emph{Holomorphic projective structures on compact complex surfaces.
  {II}}, Math. Ann. \textbf{255} (1981), no.~4, 519--521. \MR{MR618182
  (83a:32025)}

\bibitem{Kobayashi/Ochiai:1981}
\bysame, \emph{Holomorphic structures modeled after compact {H}ermitian
  symmetric spaces}, Manifolds and Lie groups (Notre Dame, Ind., 1980), Progr.
  Math., vol.~14, Birkh\"auser Boston, Mass., 1981, pp.~207--222. \MR{MR642859
  (84i:53051)}

\bibitem{Kobayashi1982}
\bysame, \emph{Holomorphic structures modeled after hyperquadrics}, T\^ohoku
  Math. J. (2) \textbf{34} (1982), no.~4, 587--629. \MR{MR685426 (84b:32039)}

\bibitem{Kodaira:1960}
K.~Kodaira, \emph{On compact complex analytic surfaces. {I}}, Ann. of Math. (2)
  \textbf{71} (1960), 111--152. \MR{132556}

\bibitem{Kodaira:1963}
\bysame, \emph{On compact analytic surfaces. {II}, {III}}, Ann. of Math. (2)
  \textbf{77} (1963), 563--626; 78 (1963), 1--40. \MR{184257}

\bibitem{Kodaira:1964}
\bysame, \emph{On the structure of compact complex analytic surfaces. {I}},
  Amer. J. Math. \textbf{86} (1964), 751--798. \MR{187255}

\bibitem{Kodaira:1966}
Kunihiko Kodaira, \emph{On the structure of compact complex analytic surfaces
  {II}}, Am. J. Math. \textbf{88} (1966), 682--721.

\bibitem{Kodaira:1968a}
\bysame, \emph{On the structure of compact complex analytic surfaces {III}},
  Am. J. Math. \textbf{78} (1968), 55--83.

\bibitem{Kodaira:1968b}
\bysame, \emph{On the structure of compact complex analytic surfaces {IV}}, Am.
  J. Math. \textbf{90} (1968), 1048--1066.

\bibitem{Kollar:1993}
J{\'a}nos Koll{\'a}r, \emph{Shafarevich maps and plurigenera of algebraic
  varieties}, Invent. Math. \textbf{113} (1993), no.~1, 177--215. \MR{1223229}

\bibitem{Kollar:2007}
J\'anos Koll\'ar, \emph{Lectures on resolution of singularities}, Annals of
  Mathematics Studies, vol. 166, Princeton University Press, Princeton, NJ,
  2007. \MR{2289519}

\bibitem{Kollar/Mori:1998}
J\'{a}nos Koll\'{a}r and Shigefumi Mori, \emph{Birational geometry of algebraic
  varieties}, Cambridge Tracts in Mathematics, vol. 134, Cambridge University
  Press, Cambridge, 1998, With the collaboration of C. H. Clemens and A. Corti,
  Translated from the 1998 Japanese original. \MR{1658959}

\bibitem{Kostant:1961}
Bertram Kostant, \emph{Lie algebra cohomology and the generalized
  {B}orel-{W}eil theorem}, Ann. of Math. (2) \textbf{74} (1961), 329--387.
  \MR{0142696}

\bibitem{Lazarsfeld:2004}
Robert Lazarsfeld, \emph{Positivity in algebraic geometry. {I}}, Ergebnisse der
  Mathematik und ihrer Grenzgebiete. 3. Folge. A Series of Modern Surveys in
  Mathematics [Results in Mathematics and Related Areas. 3rd Series. A Series
  of Modern Surveys in Mathematics], vol.~48, Springer-Verlag, Berlin, 2004,
  Classical setting: line bundles and linear series. \MR{2095471}

\bibitem{Lebrun.Salamon:1994}
Claude LeBrun and Simon Salamon, \emph{Strong rigidity of positive
  quaternion-{K}\"ahler manifolds}, Invent. Math. \textbf{118} (1994), no.~1,
  109--132. \MR{1288469}

\bibitem{Lins.Neto.Scardua:2020}
Alcides Lins~Neto and Bruno Sc\'ardua, \emph{Complex algebraic foliations}, De
  Gruyter Expositions in Mathematics, vol.~67, De Gruyter, Berlin, [2020]
  \copyright 2020. \MR{4201904}

\bibitem{Loray/MarinPerez:2009}
Frank Loray and David Mar{\'{\i}}n~P{\'e}rez, \emph{Projective structures and
  projective bundles over compact {R}iemann surfaces}, Ast\'erisque (2009),
  no.~323, 223--252. \MR{2647972}

\bibitem{McKay:2004}
Benjamin McKay, \emph{Complete complex parabolic geometries}, Int. Math. Res.
  Not. (2006), Art. ID 86937, 34. \MR{2264710 (2007m:53098)}

\bibitem{McKay:2011}
\bysame, \emph{Characteristic forms of complex {C}artan geometries}, Adv. Geom.
  \textbf{11} (2011), no.~1, 139--168. \MR{2770434}

\bibitem{McKay2011b}
\bysame, \emph{Holomorphic {C}artan geometries on uniruled surfaces}, C. R.
  Math. Acad. Sci. Paris \textbf{349} (2011), no.~15-16, 893--896. \MR{2835898}

\bibitem{McKay2013}
\bysame, \emph{The {H}artogs extension problem for holomorphic parabolic and
  reductive geometries}, Monatsh. Math. \textbf{181} (2016), no.~3, 689--713.
  \MR{3552807}

\bibitem{McKay:2016}
\bysame, \emph{Holomorphic geometric structures on {K}\"ahler-{E}instein
  manifolds}, Manuscripta Math. \textbf{153} (2017), no.~1-2, 1--34.
  \MR{3635971}

\bibitem{McKay:2020}
\bysame, \emph{{Cominuscule subvarieties of flag varieties}}, arXiv e-prints
  (2020), arXiv:2002.08193.

\bibitem{McKay:2022b}
\bysame, \emph{Characteristic forms of complex {C}artan geometries {III}:
  {G}-structures}, 2022.

\bibitem{McKay:2022}
\bysame, \emph{Characteristic forms of complex {C}artan geometries {II}}, Adv.
  Math. \textbf{477} (2025), Paper No. 110360, 55. \MR{4911617}

\bibitem{McKay:2023}
\bysame, \emph{An introduction to {C}artan geometries}, 2025.

\bibitem{Milne:2005}
J.~S. Milne, \emph{Introduction to {S}himura varieties}, Harmonic analysis, the
  trace formula, and {S}himura varieties, Clay Math. Proc., vol.~4, Amer. Math.
  Soc., Providence, RI, 2005, pp.~265--378. \MR{2192012}

\bibitem{Milne:2017}
\bysame, \emph{Algebraic groups}, Cambridge Studies in Advanced Mathematics,
  vol. 170, Cambridge University Press, Cambridge, 2017, The theory of group
  schemes of finite type over a field. \MR{3729270}

\bibitem{Mori:1982}
Shigefumi Mori, \emph{Threefolds whose canonical bundles are not numerically
  effective}, Ann. of Math. (2) \textbf{116} (1982), no.~1, 133--176.
  \MR{662120}

\bibitem{Mumford:1999}
David Mumford, \emph{The red book of varieties and schemes}, expanded ed.,
  Lecture Notes in Mathematics, vol. 1358, Springer-Verlag, Berlin, 1999,
  Includes the Michigan lectures (1974) on curves and their Jacobians, With
  contributions by Enrico Arbarello. \MR{1748380}

\bibitem{Mumford:2008}
\bysame, \emph{Abelian varieties}, Tata Institute of Fundamental Research
  Studies in Mathematics, vol.~5, Tata Institute of Fundamental Research,
  Bombay; by Hindustan Book Agency, New Delhi, 2008, With appendices by C. P.
  Ramanujam and Yuri Manin, Corrected reprint of the second (1974) edition.
  \MR{2514037}

\bibitem{Nurowski:2025}
Paweł Nurowski, \emph{Exceptional simple real {L}ie algebras
  {$\mathfrak{f}_4$} and {$\mathfrak{e}_6$} via contactifications}, J. Inst.
  Math. Jussieu \textbf{24} (2025), no.~1, 157--201. \MR{4847117}

\bibitem{ou2025characterizationuniruledcompactkahler}
Wenhao Ou, \emph{A characterization of uniruled compact {K}ähler manifolds},
  2025.

\bibitem{Painleve:2022}
Paul Painlevé, \emph{Leçons sur la théorie analytique des équations
  différentielles}, Legare Street Press, 1897.

\bibitem{Poincare:1985}
Henri Poincar\'e, \emph{Papers on {F}uchsian functions}, Springer-Verlag, New
  York, 1985, Translated from the French and with an introduction by John
  Stillwell. \MR{809181}

\bibitem{Satake:1980}
Ichir\^o Satake, \emph{Algebraic structures of symmetric domains}, Kan\^o{}
  Memorial Lectures, vol.~4, Iwanami Shoten, Tokyo; Princeton University Press,
  Princeton, NJ, 1980. \MR{591460}

\bibitem{Serre:2001}
Jean-Pierre Serre, \emph{Complex semisimple {L}ie algebras}, Springer
  Monographs in Mathematics, Springer-Verlag, Berlin, 2001, Translated from the
  French by G. A. Jones, Reprint of the 1987 edition. \MR{2001h:17001}

\bibitem{Shafarevich2013}
Igor~R. Shafarevich, \emph{Basic algebraic geometry. 1}, third ed., Springer,
  Heidelberg, 2013, Varieties in projective space. \MR{3100243}

\bibitem{Sharpe:2002}
Richard~W. Sharpe, \emph{An introduction to {C}artan geometries}, Proceedings
  of the 21st Winter School ``Geometry and Physics'' (Srn\'\i, 2001), no.~69,
  2002, pp.~61--75. \MR{2004f:53023}

\bibitem{Singer.Sternberg:1965}
I.~M. Singer and Shlomo Sternberg, \emph{The infinite groups of {L}ie and
  {C}artan. {I}. {T}he transitive groups}, J. Analyse Math. \textbf{15} (1965),
  1--114. \MR{217822}

\bibitem{Siu:2009}
Yum-Tong Siu, \emph{Abundance conjecture}, Geometry and analysis. {N}o. 2, Adv.
  Lect. Math. (ALM), vol.~18, Int. Press, Somerville, MA, 2011, pp.~271--317.
  \MR{2882447}

\bibitem{Steenrod:1999}
Norman Steenrod, \emph{The topology of fibre bundles}, Princeton University
  Press, Princeton, NJ, 1999, Reprint of the 1957 edition, Princeton
  Paperbacks. \MR{2000a:55001}

\bibitem{Sternberg:1983}
Shlomo Sternberg, \emph{Lectures on differential geometry}, second ed., Chelsea
  Publishing Co., New York, 1983, With an appendix by Sternberg and Victor W.
  Guillemin. \MR{MR891190 (88f:58001)}

\bibitem{Ueno:1975}
Kenji Ueno, \emph{Classification theory of algebraic varieties and compact
  complex spaces}, Lecture Notes in Mathematics, Vol. 439, Springer-Verlag,
  Berlin-New York, 1975, Notes written in collaboration with P. Cherenack.
  \MR{0506253 (58 \#22062)}

\bibitem{Yamaguchi:1993}
K.~Yamaguchi, \emph{Differential systems associated with simple {L}ie
  algebras}, Advanced Studies in Pure Mathematics \textbf{22} (1993), 413--494.

\bibitem{Yang:1980}
Chen~Ning Yang, \emph{Fibre bundles and the physics of the magnetic monopole},
  The {C}hern {S}ymposium 1979 ({P}roc. {I}nternat. {S}ympos., {B}erkeley,
  {C}alif., 1979), Springer, New York-Berlin, 1980, pp.~247--253. \MR{609564}

\bibitem{Yau:1977}
Shing~Tung Yau, \emph{Calabi's conjecture and some new results in algebraic
  geometry}, Proc. Nat. Acad. Sci. U.S.A. \textbf{74} (1977), no.~5,
  1798--1799. \MR{MR0451180 (56 \#9467)}

\bibitem{Yau:1978}
\bysame, \emph{On the {R}icci curvature of a compact {K}\"ahler manifold and
  the complex {M}onge-{A}mp\`ere equation. {I}}, Comm. Pure Appl. Math.
  \textbf{31} (1978), no.~3, 339--411. \MR{480350}

\end{thebibliography}

\end{document}